\documentclass{amsart}
\usepackage{amssymb, latexsym}
\usepackage{mathtools,color}
\usepackage{enumerate}
\usepackage{footnote}

\newcommand\supp{\operatorname{supp}}

\newcommand\simP{\asymp}

\newcommand{\BA}[1]{{\overrightarrow{#1}}}

\newcommand{\eps}{\varepsilon}

\newcommand{\RR}{\mathbb{R}}
\newcommand{\NN}{\mathbb{N}}
\newcommand{\HHH}{\dot{\mathcal{H}}}
\newcommand{\III}{\mathcal{I}}
\newcommand{\JJJ}{\mathcal{J}}
\newcommand{\TTT}{\mathcal{T}}

\newcommand{\brt}{\breve{t}}
\newcommand{\brw}{\breve{w}}
\newcommand{\brv}{\breve{v}}
\newcommand{\barw}{\bar{w}}
\newcommand{\barR}{\bar{R}}
\newcommand{\tu}{\tilde{u}}

\newcommand{\tw}{\tilde{w}}
\newcommand{\tE}{\widetilde{E}}
\newcommand{\tR}{\widetilde{R}}
\newcommand{\tW}{\widetilde{W}}
\newcommand{\tlambda}{\tilde{\lambda}}
\newcommand{\trho}{\tilde{\rho}}

\newcommand{\vecc}{\overrightarrow}
\newcommand{\vw}{\vec{w}}
\newcommand{\vz}{\vec{z}}
\newcommand{\vW}{\vecc{W}}

\newcommand{\ch}{\check{h}}

\newcommand{\cw}{\check{w}}

\newcommand{\HSC}{\dot{H}^{s_c}\times \dot{H}^{s_c-1}}
\newcommand{\lin}{\texttt{L}}

\theoremstyle{plain}
\newtheorem{thm}{Theorem}
\newtheorem{definition}[thm]{Definition}
\newtheorem{rem}[thm]{Remark}

\newtheorem{prop}[thm]{Proposition}
\newtheorem{lem}[thm]{Lemma}
\newtheorem{res}[thm]{Result}

\numberwithin{equation}{section} \numberwithin{thm}{section}

\begin{document}

\title[Radial solutions of supercritical wave equations]{Blow-up of the critical Sobolev norm for nonscattering radial solutions of supercritical wave equations on $\mathbb{R}^{3}$}

\author{Thomas Duyckaerts$^1$}
\thanks{$^1$LAGA, Universit\'e Paris 13 (UMR 7539). Partially supported by ERC Starting Grant no. 257293 Dispeq, ERC Advanced Grant  no. 291214 BLOWDISOL and ANR Grant SchEq}

\author{Tristan Roy$^2$}
\thanks{$^2$Nagoya University. Partially supported by ERC Advanced Grant  no. 291214, BLOWDISOL}




\vspace{-0.3in}

\begin{abstract}
We consider the wave equation in space dimension $3$, with an energy-supercritical nonlinearity which can be either focusing or defocusing. For any radial solution of
the equation, with positive maximal time of existence $T$, we prove that one of the following holds: (i) the norm of the solution in the
critical Sobolev space
goes to infinity as $t$ goes to $T$, or (ii) $T$ is infinite and the solution scatters to a linear solution forward in time. We use a variant of the channel of energy method, relying on a generalized $L^p$-energy which is almost conserved by the flow of the radial linear wave equation.
\end{abstract}

\maketitle
\tableofcontents

\section{Introduction}
In many recent works, global existence and scattering for solutions of a supercritical dispersive equation were proved assuming that an appropriate critical norm remains bounded close to the maximal time of existence. The goal of this article is to prove, in a specific case, a slightly stronger result, namely that it is sufficient to assume the boundedness of the critical norm only along a sequence of times going to the time of existence. More precisely, we consider the supercritical wave equation in space dimension $3$:
\begin{equation}
\partial_{tt} w  - \triangle w   = \iota |w|^{p-1} w,  \quad (t,x)\in I\times \RR^3,
\label{Eqn:WaveSup}
\end{equation}
with real-valued initial data
\begin{equation}
 \label{Eqn:WaveSupData}
w(0,x):=w_0(x)\in \dot{H}^{s_c}(\mathbb{R}^{3}),\quad \partial_tw(0,x):=w_{1}(x) \in \dot{H}^{s_{c}-1}(\mathbb{R}^{3}),
\end{equation}
where $(t,x)\in I\times \RR^3$, $I$ is an interval such that $0\in I\subset \RR$, $p>5$, $\iota=+1$ (focusing case) or $\iota=-1$ (defocusing case), $s_{c}$ is the critical Sobolev exponent, i.e $s_{c} =\frac{3}{2} - \frac{2}{p-1}$, and $\dot{H}^{s_{c}}(\RR^3)$ is the usual homogeneous Sobolev space.

Equation (\ref{Eqn:WaveSup}) is locally well-posed in $\dot{H}^{s_c}(\mathbb{R}^{3}) \times \dot{H}^{s_c-1}(\mathbb{R}^{3})$: for any initial data $(w_0,w_1)$, there exists a solution
$$\vec{w} :=(w,\partial_tw)\in C^0\big((T_-(w),T_+(w)),\dot{H}^{s_c}(\mathbb{R}^{3}) \times \dot{H}^{s_c-1}(\mathbb{R}^{3})\big)$$
defined on a maximal interval of existence $(T_-(w),T_+(w))$ that satisfies (\ref{Eqn:WaveSup}), (\ref{Eqn:WaveSupData}) in the Duhamel sense, and is unique in a natural class of functions.
It has the following scaling invariance: if $\lambda>0$ and $w$ is a solution, then $w_{\lambda}$, defined by
$$w_{\lambda}(t,x) :=\lambda^{\frac{2}{p-1}} w(\lambda t,\lambda x)$$
is also a solution of (\ref{Eqn:WaveSup}), with maximal interval of existence $\left(\lambda^{-1}T_-(w),\lambda^{-1}T_+(w)\right)$, that satisfies
\begin{equation}
 \label{Eqn:scale_norm}
\left\|\vw_{\lambda}(0)\right\|_{\dot{H}^{s_c}(\mathbb{R}^{3}) \times \dot{H}^{s_c-1}(\mathbb{R}^{3})}=\left\|(w_0,w_1)\right\|_{\dot{H}^{s_c}(\mathbb{R}^{3}) \times \dot{H}^{s_c-1}(\mathbb{R}^{3})}\cdot
\end{equation}
With the additional assumption $(w_0,w_1)\in \dot{H}^1(\mathbb{R}^{3}) \times L^2(\mathbb{R}^{3})$, the energy
$$
E(\vec{w}(t)):=\frac{1}{2}\int_{\mathbb{R}^{3}} |\nabla w(t,x)|^2\,dx+\frac{1}{2}\int_{\mathbb{R}^{3}} |\partial_t w(t,x)|^2\,dx
- \iota \frac{1}{p+1} \int_{\mathbb{R}^{3}} |w(t,x)|^{p+1}\,dx
$$
is well-defined for all $t$ and independent of time. The assumption $p>5$ is equivalent to $s_c>1$: the equation is energy-supercritical, and the energy has little utility in the study of global well-posedness or related properties.

Our goal is to classify the solutions of (\ref{Eqn:WaveSup}) according to their dynamics. The solution $w$ is said to \emph{scatter} forward in time if $T_+(w)=+\infty$ and if there exists a solution $w_{\lin}$ of the linear wave equation
\begin{equation}
 \label{Eqn:WaveLin}
\partial_{tt}w_{\lin}-\triangle w_{\lin}=0
\end{equation}
such that
\begin{equation}
\label{Eqn:scattering}
\lim_{t\to +\infty}\left\|\vec{w}_{\lin}(t)-\vec{w}(t)\right\|_{\dot{H}^{s_c}(\mathbb{R}^{3}) \times \dot{H}^{s_c-1}(\mathbb{R}^{3})}=0.
\end{equation}

It was proved in \cite{KenMerleSuper} in the defocusing case that radial solutions of (\ref{Eqn:WaveSup}) that are bounded in the critical space scatter. More precisely, for any solution of (\ref{Eqn:WaveSup}),  (\ref{Eqn:WaveSupData}) with $\iota=-1$, and radial initial data $(w_0,w_1)$
\begin{equation}
 \label{Eqn:ResultKM}
\limsup_{t\to T_+(w)}\|\vec{w}(t)\|_{\dot{H}^{s_c} (\mathbb{R}^{3}) \times \dot{H}^{s_c-1}(\mathbb{R}^{3})}<\infty
\Longrightarrow w\text{ scatters forward in time.}
\end{equation}
This was later extended to the nonradial, defocusing case in \cite{KillipVisanSuper}, and the radial, focusing case in \cite{DuyKenMerScattSuper} (see also \cite{KenigMerleN5}, \cite{Bulut}, \cite{DodsonLawrie} in higher dimensions). Note that it follows from the scaling invariance of the equation and (\ref{Eqn:scale_norm}) that it is impossible to give a lower a priori bound of $T_+(w)$ only in terms of the $\dot{H}^{s_c}(\mathbb{R}^{3}) \times \dot{H}^{s_c-1} (\mathbb{R}^{3})$-norm of $(w_0,w_1)$: in particular, even the implication
\begin{equation*}
\limsup_{t\to T_+(w)}\|\vec{w}(t)\|_{\dot{H}^{s_c} (\mathbb{R}^{3})\times \dot{H}^{s_c-1} (\mathbb{R}^{3})}<\infty\Longrightarrow T_+(w)=+\infty,
\end{equation*}
weaker than (\ref{Eqn:ResultKM}), does not follow from the local Cauchy theory for equation (\ref{Eqn:WaveSup}).

In this paper, we improve (\ref{Eqn:ResultKM})  in the radial case:
\begin{thm}
Assume $p>5$.
Let $w$ be a solution of (\ref{Eqn:WaveSup}) with radial data $(w_{0},w_{1}) \in \dot{H}^{s_{c}}(\mathbb{R}^{3}) \times
\dot{H}^{s_{c}-1}(\mathbb{R}^{3})$. Then:
\begin{itemize}
\item either
\begin{align}
\lim_{t \rightarrow T_{+}(w)} \left\| ( w(t), \partial_{t} w(t) ) \right\|_{\dot{H}^{s_{c}}(\mathbb{R}^{3}) \times \dot{H}^{s_{c}-1}(\mathbb{R}^{3})}
=+\infty
\label{Eqn:Blowuptype1}
\end{align}
\item or $T_{+}(w) =+ \infty$ and $w$ scatters forward in time.
\end{itemize}
An analogous statement holds for negative times.
\label{Thm:Main}
\end{thm}
Note that Theorem \ref{Thm:Main} is equivalent to (\ref{Eqn:ResultKM}), with the limit superior replaced by a limit inferior. We do not know any direct application of this qualitative improvement, however its analog in the case $p=5$, $\iota=1$ is crucial in the proof of the soliton resolution conjecture for the energy-critical wave equation in \cite{DuyKenMerClass}.

Theorem \ref{Thm:Main} means exactly that solutions of (\ref{Eqn:WaveSup}) are of one of the following three types: scattering solutions, solutions blowing-up in finite time with a critical norm going to infinity at the maximal time of existence, and global solutions with a critical norm going to infinity for infinite times.

In the defocusing case $\iota=-1$, it is conjectured that all solutions of (\ref{Eqn:WaveSup}) with initial data in the critical space $\HSC$ scatter. The difficulty of this conjecture is of course the lack of conservation law at the level of regularity of this critical space. The only supercritical dispersive equations for which scattering was proved for all solutions are wave and Schr\"odinger equations with defocusing barely supercritical nonlinearities: see e.g \cite{RoyLog,RoyLogSchrod,Shih,TaoLog} \footnote{In \cite{RoyLog} the scattering with data in $\tilde{H}^{2}:=\dot{H}^{2}\cap \dot{H}^1 (\mathbb{R}^3)$ is not explicitely mentionned. However, it can be easily derived from the finite bounds of the $L_{t}^{4}(\RR, L_{x}^{12})$ norm and the $L_{t}^{\infty}(\RR, \tilde{H}^2)$ norm of the solution, by using a similar argument to that in \cite{RoyLogSchrod} to prove scattering.}.

In the focusing case $\iota=1$, solutions blowing up in finite time are known. One type of finite time blow-up is given by the ODE $y''=|y|^{p-1}y$, and is believed to be stable: see \cite{DonningerSchorkhuber,DonningerSchorkhuber2} for stability inside the wave cone. For large space dimensions ($\geq 11$) and large $p$, ``geometric" blow-up solutions, based on a stationary solution of (\ref{Eqn:WaveSup}) which is not in $\dot{H}^{s_c}$ are obtained in \cite{Collot}. These solutions  belong to the critical space $\dot{H}^{s_c}\times \dot{H}^{s_c-1}$ and satisfy  (\ref{Eqn:Blowuptype1}),
however all subcritical Sobolev norms remain bounded.

It is not known if there exist global solutions of (\ref{Eqn:WaveSup}) satisfying (\ref{Eqn:Blowuptype1}). This type of solution is excluded in the energy-critical case $p=5$ in \cite{DuyKenMerClass} using conservation of energy and a monotonicity formula, two tools that are not available in the case $p>5$. Let us mention the construction, in the case $p=7$, of a solution $w$ of (\ref{Eqn:WaveSup}), with initial data $(w_0,w_1)$ which does not belong to $\dot{H}^{s_c}\times \dot{H}^{s_c-1}$ but is in all higher order Sobolev spaces $\dot{H}^s\times \dot{H}^{s-1}$, $s>s_c$.

Conditional global well-posedness results similar to (\ref{Eqn:ResultKM}) are also known for other type of equations: radial wave equations with energy-subcritical nonlinearities (see \cite{Shen12P}), defocusing Schr\"odinger equations (see \cite{KenMerleSchrod} for cubic nonlinearity in space dimension $3$ and \cite{KillipVisanNLS} for supercritical nonlinearities in large space dimensions). In \cite{MerleRaphael08}, a lower bound of the critical norm (of the form $|\log(T_+-t)|^{\eps}$) is obtained for finite energy, radial solutions of focusing energy-subcritical nonlinear Schr\"odinger equations blowing-up in finite time $T_+$. This can be seen as a quantitative version of Theorem \ref{Thm:Main}, restricted to finite time blow-up solutions, and in a subcritical context.  Except for this work we do not know any result of the type of Theorem \ref{Thm:Main} for dispersive equations.

Results in the same spirit than (\ref{Eqn:ResultKM}) for the Navier-Stokes equation are known since \cite{IsSeSh03}, where it is proved that an a priori bound of the scale-invariant $L^3$-norm implies global well-posedness and regularity. We refer to \cite{KenigKoch11}, \cite{GaKoPl13} and \cite{GaKoPl14}, for related statements with proofs based on profile decomposition. In \cite{Seregin12}, it is proved that blow-up in finite time implies that the $L^3$-norm goes to infinity, which might be seen as an analog of Theorem \ref{Thm:Main} for Navier-Stokes equation.

\medskip

The proof of Theorem \ref{Thm:Main} relies on the channel of energy method initiated in \cite{DuyKenMerSmall,DuyKenMerProf} to study the radial focusing energy-critical wave equation. In \cite{DuyKenMerClass}, this method was used to get the resolution into solitons for any radial solution of the equation that does not satisfy (\ref{Eqn:Blowuptype1}) with $s_c=1$. It is based on the observation that the exterior energy of any nonzero, finite-energy solution
$w_{\lin}$ of the linear wave equation (\ref{Eqn:WaveLin}) satisfies a lower bound. Namely, for some $\eta,r_0>0$
\begin{equation}
\label{Eqn:ext_channel}
\forall t\geq 0\text{ or }\forall t\leq 0,\quad \int_{|x|\geq r_0+|t|} |\nabla w_{\lin}(t)|^2+|\partial_t w_{\lin}(t)|^2\,dx\geq \eta.
\end{equation}
A key step in the proof is the characterization of all (global) solutions of the nonlinear equation that do not satisfy the preceding dispersive property, i.e. such that
\begin{equation}
\label{Eqn:no_channel}
\forall r_0>0,\quad \liminf_{t\to \pm\infty}\int_{|x|\geq r_0+|t|} |\nabla w(t)|^2+|\partial_t w(t)|^2\,dx=0.
\end{equation}
This rigidity result is then used to determine the profiles in a profile decomposition for a bounded sequence $\{\vec{w}(t_n)\}_n$, $t_n\to T_+(w)$.

The exterior energy appearing in (\ref{Eqn:ext_channel}) is not invariant by the scaling of (\ref{Eqn:WaveSup}) and not well-defined for solutions with initial data in $\dot{H}^{s_c}\times \dot{H}^{s_c-1}$ in the supercritical range $p>5$. It is tempting to replace, in (\ref{Eqn:ext_channel}) $w_{\lin}$ by $D^{s_c-1}w_{\lin}$ to obtain a scale invariant quantity, however this would lose the local character of the norm which is crucial when using finite speed of propagation for equation (\ref{Eqn:WaveSup}).
The main novelty of this article is to replace the standard energy of the linear wave equation by a generalized energy which is local and invariant by the scaling of (\ref{Eqn:WaveSup}). Namely, if $m>2$, and $w_{\lin}$ is a radial solution to the wave equation, we define its $L^m$-generalized energy  by:
$$E_m(t)=E_m[w_{\lin}](t)=\int_0^{\infty}\left|\partial_r (rw_{\lin}(t,r))\right|^m+\left|\partial_r (rw_{\lin}(t,r))\right|^m\,dr.$$
Then $E_m$ is almost conserved by the linear flow: their exists a constant $C>0$, depending only on $m$ such that
$$ \forall t,\quad C^{-1}E_m(0)\leq E_m(t)\leq CE_m(0).$$
Furthermore, for $m=\frac{p-1}{2}$, $E_m$ is well-defined if $(w_0,w_1)\in \dot{H}^{s_c}(\mathbb{R}^{3}) \times \dot{H}^{s_c-1}(\mathbb{R}^{3})$, and is invariant by the scaling of (\ref{Eqn:WaveSup}). Finally, it is possible to prove an exterior energy property similar to (\ref{Eqn:ext_channel}) for the generalized energy.

The proof of Theorem \ref{Thm:Main} follows the lines of the proof of \cite{DuyKenMerClass} (using also some of the arguments of \cite{DuyKenMerScattSuper}, specific to the supercritical wave equation), replacing the usual energy by the generalized energy $E_m$. The main obstruction to the exterior energy property for the nonlinear equation is the existence of a nonzero radial stationary solution. In the energy-critical case, this solution belongs to the space $\dot{H}^1(\RR^3)$, is of order $1/r$ for large $r$, and unique up to scaling and sign change. In the supercritical case, there is no solution in the critical space $\dot{H}^{s_c}(\RR^3)$, but there exist singular stationary solutions of the same order $1/r$ for large $r$ (see \cite{DuyKenMerScattSuper} for the focusing case and Proposition \ref{Prop:SolStat} below for the defocusing case).

The outline of the article is as follows: after some preliminaires on equation (\ref{Eqn:WaveSup})  (Section \ref{Sec:preliminaries}) we define, and give some properties of generalized energies for the radial linear wave equation (Section \ref{Sec:energy}). In Section \ref{Sec:channels}, we prove lower bound for the exterior (generalized) energy of nonzero solutions of (\ref{Eqn:WaveSup}). Section \ref{Sec:global} gives the proof of Theorem \ref{Thm:Main} in the global case. Section \ref{Sec:blowup} concerns the finite time blow-up case.

We start with some notation.

\subsection*{Acknowledgments}
The first author would like to thank Pierre-Gilles Lemari\'e-Rieusset for references on Navier-Stokes.

\subsection*{Notation}

If $a$ and $b$ are two positive quantities we will write $a\lesssim b$ when there exists a constant $C>0$ (which depends only on $p$) such that $a\leq Cb$. When the constant is allowed to depend on another quantity $M$, we will write $a\lesssim_M b$. We will write $a\approx b$ when $a\lesssim b$ and $b\lesssim a$. We will write  $a \ll b$ (resp. $a \gg b$) if there exists a large constant $C>0$ such that $b \geq C a$ (resp. $a \geq C b$). \\
\\
If $f$ is a function depending on $t$ and $r:=|x|$, let
\begin{equation*}
\vec{f}:= (f,\partial_{t} f).
\end{equation*}
Given $f_{1}$,...,$f_{N}$ $N$ functions depending on $t$ and $r$, and $F$ a map, let
\begin{equation*}
F \left( \partial_{r,t} f_{1}, ..., \partial_{r,t}f_{N} \right) :=
F \left( \partial_{r} f_{1},...,\partial_{r}f_{N} \right) +  F \left( \partial_{t}f_{1},...\partial_{t}f_{N}  \right).
\end{equation*}
Given $s \geq 0$ and $n$ a positive integer, we define
\begin{equation*}
\dot{\mathcal{H}}^{s}(\mathbb{R}^{n}) := \dot{H}^{s} (\mathbb{R}^{n}) \times \dot{H}^{s-1}(\mathbb{R}^{n}),
\end{equation*}
where $\dot{H}^{s}$ denotes the standard homogeous Sobolev space. \\
To lighten the notation, if $n=3$, then we will write $\dot{H}^{s}$ instead of $\dot{H}^s(\RR^3)$, and we will proceed similarly for the other spaces that we use in this paper, such as $L^p(\RR^3)$, etc... We let $L^p_t(I,L^q_x)$ be the space of measurable functions $f$ on $I\times \RR^3$ such that
$$ \|f\|_{L^p_t(I,L^q_x)}=\left(\int_I \|f(t)\|_{L^p}^q\,dt  \right)^{1/q}<\infty.$$ \\
In all the paper, we let
\begin{equation*}
 m:=\frac{p-1}{2}.
\end{equation*}
Hence $s_c=\frac{3}{2}-\frac{1}{m}$.\\
\\
Let $\chi$ be a radial smooth function such that
\begin{equation}
\left\{
\begin{array}{l}
\chi(|x|) = 1, \, |x| \geq \frac{1}{2} \\
\chi(|x|) = 0, \, |x| \leq \frac{1}{4}\cdot
\end{array}
\right.
\nonumber
\end{equation}
Given $R > 0$, we denote by $\chi_{R}(x) := \chi \left( \frac{|x|}{R} \right)$. We let $B_R$ be the Euclidean ball of $\RR^3$:
$$B_R=\left\{x\in \RR^3\;:\;|x|<R\right\}.$$
We denote by $\TTT_{R}$ the operator
\begin{equation}
\begin{array}{ll}
f & \rightarrow \TTT_{R}(f) :=
\left\{
\begin{array}{ll}
f(R), \, |x| \leq R \\
f(|x|), \,  |x| \geq R.
\end{array}
\right.
\end{array}
\nonumber
\end{equation}
\\
Let $S(t)$ denote the linear propagator, i.e
\begin{equation*}
S(t)(w_0,w_1) :=
\cos{(tD)}w_0+\frac{\sin{(tD)}}{D}w_1,\quad D=\sqrt{-\triangle}.
\end{equation*}
\\
The notation $\simP$ is defined in Subsection \ref{SS:def_profiles}.\\
\\
\underline{Assumption}: In this paper we consider only radial functions, i.e functions depending on $r:=|x|$ with $|\cdot|$ denoting the Euclidean norm on $\RR^3$. \footnote{Hence, if we write for instance $f \in \dot{H}^{s}$, then we do not only assume that $f$ lies in $\dot{H}^{s}(\mathbb{R}^{3})$ but also that $f$ is radial.}

\section{Preliminaries}

\label{Sec:preliminaries}
We recall in this section some facts about local well-posedness, singular statio\-nary solutions of (\ref{Eqn:WaveSup}) and profile decomposition.

\subsection{Local well-posedness}

We recall a local-wellposedness result:
\begin{prop}{(see for example \cite{KenMerleSuper})}
There exists $\delta_0  \ll 1$ and $C_{0} \gg1$ with the following properties.
Let $\overrightarrow{w_0} \in \HHH^{s_c}$. Then
\begin{itemize}
\item if an open interval $I$ containing $0$ satisfies
\begin{align}
\| S(t) \overrightarrow{w_0} \|_{L^{4m}_t(I,L^{4m}_x) }< \delta_0,
\label{Eqn:Smallness}
\end{align}
then there exists a unique solution
$$\vec{w}\in C \left( I, \, \HHH^{s_c}\right), $$
of (\ref{Eqn:WaveSup}) with initial data $\vec{w}(0):= \overrightarrow{w_0}$ such that

\begin{align*}
\| w \|_{L^{4m}_t (I,L^{4m}_x)}  < 2 \delta_0, \quad \| D^{s_{c} - \frac{1}{2}} w \|_{L^{4}_t(I,L^{4}_x)}
< \infty  ;
\end{align*}
\item  if $\| \overrightarrow{w_{0}} \|_{\HHH^{s_c}} \leq \frac{\delta_{0}}{C_{0}}$ then (\ref{Eqn:Smallness}) is satisfied with $I=\RR$, $w$ is global, scatters, and
\begin{equation*}
\max  \left( \| w \|_{L^{4m}_t (\mathbb{R},L^{4m}_x)}, \,
\| D^{s_{c} - \frac{1}{2}} w \|_{L^{4}_t (\mathbb{R},L^4_x)}, \,
\|  \vec{w} \|_{L_{t}^{\infty} \left(\mathbb{R},\HHH^{s_c}\right)} \right) <  \delta_0 \cdot
\end{equation*}
\end{itemize}
\label{Prop:CauchyPbHsc}
\end{prop}

The proof is based upon the Strichartz estimates (see for example \cite{GinibVelo,LinSog}):
\begin{multline}
 \| \vec{w} \|_{L^{\infty}_t\left(I,\HHH^{s_c} \right)} +
 \| D^{s_{c} - \frac{1}{2}} w \|_{L^{4}_t(I,L^{4}_x)}  \\
 + \| w \|_{L^{4m}_t(I,L^{4m}_x)} + \| w \|_{L_{t}^{\frac{5m}{2}} \left(I,L_{x}^{5m}\right) } \\
 \lesssim \left\| \overrightarrow{w_{0}} \right\|_{\HHH^{s_c} }
 + \| D^{s_{c} - \frac{1}{2}} F  \|_{L_{t}^{\frac{4}{3}} (I, L_{x}^{\frac{4}{3}})},
\label{Eqn:Strich}
\end{multline}
for any solution of $\partial_t^2w-\Delta w=F$ with initial data $\overrightarrow{w_0}$.
\begin{rem}
\begin{enumerate}

\item

More generally, if an $I$ is an interval containing $t_{0} \in \mathbb{R}$ satisfies $\| S(t-t_0) \overrightarrow{w_0} \|_{L_t^{4m} (I, L_x^{4m})} < \delta_0$,
then there exists a unique solution $\vec{w} \in C \left( I, \HHH^{s_c} \right)$ of (\ref{Eqn:WaveSup}) with $\vec{w}(t_0) := \overrightarrow{w_0}$ that satisfies the same conclusions as Proposition \ref{Prop:CauchyPbHsc}.

\item

This allows to define the maximal interval of existence $$I_{max}(w):= \left( T_{-}(w), T_{+}(w) \right),$$
 i.e the union of all the open
intervals  $J$ containing $t_0$ for which there exists a solution $v$ of (\ref{Eqn:WaveSup}) with the same initial data such that $\vec{v} \in C (J, \HHH^{s_c})$, $\| D^{s_c - \frac{1}{2}} v \|_{L_t^{4} (J; L_{x}^4)} < \infty$ and $\|v \|_{L^{4m}_t (J,L^{4m}_x)} < \infty$.

\item Let us also mention the following standard scattering criterion: if
$$ \| w \|_{L^{4m}_t ([t_0,T_+(w)),L^{4m}_x)} < \infty,$$
then $T_+(w)=+\infty$ and $w$ scatters forward in time.

\end{enumerate}
\label{Rem:ConseqLWP}
\end{rem}

We next recall a local well-posedness result, in $\HHH^1$ for an equation that is derived from (\ref{Eqn:WaveSup}) (see  \cite{DuyKenMerScattSuper}, Lemma 3.3).
\begin{prop}
\label{Prop:CauchyPbH1}
There exists $\delta_1>0$ with the following property. Let $R >0$. Let $I$ be an interval with $0\in I$, $V\in L^{4m}_t(I,L^{4m}_x)$, and $\overrightarrow{h_0} \in \HHH^1$ such that
\begin{gather}
\label{Eqn:smallV}
\|D_x^{1/2}V\|_{L^{4}_t(I,L^{4}_x)}<\delta_1,\quad \|V\|_{L^{4m}_t(I,L^{4m}_x)}<\delta_1\sqrt{R}^{\frac{p-5}{p-1}}\\
\label{Eqn:Assh0h1}
\| \overrightarrow{h_0} \|_{\HHH^1} <  \delta_1 \,  \sqrt{R}^{\frac{p-5}{p-1}}.
\end{gather}
Then the equation:
\begin{equation}
\left\{
\begin{array}{ll}
\partial_{tt} h - \triangle h  & = | V+\chi_{R} h|^{p-1} (V+\chi_{R} h)-|V|^{p-1}V \\
\vec{h}(0) & := \overrightarrow{h_0},
\end{array}
\right.
\nonumber
\end{equation}
has a unique solution $h$ such that
$$h\in L^{8}_t(I,L^8_x),\quad D_x^{1/2}h\in L^{4}_t(I,L^{4}_x), \quad \vec{h}\in C^{0}(I,\HHH^{1}).$$
Furthermore
\begin{equation*}
\sup_{t \in I}  \left\| \vec{h}(t)  - \overrightarrow{S(t) \overrightarrow{h_0}}(t)  \right\|_{\HHH^1}
\leq \frac{1}{100} \| \overrightarrow{h_0} \|_{\HHH^1}\cdot
\end{equation*}
If $V=0$, one can take $I=\RR$ and the preceding estimate can be upgraded to:
\begin{equation*}
\sup_{t \in \mathbb{R}}  \left\| \vec{h}(t)  - \overrightarrow{S(t) \overrightarrow{h_0}}(t)  \right\|_{\HHH^1}
\lesssim \frac{1}{R^{\frac{p-5}{2}}} \| \overrightarrow{h_0} \|^{p}_{\HHH^1}\cdot
\end{equation*}
\end{prop}

\subsection{Stationary solutions}
We next state a result regarding some singular stationary solutions of (\ref{Eqn:WaveSup}):
\begin{prop}
Let $\ell \in \mathbb{R}-\{0\}$. There exists $R_{\ell}\geq 0$ and a $C^{2}$ solution of
\begin{equation}
\label{Eqn:Zell1}
\triangle Z_{\ell}  + \iota |Z_{\ell}|^{p-1} Z_{\ell} = 0 \quad\text{on}\quad\mathbb{R}^{3} \cap \{ |x|>R_{\ell}\}
\end{equation}
such that
\begin{align}
\label{Eqn:Zell2}
|r Z_{\ell}(r)  - \ell| &\lesssim \frac{1}{r^{2}}, \quad r \gg 1\\
\label{Eqn:Zell3}
\lim_{r \rightarrow \infty} r^{2} \frac{d Z_{\ell}}{d r} &= - \ell
\end{align}
Furthermore
\begin{itemize}
 \item if $\iota=+1$ (focusing nonlinearity),  $R_{\ell}=0$ and $Z_{\ell} \notin \dot{H}^{s_c}$.
\item if $\iota=-1$ (defocusing nonlinearity), $R_{\ell}>0$ and
\begin{equation}
\label{Eqn:limZl}
\lim_{r\to R_{\ell}}|Z_{\ell}(r)|=+\infty.
\end{equation}
\end{itemize}
\label{Prop:SolStat}
\end{prop}
\begin{proof}
The case of a focusing nonlinearity $\iota=1$ is treated in \cite[Proposition 3.2]{DuyKenMerScattSuper}, and we will only consider a defocusing nonlinearity $\iota=-1$.

In this case, the existence of a solution $Z_{\ell}$ defined for large $r$ and satisfying (\ref{Eqn:Zell1}) (for large $r$), (\ref{Eqn:Zell2}) and (\ref{Eqn:Zell3}) follows also from the proof of Proposition 3.2 of \cite{DuyKenMerScattSuper}. Let $(R_{\ell},+\infty)$ be the maximal interval of existence of $Z_{\ell}$, as a solution of the ordinary differential equation (in the $r$ variable):
$$ Z_{\ell}''+\frac{2}{r}Z_{\ell}'-|Z_{\ell}|^{p-1}Z_{\ell}=0.$$
We now prove that $R_{\ell}>0$. We assume $\ell>0$ to fix ideas and argue by contradiction. Assume $R_{\ell}=0$. Let $h(s)=Z_{\ell}(1/s)$, $s\in (0,\infty)$. Then $h$ is a $C^2$ solution of
\begin{equation}
 \label{Eqn:ODE_h}
h''(s)=\frac{1}{s^4}|h(s)|^{p-1}h(s),\quad s>0
\end{equation}
that satisfies (by (\ref{Eqn:Zell2}), (\ref{Eqn:Zell3})):
$$\lim_{s\to 0} \frac{h(s)}{s}=\ell,\quad \lim_{s\to 0} h'(s)=\ell.$$
By (\ref{Eqn:ODE_h}),
$$\lim_{s\to 0} \frac{h''(s)}{s^{p-4}}=\ell^p,$$
and thus $h''(s)>0$ for small positive $s$.
Combining these estimates with equation (\ref{Eqn:ODE_h}) and a simple bootstrap argument we obtain:
\begin{equation}
\label{Eqn:lower_bnd_h}
 \forall s>0,\quad h''(s)>0,\quad h'(s)\geq \frac{\ell}{2},\quad h(s)\geq \frac{\ell}{2} s.
\end{equation}
We next prove by induction:
\begin{equation}
 \label{Eqn:lower_bnd_h_2}
\forall n\geq 1,\quad \exists c_n>0,\quad \forall s>0,\quad h(s)\geq c_ns^n.
\end{equation}
Indeed, (\ref{Eqn:lower_bnd_h_2}) holds for $n=1$. Assuming that it holds for some $n\geq 1$, we obtain by (\ref{Eqn:ODE_h})
$$ h''(s)\geq c_n^p s^{np-4}.$$
Integrating twice between $0$ and $s$, we deduce $h(s)\geq c_n's^{pn-2}$, which yields (\ref{Eqn:lower_bnd_h_2}) at rank $n+1$, since $pn-2>5n-2>n+1$.

Let
$$F(s)=\frac{1}{2}\left(h'(s)\right)^2-\frac{1}{(p+1)s^4} (h(s))^{p+1}.$$
Then $F'(s)=\frac{4}{(p+1)s^5}h^{p+1}$. By (\ref{Eqn:lower_bnd_h_2}), $\lim_{s\to\infty}F(s)=+\infty$. It follows that for large $s$,
$$ h'(s)\geq \sqrt{\frac{2}{p+1}} \frac{h^{\frac{p+1}{2}}(s)}{s^2}\geq \frac{1}{C}h^2(s),$$
by (\ref{Eqn:lower_bnd_h_2}) again. Hence for large $s$,
$$\frac{d}{ds}\frac{1}{h(s)}\leq -1/C,$$
a contradiction with the fact that $h$ is positive and defined on $(0,\infty)$. \\
By the standard blow-up criterion $|Z_{\ell} (r)| + |Z^{'}_{\ell}(r)| \rightarrow_{r \rightarrow R_{\ell}} \infty$.
Notice that since $Z_{\ell}$ decreases, this means that $Z_{\ell}(r) $ has a limit as $r \rightarrow R_{\ell}$. But
this limit cannot be finite. If not this would imply by (\ref{Eqn:ODE_h}) that $Z^{'}_{\ell}$
is bounded, which is not possible. Hence (\ref{Eqn:limZl}) holds.
\end{proof}

\begin{rem}
\label{R:Zell}
 Note that
 $$\lambda^{\frac{2}{p-1}}Z_{\ell}(\lambda r)$$
 is a stationary solution of (\ref{Eqn:WaveSup}), defined for $r>\frac{R_{\ell}}{\lambda}$, and such that
 $$\lambda^{\frac{2}{p-1}}Z_{\ell}(\lambda r)\sim \lambda^{\frac{2}{p-1}-1}\frac{\ell}{r}=\lambda^{\frac{3-p}{p-1}}\frac{\ell}{r}$$
 as $r\to\infty$. By uniqueness in the fixed point defining $Z_{\ell}$ (see the proof of Proposition 3.2 in \cite{DuyKenMerScattSuper}), if $\ell>0$,
 $$Z_{\ell}(r)=\lambda^{\frac{2}{p-1}}Z_{1}(\lambda r),\quad R_{\ell}=\frac{R_1}{\lambda},\quad \lambda:=\ell^{-\frac{p-1}{p-3}},$$
 and if $\ell<0$,
$$Z_{\ell}(r)=-\lambda^{\frac{2}{p-1}}Z_{1}(\lambda r),\quad R_{\ell}=\frac{R_1}{\lambda},\quad \lambda:=|\ell|^{-\frac{p-1}{p-3}}.$$
 \end{rem}

\begin{rem}
 \label{R:potentiel}
We will often linearize equation (\ref{Eqn:WaveSup}) around the singular solution $Z_{1}$, and use Proposition \ref{Prop:SolStat}. To make this possible we will need the following estimates about two potentials obtained from $Z_{1}$ by truncation ($\chi$ and $\TTT_R$ are defined in the notations given at the end of the introduction):
\begin{enumerate}
 \item \label{I:V1}Let $V := \TTT_{R}Z_{1}$, $R>R_{1}$. Then there exists $\theta:=\theta(R)>0$ such that $V$, $I=\left[-\theta,\theta\right]$ satisfy the assumptions of Proposition \ref{Prop:CauchyPbH1}.
\item \label{I:V2}Let $S_1>R_{1}$ be a large parameter, and
$$V(t,x) := \chi\left(\frac{x}{S_1+|t|}\right)Z_{1}(x).$$
Then $V$ satisfies the assumptions of Proposition \ref{Prop:CauchyPbH1} with $I=\RR$.
\end{enumerate}
\end{rem}
If $R>0$, we will denote by $z_R$ the solution with initial data $(\TTT_R Z_1,0)$. Then we have the following:
\begin{lem}
 \label{Lem:zR}
 \begin{enumerate}
  \item \label{I:zRa} There exists $\rho_Z>R_1$ such that, if $R>\rho_Z$, $z_R$ is global and scatters in both time directions.
 \item \label{I:zRb} For $\rho>0$, denote by
 $$\theta_{\rho}:=\inf_{R\geq \rho} T_+(z_R)\in [0,\infty].$$
 Then $\theta_{\rho}$ is a strictly positive, nondecreasing function of $\rho$.
 \end{enumerate}
\end{lem}
\begin{proof}
Point (\ref{I:zRa}) follows from the small data theory and Result \ref{Res:Cont_R}.

 The fact that $\theta_{\rho}$ is nondecreasing follows immediately from the definition.
 The positivity of $\theta_{\rho}$ is obvious if $\rho>\rho_Z$: in this case $\theta_{\rho}=+\infty$. If $\rho\leq \rho_Z$, we have:
 $$\theta_{\rho}:=\inf_{\rho\leq R\leq \rho_{Z}} T_+(z_R).$$
 Since
 $$\left\{(z_R(0),0),\; \rho\leq R\leq \rho_{Z}\right\}$$
 is a compact subset of $\HHH^{s_c}$, the result follows from the Cauchy theory.
\end{proof}

\subsection{Profile decomposition}
\label{Sec:ProfDecomp}

\subsubsection{Definition}
\label{SS:def_profiles}
Throughout the manuscript, we constantly use the profile decomposition of a bounded sequence $ \{ (w_{0,n},w_{1,n}) \}_{n \in \mathbb{N}}$
in $ \dot{\mathcal{H}}^{s_{c}}$. Adapting the proof of \cite{BahGer}, there exists a subsequence
of $ \{ (w_{0,n},w_{1,n}) \}_{ n \in \mathbb{N} } $  (still denoted by $ \{ (w_{0,n},w_{1,n}) \}_{n \in \mathbb{N}}$) such that
for all $ j \geq 1$,  there exist sequences $\{ \lambda_{j,n} \}_{n \in \mathbb{N}}$,
$\{ t_{j,n} \}_{n \in \mathbb{N}}$, $\{ \BA{\epsilon_{0,n}^{J}} \}_{n \in \mathbb{N}}$ with $\lambda_{j,n} > 0$, $t_{j,n} \in \mathbb{R}$, and $\BA{\epsilon_{0,n}^{J}} \in \HHH^{s_c}$ such that:
\begin{align}
j \neq k \Longrightarrow \lim_{n \rightarrow \infty} \frac{\lambda_{j,n}}{\lambda_{k,n}} + \frac{\lambda_{k,n}}{\lambda_{j,n}} +
\frac{|t_{j,n} - t_{k,n}|}{\lambda_{j,n}} = \infty,
\label{Eqn:PseudoOrth}
\end{align}
and for all $J \geq 1$
\begin{multline*}
(w_{0,n},w_{1,n})(x)=\\
\sum_{j=1}^{J}
\left(
\frac{1}{\lambda_{j,n}^{\frac{2}{p-1}}} W^{j}_{\lin} \left( - \frac{t_{j,n}}{\lambda_{j,n}}, \frac{x}{\lambda_{j,n}} \right),
\frac{1}{\lambda_{j,n}^{\frac{2}{p-1} +1 }} \partial_{t} W^{j}_{\lin} \left( - \frac{t_{j,n}}{\lambda_{j,n}}, \frac{x}{\lambda_{j,n}}
\right) \right)
+ \BA{\eps_{0,n}^{J}}(x),
\end{multline*}
with $W^{j}_{\lin} (t):= S(t)(W_{0}^{j}, W_{1}^{j})$ and
\begin{equation}
\lim_{J \rightarrow \infty} \limsup_{n \rightarrow \infty} \| \eps_{n}^{J} \|_{L^{4m}_t (\mathbb{R},L^{4m}_x)} =0
\label{Eqn:RemSmall}
\end{equation}
(Here $\eps_{n}^{J}(t) := S(t)\BA{\eps_{0,n}^{J}}$).

We say that $\{ (w_{0,n},w_{1,n} ) \}_{n \in \mathbb{N}}$ has a profile decomposition
$ \left\{  W^{j}_{\lin},  \left\{ (t_{j,n}, \lambda_{j,n}) \right\}_{n \in \mathbb{N}} \right\}_{j \geq 1}$
with profiles $\{ W_{\lin}^{j} \}_{j\geq 1}$ and
parameters $\left\{ \left\{ (\lambda_{j,n}, t_{j,n}) \right\}_{n \in \mathbb{N}} \right\}_{j\geq 1}$ and that $\BA{\eps_{0,n}^{J}}$ is the remainder.

Let $w_{\lin,n}(t):= S(t) (w_{0,n},w_{1,n})$.  The profiles satisfy, for all $j,J\geq 1$,
\begin{align}
\label{Eqn:WeakProfile1}
\left( \lambda_{j,n}^{\frac{2}{p-1}} w_{\lin,n} (t_{j,n}, \lambda_{j,n} \cdot), \lambda_{j,n}^{\frac{2}{p-1} +1} \partial_{t} w_{\lin,n} (t_{j,n}, \lambda_{j,n}
\cdot)
\right) &\xrightharpoonup[n \rightarrow \infty]{} (W_{0}^{j},W_{1}^{j}), \\
\label{Eqn:WeakProfile2}
j\leq J\Longrightarrow
\left(
\lambda_{j,n}^{\frac{2}{p-1}} \eps_{n}^{J} (t_{j,n}, \lambda_{j,n} \cdot), \lambda_{j,n}^{\frac{2}{p-1} + 1 } \partial_{t} \eps_{n}^{J} (t_{j,n},
\lambda_{j,n}
\cdot)
\right) &\xrightharpoonup[n \rightarrow \infty]{} 0,
\end{align}
and
\begin{equation}
\label{Eqn:Pyth}
\left\| (w_{0,n},w_{1,n}) \right\|^{2}_{ \dot{\mathcal{H}}^{s_{c}} }
= \sum_{j=1}^{J}  \left\| \BA{W_{\lin}^{j}} \left( - \frac{t_{j,n}}{ \lambda_{j,n}}
\right)   \right\|^{2}_{ \dot{\mathcal{H}}^{s_{c}}}
 + \| \BA{\eps_{0,n}^{J}}\|^{2}_{ \dot{\mathcal{H}}^{s_{c}}} + o_{n}(1).
\end{equation}
Translating in time and rescaling $W_{\lin}^{j}$, and extracting a subsequence, we may assume
\begin{equation}
\forall j\geq 1,\quad \lim_{n \rightarrow \infty}  - \frac{t_{j,n}}{\lambda_{j,n}} = \pm \infty \quad \text{or} \quad \forall n,\;t_{j,n}=0.
\label{Eqn:AssParam}
\end{equation}
Recall that it possible, assuming (\ref{Eqn:AssParam}), to construct a solution $W^{j}$ of (\ref{Eqn:WaveSup}) that
satisfies
\begin{align}
\lim_{n \rightarrow \infty} \left\| \BA{W^{j}} \left( -\frac{ t_{j,n}}{\lambda_{j,n}} \right) - \BA{W^{j}_{\lin}} \left( - \frac{t_{j,n}}{\lambda_{j,n}}
\right)  \right\|_{ \dot{\mathcal{H}}^{s_{c}} } =0.
\label{Eqn:ScattProf}
\end{align}
See for example \cite{KenMerleSuper}. Such a
construction is unique and we say that $W^{j}$ is the nonlinear profile associated to
$W^{j}_{\lin}$ and $ \left\{ (t_{j,n}, \lambda_{j,n}) \right\}_{n \in \mathbb{N}}$. Recall that if
$ \lim_{n \rightarrow \infty}  - \frac{t_{j,n}}{\lambda_{j,n}} = \infty $ (resp. $-\infty$)
and $\{ s_n \}_{n \in \mathbb{N}}$ is a sequence such that  $\lim_{n \rightarrow \infty} s_n = \infty$ ( resp. $- \infty$) then

\begin{align}
\left\| \BA{W^{j}} (s_n) - \BA{W^{j}_{\lin}} (s_n)  \right\|_{ \dot{\mathcal{H}}^{s_{c}} } =0.
\label{Eqn:ScattProf2}
\end{align}
We also denote by $W_{n}^{j}$ (resp. $W_{\lin,n}^{j}$) the normalized nonlinear profile (resp. the normalized linear profile), i.e
\begin{align}
\label{Eqn:defWnj}
W_{n}^{j}(t,x) &:= \frac{1}{\lambda_{j,n}^{\frac{2}{p-1}}} W^{j} \left( \frac{t - t_{j,n}}{\lambda_{j,n}}, \frac{x}{ \lambda_{j,n}} \right) \\
\label{Eqn:defWlnj}
W_{\lin,n}^{j}(t,x) &:= \frac{1}{\lambda_{j,n}^{\frac{2}{p-1}}} W_{\lin}^{j} \left( \frac{t - t_{j,n}}{\lambda_{j,n}}, \frac{x}{ \lambda_{j,n}} \right)
\end{align}
If $\{ (w_{0,n},w_{1,n}) \}_{n \in \mathbb{N}}$ is a bounded sequence in $\HHH^{s_c}$, we will write:
\begin{equation}
\label{Eq:sim_p}
 (w_{0,n},w_{1,n}) \simP \sum_{j\geq 1} \left(\frac{1}{\lambda_{j,n}^{\frac{2}{p-1}}} W^j_{\lin}\left(\frac{-t_{j,n}}{\lambda_{j,n}},\frac{\cdot}{\lambda_{j,n}}\right),\frac{1}{\lambda_{j,n}^{1+\frac{2}{p-1}}} \partial_{t}W^j_{\lin}\left(\frac{-t_{j,n}}{\lambda_{j,n}},\frac{\cdot}{\lambda_{j,n}}\right)\right),
\end{equation}
when $\{ (w_{0,n},w_{1,n}) \}_{n \in \mathbb{N}}$  has a profile decomposition with profiles $\{W^j_{\lin}\}_j$ and parameters
$\left\{ (\lambda_{j,n}, t_{j,n})_{n \in \mathbb{N}} \right\}_{j\geq 1}$. We will often shorten (\ref{Eq:sim_p}) into
\begin{equation*}
(w_{0,n},w_{1,n})  \simP \sum_{j\geq 1} \BA{W^{j}_{\lin,n}}(0).
\end{equation*}
If moreover $\{\rho_n\}_n$ is a sequence of positive numbers, we will write
\begin{equation}
\label{Eq:sim_p_out}
(w_{0,n},w_{1,n})  \simP \sum_{j\geq 1} \BA{W^{j}_{\lin,n}}(0),\quad |x|>\rho_n,
\end{equation}
 when there exists a sequence $ \{ ( \tilde{w}_{0,n}, \tilde{w}_{1,n}) \}_{n \in \mathbb{N}}$, bounded in $\HHH^{s_c}$ such that
\begin{equation*}
( \tilde{w}_{0,n}, \tilde{w}_{1,n} ) \simP \sum_{j\geq 1} \BA{W^{j}_{\lin,n}}(0),
\end{equation*}
 and for all $n$,
 \begin{equation*}
 ( w_{0,n},w_{1,n} )(r)  = ( \tilde{w}_{0,n}, \tilde{w}_{1,n})(r),\quad |x|>\rho_n.
 \end{equation*}

\subsubsection{Approximation by nonlinear profiles}
The next proposition states that under suitable assumptions, the solution $w_{n}$ of (\ref{Eqn:WaveSup}) with initial data
$\vec{w}_{0,n}$ can be well-approximated, for large $n$, by sum of the nonlinear profiles. We omit the proof, which is close to the proof of the main theorem in \cite{BahGer}, using the perturbation theorem in Section 2 of \cite{KenMerleSuper}.
\begin{prop}
Let $\{ \BA{w_{0,n}} \}_{n \in \mathbb{N}}$ be a bounded sequence in $ \dot{\mathcal{H}}^{s_{c}}$
that has a profile decomposition $\{ W_{\lin}^{j},(t_{j,n}, \lambda_{j,n})_{n}  \}_{j\geq 1}$. Let $\{ \theta_{n} \}_{n \in \mathbb{N}}$ be a sequence such
that $\theta_{n} \in [0, \infty)$. Assume that for all $j$, the nonlinear profile $W^j$ scatters forward in time or:
\begin{equation}
\label{Eqn:limsup}
\limsup_{n\to\infty}\frac{\theta_{n} - t_{j,n}}{\lambda_{j,n}} < T_{+}(W^{j}).
\end{equation}
Let $w_{n}$ be the solution of (\ref{Eqn:WaveSup}) with data $\BA{w_{0,n}}$. Then for $n \gg 1$, $w_{n}$ is defined
on $[0, \theta_{n}]$,
\begin{equation*}
\limsup_{n \rightarrow \infty} \| w_{n} \|_{L_{t}^{4m} ((0, \theta_{n}),L^{4m}_x)} < \infty,
\end{equation*}
and
\begin{equation}
\label{Eqn:Conclu_Perturb}
\forall t \in [0,\theta_{n}], \quad w_{n}(t,x) = \sum_{j=1}^{J} W_{n}^{j}(t,x) + \eps_{n}^{J}(t,x) + r_{n}^{J}(t,x),
\end{equation}
with
\begin{equation*}
\lim_{J \rightarrow \infty}
\left[ \limsup_{n \rightarrow \infty}
\left( \| r_{n}^{J} \|_{L_{t}^{4m} ((0,\theta_{n}),L^{4m}_x) }
 +  \| \BA{r_{n}^{J}} \|_{ L_{t}^{\infty} \left((0,\theta_{n}), \dot{\mathcal{H}}^{s_{c}} \right)}\right)
\right]
=0.
\end{equation*}
A similar conclusion holds if $\theta_{n} \in (-\infty,0]$ for all $n$, and, for all $j$, $W^j$ scatters backward in time or
\begin{equation}
\label{Eqn:liminf}
\liminf_{n\to\infty}\frac{\theta_{n} - t_{j,n}}{\lambda_{j,n}}> T_{-}(W^{j}).
\end{equation}
\label{Prop:Perturb}
\end{prop}
\begin{rem}
\label{Rem:Perturb}
 Under the assumptions of Proposition \ref{Prop:Perturb}, we can deduce from (\ref{Eqn:Conclu_Perturb}) and after extraction of a subsequence, a profile decomposition of $\vec{w}_n(\theta_n)$.

Let $s_{j,n}:=t_{j,n}-\theta_n$.
By a diagonal extraction argument, we can assume that for all $j\geq 1$, the following limit exists:\begin{equation}
\label{Eqn:lim_exist}
\lim_{n\to\infty} \frac{-s_{j,n}}{\lambda_{j,n}}\in [-\infty,+\infty]
\end{equation}
For $j\geq 1$, we have
$$ \lim_{n\to\infty}\left\|\BA{W_n^j } \left(\theta_n\right)- \BA{V^j_{\lin,n}}(0)\right\|_{\HHH^{s_c}}=0,$$
where
$$V_{\lin,n}^j(t,x)=\frac{1}{\lambda_{j,n}^{\frac{2}{p-1}}}V^j_{\lin}\left( \frac{t-s_{j,n}}{\lambda_{j,n}},\frac{x}{\lambda_{j,n}} \right),$$
and $V_{\lin}^j$ is the only solution of the linear wave equation that satisfies
\begin{equation*}
\lim_{n\to\infty}\left\|\BA{W^j}\left(\frac{-s_{j,n}}{\lambda_{j,n}}\right)-\BA{V^j_{\lin}}\left(\frac{-s_{j,n}}{\lambda_{j,n}}\right)\right\|_{\HHH^{s_c}}=0.
\end{equation*}
(The existence and uniqueness of $V_{\lin}^j$ follows from (\ref{Eqn:lim_exist}).)

One easily checks that if $j,k\geq 1$, $j\neq k$,
$$\lim_{n\to\infty}\left|\log\frac{\lambda_{j,n}}{\lambda_{k,n}}\right|+\left|\frac{s_{j,n}-s_{k,n}}{\lambda_{k,n}}\right|= \lim_{n\to\infty}\left|\log\frac{\lambda_{j,n}}{\lambda_{k,n}}\right|+\left|\frac{t_{j,n}-t_{k,n}}{\lambda_{k,n}}\right|=+\infty.$$

Finally, the conclusion of Proposition \ref{Prop:Perturb} implies
$$\BA{w_n}(\theta_n)\simP \sum_{j\geq 1} \BA{V^j_{\lin,n}}(0).$$
 \end{rem}
Let us conclude this section by giving a version of Proposition \ref{Prop:Perturb} and Remark \ref{Rem:Perturb} outside a ball that is a direct consequence
of the finite speed of propagation:
\begin{prop}
 \label{Prop:Perturb_outside}
 Let $\{\BA{w_{0,n}} \}_n$ be a bounded sequence in $\HHH^{s_c}$ and $\{\rho_n\}_n$ be a sequence of positive numbers. For all $n$, denote by $w_n$ the solution of (\ref{Eqn:WaveSup}) with initial data $\BA{w_{0,n}}$ and let $\theta_n \in [0,T_+(w_n))$. Assume that there exist profiles $\{W^j_{\lin}\}_{j\geq 1}$ and parameters $\left\{ (\lambda_{j,n},t_{j,n})_n \right\}_j$ such that
 $$ \BA{w_{0,n}} \simP \sum_{j\geq 1} \BA{W^j_{\lin,n}}(0),\quad |x|>\rho_n,$$
 and that satisfy the assumptions of Proposition \ref{Prop:Perturb}. Then
 $$ \BA{w_n}(\theta_n)\simP \sum_{j\geq 1} \BA{V^j_{\lin,n}}(0),\quad |x|>\rho_n+\theta_n,$$
 where the modulated linear profiles $V_{\lin,n}^j$ are as in Remark \ref{Rem:Perturb}.

 An analogous statement holds for negative times.
 \end{prop}
Proposition \ref{Prop:Perturb_outside} follows readily from Proposition \ref{Prop:Perturb}, Remark \ref{Rem:Perturb} and finite speed of propagation.

\section{Generalized energies for the linear equation}

\label{Sec:energy}

\subsection{Definition}
Recall that $m=\frac{p-1}{2}$, and that all the functions are implicitely assumed to be radial in the space variable.
\begin{definition}
\label{D:Lm_energy}
Let $w_{\lin}$ be a solution of the linear wave equation
\begin{equation}
\label{Eqn:LW}
\left\{
\begin{aligned}
\partial_{tt}w_{\lin}-\triangle w_{\lin}&=0\\
\BA{w_{\lin}}(0)& : =(w_0,w_1)
\end{aligned}
\right.
\end{equation}
 The \emph{$L^m$-generalized energy} (in brief \emph{$L^m$-energy}) of $w_{\lin}$ is the time-dependent quantity:
$$ E_m[w_{\lin}](t)=\int_0^{\infty} |\partial_r(rw_{\lin})|^m+|\partial_t(rw_{\lin})|^m\,dr.$$
\end{definition}
Note that the $L^2$-generalized energy is (up to a multiplicative constant) the standard energy of $w_{\lin}$, which is well-defined, and independent of $t$, if $(w_0,w_1)\in \HHH^1$. In this section, we will prove that the $L^m$-generalized energy is well defined if $(w_0,w_1)\in \HHH^{s_c}$ and is almost conserved with time, in the sense that the $L^m$-energy of a solution at any time is comparable to the energy at $t=0$. We will also prove a lower bound for the exterior generalized energy outside a wave cone, analog to the exterior energy estimate used to classify solutions of the energy-critical wave equation (see \cite[Lemma 4.2]{DuyKenMerSmall}).

\subsection{Preliminary estimates}

\begin{lem}

If $\phi\in \dot{H}^{s_c-1}$, then $r^{1-\frac{2}{m}}\phi\in L^m$, and
\begin{equation}
\| r^{1-\frac{2}{m}}\phi \|_{L^{m}} \lesssim \| \phi \|_{\dot{H}^{s_c-1}}.
\label{Eqn:HardyPosOrigin1}
\end{equation}
If $\phi\in \dot{H}^{s_c}$, then $r^{1-\frac{2}{m}}\partial_r\phi\in L^m$, $r^{-\frac{2}{m}}\phi\in L^m$, $\phi\in L^{3m}$,  and
\begin{equation}
\| r^{1-\frac{2}{m}} \partial_{r} \phi \|_{L^{m}} \lesssim  \| \phi \|_{\dot{H}^{s_c}},
\label{Eqn:HardyDerOrigin}
\end{equation}
and
\begin{equation}
\left\| r^{-\frac{2}{m}}\phi \right\|_{L^{m}} +\left\|\phi\right\|_{L^{3m}}\lesssim \| r^{1-\frac{2}{m}} \partial_{r} \phi \|_{L^{m}}.
\label{Eqn:HardyPosInfty}
\end{equation}
\label{Lem:HardyIneq}
\end{lem}
\begin{proof}
Estimates (\ref{Eqn:HardyPosOrigin1}) and (\ref{Eqn:HardyDerOrigin}) are given by Lemma 3.2 of \cite{KenMerleSuper}.

It remains to prove (\ref{Eqn:HardyPosInfty}). We start with the bound of the norm of $r^{-\frac{2}{m}}\phi$ in $L^m$. For further use we will prove a slightly more general estimate. By density, we can assume $\phi\in C_0^{\infty}$. Let $R\geq 0$. An integration by parts gives:
$$-\int_R^{\infty} |\phi(r)|^m\,dr-|\phi(R)|^mR=\int_R^{\infty} r\partial_r(|\phi|^m)\,dr= m \int_R^{\infty} r \partial_r\phi \phi|\phi|^{m-2} \,dr.$$
By H\"older's inequality,
$$R|\phi(R)|^m+\int_R^{\infty} |\phi|^m\,dr\leq m\left(\int_R^{\infty} r^m|\partial_r\phi|^m\,dr\right)^{\frac 1m}\left(\int_R^{\infty}|\phi|^m\,dr\right)^{\frac{m-1}{m}}.$$
Combining with Young's inequality
\begin{equation}
\label{Eqn:Young}
ab\leq \frac{1}{m}(ca)^m+\frac{m-1}{m}(b/c)^{\frac{m}{m-1}},\quad a,b,c>0,
\end{equation}
we obtain
\begin{equation}
 \label{Eqn:useful}
 \int_R^{\infty} |\phi|^m\,dr+R |\phi(R)|^m \lesssim \int_R^{\infty} r^m|\partial_r\phi|^m\,dr.
\end{equation}
Letting $R=0$, we obtain the bound of the first term of the left-hand side in (\ref{Eqn:HardyPosInfty}).

We next prove the other bound in (\ref{Eqn:HardyPosInfty}). By the critical Sobolev inequality,
\begin{equation*}
 \|\phi\|_{L^{3m}}=\left\| |\phi|^{\frac{m}{2}}\right\|^{\frac{2}{m}}_{L^{6}} \lesssim \left\| \partial_r\left( |\phi|^{\frac m2}\right)\right\|_{L^2}^{\frac{2}{m}}
\end{equation*}
By H\"older inequality, and using that $\big|\partial_r |\phi| \big |=|\partial_r \phi|$ on the set $\{\phi \neq 0\}$, we deduce
\begin{multline*}
 \|\phi\|_{L^{3m}}\lesssim \left\|\partial_r \phi\,|\phi|^{\frac{m-2}{2}}\right\|_{L^2}^{\frac{2}{m}} \lesssim \left\|r^{1-\frac{2}{m}} \partial_r \phi\right\|^{\frac{2}{m}}_{L^m}\left\|r^{\frac{2}{m}-1}|\phi|^{\frac{m-2}{2}}\right\|^{\frac{2}{m}}_{L^{\frac{2m}{m-2}}}\\ = \left\|r^{1-\frac{2}{m}} \partial_r \phi\right\|^{\frac{2}{m}}_{L^m}\left\|r^{-\frac{2}{m}}\phi\right\|^{\frac{m-2}{m}}_{L^{m}},
\end{multline*}
and the result follows, using the bound of the first term of the left-hand side in (\ref{Eqn:HardyPosInfty}).
\end{proof}

\begin{lem}
Let $R\in [0,\infty)$. Then for all $\phi \in \dot{H}^1$,
\begin{equation}
\int_{R}^{\infty}  r^{2} |\partial_{r} \phi|^{2} \, dr = \int_{R}^{\infty} |\partial_{r} (r\phi)|^{2} \, dr  + R |\phi(R)|^{2}.
\label{Eqn:UtoVprec}
\end{equation}
and for all $\phi\in \dot{H}^{s_c}$,
\begin{equation}
\int_{R}^{+\infty}  r^{m} |\partial_{r} \phi|^{m} \, dr \approx  \int_{R}^{\infty} |\partial_{r} (r\phi)|^{m} \, dr  + R |\phi(R)|^{m}.
\label{Eqn:UtoV}
\end{equation}
\label{Lem:UtoV}
\end{lem}

\begin{rem}
\begin{enumerate}
\item Notation: if $R:=0$ then $R |\phi(R)|^{2}:=0$ and
$R |\phi(R)|^{m}:=0$. \\
\item Note that $|\phi(R)|^{m}$ is well-defined for $R>0$ since $\phi\in \dot{H}^{s_c}$ is a continuous function of the radial variable outside the origin.
\end{enumerate}
\end{rem}

\begin{proof}
The proof of (\ref{Eqn:UtoVprec}) is straightforward and therefore omitted.

We prove (\ref{Eqn:UtoV}). By density we can assume $\phi\in C^{\infty}$.

\medskip

\noindent\emph{Proof of the estimate $\gtrsim$}.

\begin{equation*}
 \int_R^\infty |\partial_r(r\phi)|^m\,dr=\int_R^{\infty} |\phi+r\partial_r\phi|^m\,dr
 \lesssim \int_R^{\infty} |\phi|^m\,dr+\int_R^{\infty} |\partial_r\phi|^mr^m\,dr,
\end{equation*}
Combining with the estimate (\ref{Eqn:useful}) in the proof of Lemma \ref{Lem:HardyIneq}, we obtain
\begin{equation*}
\int_{R}^{+\infty}  r^{m} |\partial_{r} \phi|^{m} \, dr \gtrsim \int_{R}^{\infty} |\partial_{r} (r\phi)|^{m} \, dr  + R |\phi(R)|^{m}.
\end{equation*}

\medskip

\noindent\emph{Proof of the estimate $\lesssim$}.
In view of
$$\int_{R}^{\infty}|\partial_r\phi|^mr^m\,dr=\int_R^{\infty} |\partial_r(r\phi)-\phi|^m\,dr\lesssim \int_R^{\infty} |\partial_r(r\phi)|^m\,dr+\int_R^{\infty}|\phi|^m\,dr,$$
we are reduced to prove
\begin{equation}
 \label{boundLm}
 \int_R^{\infty} |\phi|^m\,dr\lesssim R|\phi(R)|^m+\int_R^{\infty}|\partial_r(r\phi)|^m\,dr.
\end{equation}
For this, we write
\begin{multline*}
 \int_R^{\infty}  \partial_r(r\phi) \phi |\phi|^{m-2}\,dr=\frac 1m\int_R^{\infty} \partial_r(|r\phi|^m)\frac{1}{r^{m-1}}\,dr\\
=\frac{m-1}{m} \int_R^{\infty} |\phi|^m\,dr-\frac{R}{m}|\phi(R)|^m.
\end{multline*}
By H\"older's inequality, we obtain
$$\frac{m-1}{m} \int_R^{\infty} |\phi|^m\,dr\leq \frac{R}{m}|\phi(R)|^m +\left(\int_R^{\infty} |\partial_r(r\phi)|^m\,dr\right)^{\frac{1}{m}}\left(\int_R^{\infty} |\phi|^m\,dr\right)^{\frac{m-1}{m}},$$
which yields (\ref{boundLm}) using Young's inequality (\ref{Eqn:Young}).
\end{proof}
\begin{rem}
 By Lemma \ref{Lem:HardyIneq}, we see that $E_m[w_{\lin}](t)$ is well defined if the initial data $(w_0,w_1)$ is in $\HHH^{s_c}$. By Lemma \ref{Lem:UtoV},
 $$E_m[w_{\lin}](t)\approx \int_0^{\infty}r^m|\partial_rw_{\lin}(t,r)|^m\,dr+\int_0^{\infty}r^m|\partial_tw_{\lin}(t,r)|^m\,dr.$$
\label{Rem:EquivGenNrj}
\end{rem}

\subsection{Almost conservation of the generalized energy}
\begin{lem}
Let $w_{\lin}$ be a solution of the linear wave equation, and $E_m(t)=E_m[w_{\lin}](t)$ be the $L^m$-energy given by Definition \ref{D:Lm_energy}. Then
\begin{equation}
\label{Eqn:Conserv}
\forall t\in \RR,\quad E_m(t) \approx E_m(0).
\end{equation}
\label{Lem:ConservQuantLin}
\end{lem}
\begin{proof}
Notice that if $w_{\lin}$ is the solution of the linear wave equation with data
$(w_{0},w_{1}) \in \HHH^{s_c} $, then
there exists $f \in L^{m}_{loc} (\mathbb{R})$ such that $\dot{f} \in L^{m} (\mathbb{R})$ and
\begin{equation}
r w_{\lin}(t,r) = f(t+r) - f(t-r)
\label{Eqn:wf}
\end{equation}
%
More precisely, if $r>0$,
\begin{align*}
\dot{f}(r) = \frac{1}{2} \big( \partial_{r} (r w_{0}) (r) + r w_{1}(r) \big), \\
\dot{f}(-r) = \frac{1}{2}  \big( \partial_{r} (r w_{0}) (r) - r w_{1}(r) \big),
\end{align*}
and
\begin{equation}
f(s) =
\begin{cases}
s >0\,: & \frac{1}{2} s w_{0} (s) + \frac{1}{2} \int_{0}^{s} (\sigma w_{1})(\sigma) \, d \sigma \\
s< 0\,:  & \frac{1}{2}  s w_{0} (-s) + \frac{1}{2} \int_{0}^{-s} (\sigma w_{1}) (\sigma) \, d \sigma.
\end{cases}
\nonumber
\end{equation}
These statements are derived from the equation $\partial_{tt}(rw_{\lin})-\partial_{rr}(rw_{\lin})=0$ and Lemma \ref{Lem:HardyIneq}.

Hence, for $t\in \RR$, using $|a+b|^m+|a-b|^m\approx |a|^m+|b|^m,\quad (a,b)\in \RR^2$,
\begin{align*}
E_m(t)&=\int_0^{\infty}|\dot{f}(t+r)+\dot{f}(t-r)|^m+|\dot{f}(t+r)-\dot{f}(t-r)|^m\,dr\\
&\approx \int_0^{\infty}|\dot{f}(t+r)|^m+|\dot{f}(t-r)|^m\,dr\approx \int_{-\infty}^{+\infty}|\dot{f}(r)|^m\,dr,
\end{align*}
which is independent of $t$. This concludes the proof of (\ref{Eqn:Conserv}).
\end{proof}
\subsection{Bound from below of the exterior generalized energy}
\begin{lem}
Let $w_{\lin}$ be a solution of the linear wave equation, with initial data in $\HHH^{s_c}$ and $R>0$. Let
\begin{equation*}
E_{m,R}(t) := \int_{R + |t|}^{\infty} (\partial_{t}(r w_{\lin}))^{m} + ( \partial_{r}(r w_{\lin}) )^{m} \, dr.
\end{equation*}
Then
\begin{equation}
E_{m,R}(t) \gtrsim E_{m,R}(0), \; \text{ for all } t \geq 0 \, \text{ or for all } t \leq 0.
\label{Eqn:EstIR}
\end{equation}
\label{Lem:Nrjchannel}
\end{lem}
\begin{rem}
Lemma \ref{Lem:Nrjchannel} is already known in the energy-critical case $s_c=1$, $m=2$. More precisely, for finite energy solutions of the linear wave equation (see \cite{DuyKenMerSmall}):
\begin{equation*}
E_{2,R}(t) \geq \frac 12 E_{2,R}(0), \; \text{ for all } t \geq 0 \, \text{ or for all } t \leq 0.
\end{equation*}
\label{Rem:NrjchannelH1}
\end{rem}

\begin{proof}
With the notations of the previous proof, we have
\begin{align*}
E_{m,R}(t) & = \int_{R + |t|}^{\infty} |\dot{f}(t+r) + \dot{f}(t-r)|^{m} \, dr + \int_{R + |t|}^{\infty} |\dot{f}(t+r) - \dot{f}(t-r)|^{m} \, dr \\
&\approx \int_{R + |t|}^{\infty} |\dot{f}(t+r)|^{m} + |\dot{f}(t-r)|^{m} \, dr.
\end{align*}
Denoting by $\tE_{m,R}(r)$ the quantity  appearing at the last line of the previous inequality, we will prove
\begin{align}
\tE_{m,R}(t) \geq \frac 12 \tE_{m,R}(0), \; \text{ for all } t \geq 0 \, \text{ or for all } t \leq 0,
\label{Eqn:EstIR_bis}
\end{align}
which will conclude the proof of (\ref{Eqn:EstIR}).

If $t\geq 0$, we have
$$\tE_{m,R}(t)=\int_{-\infty}^{-R} |\dot{f}(r)|^m\,dr+\int_{R+2t}^{+\infty} |\dot{f}(r)|^m\,dr,$$
and if $t\leq 0$,
$$\tE_{m,R}(t)=\int_{R}^{+\infty} |\dot{f}(r)|^m\,dr+\int_{-\infty}^{2t-R} |\dot{f}(r)|^m\,dr.$$
Using that
$$\tE_{m,R}(0)=\int_{R}^{+\infty} |\dot{f}(r)|^m\,dr+\int_{-\infty}^{-R} |\dot{f}(r)|^m\,dr,$$
we deduce that if $\int_{-\infty}^{-R}|\dot{f}(r)|^m\,dr\geq \int_{R}^{+\infty}|\dot{f}(r)|^m\,dr$, then
$$\forall t\geq 0, \quad \tE_{m,R}(t)\geq \frac 12 \tE_{m,R}(0),$$
and if $\int_{-\infty}^{-R}|\dot{f}(r)|^m\,dr\leq \int_{R}^{+\infty}|\dot{f}(r)|^m\,dr$
then
$$\forall t\leq 0, \quad \tE_{m,R}(t)\geq \frac 12 \tE_{m,R}(0),$$
which proves (\ref{Eqn:EstIR_bis}) as announced.
\end{proof}
\subsection{Linear approximation for data with small generalized energy}

\begin{prop}
\label{Prop:CauchyPb}
Let $A>0$. There exists $\delta_2=\delta_2(A)$ with the following properties.
Let $w$ be a solution of (\ref{Eqn:WaveSup}) with data $\BA{w_0} \in \HHH^{s_c} $ such that $\|\BA{w_0} \|_{\HHH^{s_c}}\leq A$. Let
\begin{align*}
\delta:=  \| r^{1-\frac{2}{m}} \partial_{r} w_{0} \|_{L^{m}} + \| r^{1-\frac{2}{m}} w_{1} \|_{L^{m}} \cdot
\end{align*}
Assume that $\delta \leq \delta_2$. Then $w$ is global,
\begin{equation*}
\| w \|_{L^{4m}_t(I,L^{4m}_x)} \lesssim  \delta^{\frac{3}{8}} \| \BA{w_0} \|^{\frac{5}{8}}_{\HHH^{s_c} }
\end{equation*}
and
\begin{multline*}
\| r^{1-\frac{2}{m}} ( \partial_{r} w - \partial_{r} w_{\lin}) \|_{ L_{t}^{\infty}(\RR,L_{x}^{m})}+\| r^{1-\frac{2}{m}} ( \partial_{t} w - \partial_{t} w_{\lin}) \|_{ L_{t}^{\infty}(\RR, L_{x}^{m})}\\
 \lesssim \delta^{\frac{3m}{4}} \|
\BA{w_0} \|^{\frac{5m}{4}+1}_{\HHH^{s_c} }.
\end{multline*}
\end{prop}
\begin{proof}
Using H\"older's inequality, Lemma \ref{Lem:HardyIneq}, the Strichartz estimates (\ref{Eqn:Strich}), Remark \ref{Rem:EquivGenNrj}, and Lemma \ref{Lem:ConservQuantLin}, we have
\begin{multline*}
 \|w_{\lin}\|_{L^{4m}_t(\RR,L^{4m}_x)} \lesssim \left\| r^{1-\frac{2}{m}} \partial_r w_{\lin}\right\|^{\frac{3}{8}}_{L^{\infty}_t(\RR,L_x^{m})}\|w_{\lin}\|^{\frac{5}{8}}_{L^{\frac{5}{2}m}_t(\RR,L^{5 m}_x)}\\ \lesssim \delta^{\frac{3}{8}}\left\|\BA{w_0} \right\|^{\frac{5}{8}}_{\HHH^{s_c}} \lesssim \delta^{\frac{3}{8}}A^{\frac{5}{8}}.
\end{multline*}
By the Cauchy problem theory for equation (\ref{Eqn:WaveSup}) (see Proposition \ref{Prop:CauchyPbHsc}), the solution $w$ is global if $\delta^{\frac{3}{8}}A^{\frac{5}{8}}$ is small. (It could also be deduced from the argument below).

Next from Strichartz inequality (\ref{Eqn:Strich})  we see that for all interval $I\subset \RR$,
\begin{equation}
\begin{array}{ll}
\| D^{s_{c} - \frac{1}{2} } w \|_{L^{4}_t(I,L^{4}_x)}& \lesssim
\| \BA{w_{0}} \|_{\HHH^{s_c} } +
\| D^{s_{c} - \frac{1}{2}} w \|_{L_{t}^{4} (I,L^{4}_x)}
\| w \|^{p-1}_{L^{4m}_t(I,L^{4m}_x)} \\
\| w \|_{L^{4m}_t(I,L^{4m}_x)}  & \lesssim  \| w_{\lin} \|_{L^{4m}_t(I,L^{4m}_x)}
+ \| D^{s_{c} - \frac{1}{2}} w \|_{L^{4}_t(I,L^{4}_x)}
\| w \|^{p-1}_{L^{4m}_t(I,L^{4m}_x)},
\end{array}
\nonumber
\end{equation}
where we have used the chain rule for fractional derivatives
$$\left\|D^{s_c-\frac{1}{2}}(|w|^{p-1} w) \right\|_{L^{4/3}_tL^{4/3}_x}\lesssim \left\|D^{s_c-\frac{1}{2}}w\right\|_{L^4_tL^4_x}\|w\|^{p-1}_{L^{2(p-1)}_tL^{2(p-1)}_x}$$
(see \cite{KenPonVeg93}).  Combining these three estimates and using a bootstrap argument we see that
\begin{equation}
\begin{array}{ll}
\| D^{s_{c} - \frac{1}{2} } w \|_{L^4_t(\RR,L^4_x)} & \lesssim \| \BA{w_{0}} \|_{\HHH^{s_c}} \\
\| w \|_{L^{4m}_t(\RR,L^{4m}_x)} & \lesssim  \delta^{\frac{3}{8}} \| \BA{w_{0}} \|^{\frac{5}{8}}_{\HHH^{s_{c}}} \cdot
\end{array}
\nonumber
\end{equation}
Hence, using (\ref{Eqn:HardyPosOrigin1}), (\ref{Eqn:HardyDerOrigin}), Strichartz estimates (\ref{Eqn:Strich}) and the chain rule for fractional derivatives, we have
\begin{align*}
\left\| r^{1-\frac{2}{m}} (\partial_{r} w - \partial_{r} w_{\lin}) \right\|_{ L_{t}^{\infty} (\RR,L_{x}^{m})} &+ \left\| r^{1-\frac{2}{m}} (\partial_{t} w - \partial_{t} w_{\lin}) \right\|_{ L_{t}^{\infty} (\RR,L_{x}^{m} )} \\ &\lesssim  \| \vec{w} - \vec{w}_{\lin} \|_{L_{t}^{\infty}
(\RR,
\HHH^{s_c}) }  \\
& \lesssim \left\| D^{s_{c} - \frac{1}{2} } w \right\|_{L^{4}_t (\RR,L^4_x) } \left\| w \right\|^{p-1}_{L^{4m}_t (\RR,L^{4m}_x)}
\\
 & \lesssim  \delta^{\frac{3(p-1)}{8}}  \| \BA{w_{0}} \|^{\frac{5}{8}(p-1)+ 1}_{\HHH^{s_{c}}}  \cdot
\end{align*}
\end{proof}
\subsection{Localization of solutions}
We gather two localization properties of solutions of wave equations in the generalized energy norm.
The first one (on the linear equation) results from the strong Huygens principle.
\begin{prop}
Let $w_{\lin}$ be a solution of the linear wave equation with initial data $(w_{0},w_{1})\in \HHH^{s_c}$. Let
$ \{ \lambda_{n} \}_{n \in \mathbb{N}} $ and $ \{ t_{n} \}_{n \in \mathbb{N}}$ be two sequences of real numbers with $\lambda_n>0$ for all $n$.
Let
\begin{equation}
w_{\lin,n}(t,x) := \frac{1}{\lambda_{n}^{\frac{2}{p-1}}} w_{\lin} \left( \frac{t}{\lambda_{n}}, \frac{x}{ \lambda_{n}} \right)
\nonumber
\end{equation}
Assume that $\lim_{n \rightarrow \infty} \frac{t_{n}}{\lambda_{n}} = l \in [-\infty, \infty]$.\\
If $l = \pm \infty$ then
\begin{align}
\lim_{R \rightarrow \infty} \limsup_{n \rightarrow \infty} \int_{ | |x| - |t_{n}| | \geq R \lambda_{n}}
\left| r^{1-\frac{2}{m}} \partial_{r,t} w_{\lin,n} (t_{n})\right|^{m} + \left| r^{-\frac{2}{m}}w_{\lin,n}(t_{n})\right|^{m} \, dx =0
\label{Eqn:Huyg1}
\end{align}
and if $l \in \mathbb{R}$ then
\begin{align}
\lim_{R \rightarrow \infty} \limsup_{n \rightarrow \infty} \int_{ \left\{ |x|  \geq R \lambda_{n} \right\} \cup  \left\{ |x| \leq \frac{1}{R} \lambda_{n}
\right\} }
\left| r^{1-\frac{2}{m}} \partial_{r,t} w_{\lin,n} (t_{n})\right|^{m} + \left| r^{-\frac{2}{m}}w_{\lin,n}(t_{n})\right|^{m} \, dx =0  \cdot
\label{Eqn:Huyg2}
\end{align}
\label{Prop:Huyg}
\end{prop}
The proof of Proposition \ref{Prop:Huyg} is given in Appendix \ref{App:ProofStrongHuyg}.
The following proposition follows from the small data theory and finite speed of propagation:
\begin{prop}
\label{Prop:small_exterior}
Let $w$ be a solution of (\ref{Eqn:WaveSup}), global for positive times. Then
\begin{equation}
\lim_{R \rightarrow \infty} \limsup_{t \rightarrow \infty} \int_{t + R}^{\infty}  \left|r\partial_{r} w(t,r)\right|^{m}
+ \left|r\partial_{t} w(t,r)\right|^{m} \, dr= 0
\label{Eqn:DecayWeight}
\end{equation}
\end{prop}
\begin{proof}
Let $\epsilon > 0$. By Results \ref{Res:Cont_R} in the appendix, one can choose a $R \gg 1$ such that
\begin{equation*}
 \left\| ( \TTT_R w_{0},  \textbf{1}_{\RR^3\setminus B_R} w_{1}  ) \right\|_{\HHH^{s_c}} \leq \alpha,
\end{equation*}
with $\alpha \ll \min{(\delta_0, \epsilon)} $ and $\delta_0$ defined in Proposition \ref{Prop:CauchyPbHsc}. Let $\tilde{w}$ be the solution
of (\ref{Eqn:WaveSup}) with initial data  $( \TTT_R w_{0},  \textbf{1}_{\RR^3\setminus B_R} w_{1}  )$. By Proposition
\ref{Prop:CauchyPbHsc}  and Lemma \ref{Lem:HardyIneq}
\begin{equation*}
\left\| (r^{1-\frac{2}{m}} \partial_{r} \tilde{w}, r^{1-\frac{2}{m}} \partial_{t} \tilde{w}) \right\|_{
L_{t}^{\infty} \left(\mathbb{R}, L_{x}^{m}\times L_{x}^{m}\right)} \ll \epsilon.
\end{equation*}
From finite speed of propagation we see that $\tilde{w}(t,r)=w(t,r)$ if $r\geq R+|t|$ and thus
\begin{align*}
\int_{R + |t|}^{\infty}  \left|r\partial_{r} w(t,r)\right|^{m} + \left|r\partial_{t} w(t,r)\right|^{m} \, dr \leq \epsilon \cdot
\end{align*}
Hence (\ref{Eqn:DecayWeight}) holds.
\end{proof}

\subsection{Bound from below of the generalized energy for a sum of profile}
Proposition \ref{Prop:Orth} shows that the generalized energy of a sum of linear profiles (resp. nonlinear profiles) in an annulus is bounded asymptotically from below by the generalized energy of one of the linear profiles (resp. nonlinear profiles) in the same annulus.

\begin{prop}
\label{Prop:Orth}
Let $\{ (w_{0,n},w_{1,n}) \}_{n \in \mathbb{N}}$ be a bounded sequence in $ \HHH^{s_c}$
that has a profile decomposition with profiles $\{ W_{\lin}^{j} \}_{j\geq 1}$ and parameters
$ \left\{ \left\{ (t_{j,n}, \lambda_{j,n}) \right\}_{j\geq 1} \right\}_{n\in \NN} $. Let
$ \left\{ (\theta_{n},\rho_{n},\sigma_n)\right\}_{n \in \mathbb{N}}$ be a sequence such that $ 0 \leq \rho_{n}<\sigma_n\leq\infty$, $\theta_n\in \RR$. Let $k\geq 1$. Then, extracting a subsequence if necessary
\begin{multline}
\label{Eqn:Orth}
o_n(1)+\left\|r^{1-\frac{2}{m}} \partial_r w_{\lin,n}(\theta_n)\right\|_{L^m(\{\rho_n<|x|<\sigma_n\})}^m +\left\|r^{1-\frac{2}{m}} \partial_t w_{\lin,n}(\theta_n)\right\|_{L^m(\{\rho_n < |x|<\sigma_n\})}^m
\\
\geq \left\|r^{1-\frac{2}{m}} \partial_r W_{\lin,n}^k(\theta_n)\right\|_{L^m(\{\rho_n<|x|< \sigma_n\})}^m +\left\|r^{1-\frac{2}{m}} \partial_t W_{\lin,n}^k(\theta_n)\right\|_{L^m(\{\rho_n<|x|< \sigma_n\})}^m
\end{multline}
where $\lim_n o_n(1)=0$, $w_{\lin,n}$ is the solution of the linear wave equation with initial data $(w_{0,n},w_{1,n})$ and $W_{\lin,n}^k$ is defined in (\ref{Eqn:defWlnj}).

Moreover, denoting by $W^j$ the corresponding nonlinear profiles and assuming that all these profiles scatter, we have
\begin{multline}
\label{Eqn:Orth:NL}
o_n(1)+\left\|r^{1-\frac{2}{m}} \partial_r w_{n}(\theta_n)\right\|_{L^m(\{\rho_n<|x|<\sigma_n\})}^m +\left\|r^{1-\frac{2}{m}} \partial_t w_{n}(\theta_n)\right\|_{L^m(\{\rho_n < |x|<\sigma_n\})}^m
\\
\geq \left\|r^{1-\frac{2}{m}} \partial_r W_{n}^k(\theta_n)\right\|_{L^m(\{\rho_n<|x|< \sigma_n\})}^m +\left\|r^{1-\frac{2}{m}} \partial_t W_{n}^k(\theta_n)\right\|_{L^m(\{\rho_n<|x|< \sigma_n\})}^m,
\end{multline}
where $w_n$ is the solution of (\ref{Eqn:WaveSup}) with data $(w_{0,n},w_{1,n})$, and $W_n^k$ is defined in (\ref{Eqn:defWnj}).
\end{prop}
The proof of Proposition \ref{Prop:Orth} relies on the following two pseudo-orthogonality lemmas:
\begin{lem}
\label{Lem:pseudo_orth1}
Under the assumptions of Proposition \ref{Prop:Orth}, if $j\neq k$,
\begin{multline}
\label{Eqn:pseudo_orth1}
  \lim_{n\to\infty} \bigg(\int_{\rho_n}^{\sigma_n} |\partial_rW_{\lin,n}^k(0,r)|^{m-2} \partial_rW_{\lin,n}^k(0,r)
  \partial_rW_{\lin,n}^j(0,r)r^m\,dr\\
  +\int_{\rho_n}^{\sigma_n} |\partial_tW_{\lin,n}^k(0,r)|^{m-2} \partial_tW_{\lin,n}^k(0,r) \partial_tW_{\lin,n}^j(0,r)r^m\,dr\bigg)=0.
 \end{multline}
\end{lem}
\begin{lem}
\label{Lem:pseudo_orth2}
Under the assumptions of Proposition \ref{Prop:Orth}, if $J\geq k$,
\begin{multline}
\label{Eqn:pseudo_orth2}
  \lim_{n\to\infty} \bigg(\int_{\rho_n}^{\sigma_n} |\partial_rW_{\lin,n}^k(0,r)|^{m-2} \partial_rW_{\lin,n}^k(0,r)\partial_r\eps_{0,n}^J(r)r^m\,dr\\
  +\int_{\rho_n}^{\sigma_n} |\partial_tW_{\lin,n}^k(0,r)|^{m-2} \partial_tW_{\lin,n}^k(0,r)\eps_{1,n}^J(r)r^m\,dr\bigg)=0,
 \end{multline}
 where $(\eps_{0,n}^J,\eps_{1,n}^J)$ is as in \S \ref{SS:def_profiles}.
\end{lem}
We first prove Proposition \ref{Prop:Orth}, assuming the two lemmas.
\begin{proof}
\noindent\emph{Proof of (\ref{Eqn:Orth})}.
By the definition of the profile decomposition, we have
\begin{equation}
 \label{Eqn:exp_wlin}
 w_{\lin,n}(t,r)=\sum_{j=1}^k W_{\lin,n}^j(t,r)+\eps_n^k(t,r).
\end{equation}
Translating $w_{\lin,n}$ in time if necessary, we may assume $\theta_n=0$ for all $n$. By (\ref{Eqn:exp_wlin}) (and its time derivative), at $t=0$, we have
\begin{multline*}
\int_{\rho_n}^{\sigma_n}\partial_{r,t} w_{\lin,n}(0,r)\left|\partial_{r,t}W_{\lin,n}^k(0,r)\right|^{m-2} \partial_{r,t} W_{\lin,n}^k(0,r)r^m\,dr
\\
= \int_{\rho_n}^{\sigma_n} \left|\partial_{r,t}W_{\lin,n}^k(0,r)\right|^{m}r^m\,dr\\
+\sum_{j=1}^{k-1} \int_{\rho_n}^{\sigma_n} \left|\partial_{r,t}W_{\lin,n}^k(0,r)\right|^{m-2} \partial_{r,t}W_{\lin,n}^k(0,r) \partial_{r,t}W_{\lin,n}^j(0,r)r^m\,dr\\ 
+\int_{\rho_n}^{\sigma_n} \left|\partial_{r,t}W_{\lin,n}^k(0,r)\right|^{m-2} \partial_{r,t}W_{\lin,n}^k(0,r)\partial_{r,t}\eps_n^k(0,r)r^m\,dr,
\end{multline*}
Using Lemmas \ref{Lem:pseudo_orth1}, \ref{Lem:pseudo_orth2} and H\"older's inequality, we obtain
\begin{multline*}
o_n(1)+\left\|r^{1-\frac{2}{m}} \partial_{r,t} W_{\lin,n}^k(0)\right\|_{L^m(\{\rho_n<|x|< \sigma_n\})}^{m-1} \left\|r^{1-\frac{2}{m}} \partial_{r,t} w_{\lin,n}(0)\right\|_{L^m(\{\rho_n<|x|< \sigma_n\})}\\
\geq \left\|r^{1-\frac{2}{m}} \partial_{r,t} W_{\lin,n}^k(0)\right\|_{L^m(\{\rho_n<|x|< \sigma_n\})}^m,
\end{multline*}
and (\ref{Eqn:Orth}) follows from Young's inequality (\ref{Eqn:Young}).

\medskip

\noindent\emph{Proof of (\ref{Eqn:Orth:NL})}.
By Proposition \ref{Prop:Perturb_outside},
\begin{equation}
 \label{Eqn:E1}
 \BA{w_n}(\theta_n)\simP \sum_{j\geq 1}\BA{V_{\lin,n}^j}(0),
\end{equation}
 where
 $$V^j_{\lin,n}(t,x):=\frac{1}{\lambda_{j,n}^{\frac{2}{p-1}}}V^j_{\lin}\left(\frac{t-s_{j,n}}{\lambda_{j,n}},\frac{x}{\lambda_{j,n}}\right),\quad s_{j,n}:=t_{j,n}-\theta_n,$$
 and $V^j_{\lin}$ is as in Remark \ref{Rem:Perturb}.

By (\ref{Eqn:Orth}),
\begin{multline*}
o_n(1)+\left\|r^{1-\frac{2}{m}} \partial_r w_{n}(\theta_n)\right\|_{L^m(\{\rho_n<|x|<\sigma_n\})}^m +\left\|r^{1-\frac{2}{m}} \partial_t w_{n}(\theta_n)\right\|_{L^m(\{\rho_n < |x|<\sigma_n\})}^m
\\
\geq \left\|r^{1-\frac{2}{m}} \partial_r V_{\lin,n}^k(0)\right\|_{L^m(\{\rho_n<|x|< \sigma_n\})}^m +\left\|r^{1-\frac{2}{m}} \partial_t V_{\lin,n}^k(0)\right\|_{L^m(\{\rho_n<|x|< \sigma_n\})}^m.
\end{multline*}
Combining with the definition of $V^j_{\lin,n}$, we get the desired conclusion.
 \end{proof}
We are left with proving Lemmas \ref{Lem:pseudo_orth1} and \ref{Lem:pseudo_orth2}.
\begin{proof}[Proof of Lemma \ref{Lem:pseudo_orth1}]
\noindent\emph{Step 1. Preliminaries.}
Using Lemma \ref{Lem:HardyIneq}, Lemma \ref{Lem:ConservQuantLin}, and the density of  $C_0^{\infty}$ functions in $\dot{H}^{s_c}$ and $\dot{H}^{s_c-1}$, we see that it is sufficient to prove the lemma assuming
\begin{equation}
 \label{Eqn:compact_supp}
 W_{0}^j,\;W_1^j,\;W_{0}^k,\;W_1^k \in C_0^{\infty}.
\end{equation}
The explicit solutions of the linear wave equation in dimension $3$ are given by
\begin{equation}
\label{Eqn:Wj_Wk}
rW^{j}_{\lin}(t,r)=f^{j}(t+r)-f^{j}(t-r),\quad rW^{k}_{\lin}(t,r)=f^{k}(t+r)-f^{k}(t-r),
\end{equation}
where $f^j$ (respectively $f^k$) is defined as $f$ after (\ref{Eqn:wf}), with $w=W^j_{\lin}$ (respectively $w=W^k_{\lin}$). It follows from the definition of $f^k$ and $f^j$ that $\dot{f}^k$ and $\dot{f}^j$ are compactly supported.
By the strong Huygens' principle and (\ref{Eqn:compact_supp}), there exists a constant $K>0$ (depending on $W^j_{\lin}$, $W^k_{\lin}$) such that
\begin{equation}
 \label{Eqn:supp_Wjln}
 \forall l\in \{j,k\},\quad
 r\in \supp \BA{W_{\lin,n}^{l}}(0)\Longrightarrow \Big| r-|t_{l,n}|\Big|\leq K\lambda_{l,n}.
\end{equation}
Using (\ref{Eqn:compact_supp}) and (\ref{Eqn:Wj_Wk}), it is easy to see that the following bound holds:
\begin{equation}
\label{Eqn:bound_Wj_Wk}
 \forall t\in \RR,\;\forall r>0,\;\forall l\in \{j,k\}\quad \left|W^{l}_{\lin}(t,r)\right|+\left|\partial_{r,t}W^{l}_{\lin}(t,r)\right| \lesssim \frac{1}{1+|t|+|r|}.
\end{equation}
We conclude this step by noting that we can assume $\lambda_{k,n}=1$.
Indeed, the change of variable $s=r/\lambda_{k,n}$ gives
\begin{multline*}
\int_{\rho_n}^{\sigma_n} \left|\partial_{r,t}W_{\lin,n}^k(0,r)\right|^{m-2} \partial_{r,t}W_{\lin,n}^k(0,r) \partial_{r,t}W_{\lin,n}^j(0,r)r^m\,dr\\
=\int_{\rho_n'}^{\sigma_n'} \left|\partial_{r,t}W^k_{\lin}(-t_{k,n}',s)\right|^{m-2} \partial_{r,t}W^k_{\lin}(-t_{k,n}',s)
\partial_{r,t} W^j_{\lin}\left(-\frac{t_{j,n}'}{\lambda_{j,n}'},\frac{s}{\lambda_{j,n}'}\right)\frac{s^m}{{\lambda_{j,n}'}^{1+\frac{1}{m}}}\,ds,
\end{multline*}
where
$$\lambda_{j,n}'=\frac{\lambda_{j,n}}{\lambda_{k,n}},\; \rho_n'=\frac{\rho_n}{\lambda_{k,n}},\;\sigma_n'=\frac{\sigma_n}{\lambda_{k,n}},\;t_{j,n}'=\frac{t_{j,n}}{\lambda_{k,n}},\; t_{k,n}'=\frac{t_{k,n}}{\lambda_{k,n}},$$
and these new parameters satisfy the usual orthogonality condition (with $\lambda_{k,n}'=1$).

We next prove the conclusion of Lemma \ref{Lem:pseudo_orth1}, assuming (\ref{Eqn:compact_supp}) and $\lambda_{k,n}=1$. Extracting subsequences in $n$ if necessary, we can assume that $\lambda_{j,n}$ has a limit in $[0,+\infty]$ as $n\to\infty$. We distinguish two cases: when the limit is $0$ or $+\infty$ (treated in Step 2), and when the limit is in $(0,+\infty)$ (treated in Step 3).

\medskip

\noindent \emph{Step 2.} In this step we prove the desired result when
\begin{equation}
\label{Eqn:lim_lambda1}
 \forall n,\; \lambda_{k,n}=1\text{ and }\lim_{n\to\infty} \lambda_{j,n}\in \{0,+\infty\}
\end{equation}
We denote by $S_n$ the set
\begin{equation}
 \label{def_Eqn}
 S_n :=\left\{r>0\; \text{s.t.}\; \big|r-|t_{k,n}|\big|\leq K\text{ and }\big|r-|t_{j,n}|\big|\leq K\lambda_{j,n}\right\},
\end{equation}
By H\"older's inequality and (\ref{Eqn:supp_Wjln}),
\begin{multline*}
\left|\int_{\rho_n}^{\sigma_n} |\partial_{r,t}W_{\lin,n}^k(0,r)|^{m-2}  \partial_{r,t}W_{\lin,n}^k (0,r)\partial_{r,t}W_{\lin,n}^j(0,r)r^m\,dr\right|
\\
\leq \underbrace{\left(\int_{S_n} |\partial_{r,t}W_{\lin,n}^k(0,r)|^mr^m\,dr\right)^{\frac{m-1}{m}}}_{I_n}\underbrace{\left(\int_{S_n} |\partial_{r,t}W_{\lin,n}^j(0,r)|^m r^m\,dr\right)^{\frac{1}{m}}}_{II_n}.
\end{multline*}
We note that by Lemma \ref{Lem:HardyIneq} and Lemma \ref{Lem:ConservQuantLin}, both terms $I_n$ and $II_n$ are bounded. We distinguish two subcases.

If $\lambda_{j,n}\to 0$, using that the length of $S_n$ is smaller that $2K\lambda_{j,n}$, and the bound $|\partial_{r,t}W^k_{\lin}|\lesssim 1/r$ in (\ref{Eqn:bound_Wj_Wk}), we obtain
$$\big|I_n\big|\lesssim_K \lambda_{j,n}^{\frac{m-1}{m}}\underset{n\to\infty}{\longrightarrow} 0.$$

If $\lambda_{j,n}\to+\infty$, using that the length of $S_n$ is smaller that $K$ and the bound of $|\partial_{r,t}W^j|$ in (\ref{Eqn:bound_Wj_Wk}), which implies $|\partial_{r,t}W^j_{\lin,n}| \lesssim r^{-1}\lambda_{j,n}^{-1/m}$, we obtain
$$ \big|II_n\big| \lesssim \left(\int_{S_n} \frac{1}{r^m\lambda_{j,n}}r^m\,dr\right)^{1/m}\lesssim \frac{1}{\lambda_{j,n}^{1/m}}\underset{n\to\infty}{\longrightarrow} 0.$$

In both cases, the conclusion of the lemma follows.

\medskip

\noindent \emph{Step 3.} In this step we prove the desired result when
\begin{equation}
\label{Eqn:lim_lambda2}
 \forall n,\; \lambda_{k,n}=1\text{ and }\lim_{n\to\infty} \lambda_{j,n}\in (0,\infty).
\end{equation}
Rescaling if necessary, we can assume without loss of generality that the limit of $\lambda_{j,n}$ is $1$. In this case
$$ \left(\frac{1}{\lambda_{j,n}^{\frac{2}{p-1}}}W^j_{\lin,n}\left(0,\frac{\cdot}{\lambda_{j,n}}\right),
\frac{1}{\lambda_{j,n}^{1+\frac{2}{p-1}}}\partial_tW^j_{\lin,n}\left(0,\frac{\cdot}{\lambda_{j,n}}\right)\right)\underset{n \to\infty}{\longrightarrow} (W^j_{0},W^j_1),$$
in $\HHH^{s_c}$, so that we can assume furthermore:
$$ \forall n,\quad \lambda_{j,n}=1.$$

Finally, we must prove that the following sum of $2$ integrals goes to $0$ as $n$ goes to infinity:
$$ J_n:=\int_{\rho_n}^{\sigma_n}\left|\partial_{r,t}W^k_{\lin}(-t_{k,n},r)\right|^{m-2} \partial_{r,t} W^k_{\lin}(-t_{k,n},r)\partial_{r,t}W^j_{\lin}(-t_{j,n},r)r^m\,dr.$$
If
$$ \lim_{n\to\infty} \Big| |t_{k,n}|-|t_{j,n}|\Big|=+\infty,$$
then $S_n$ is empty for large $n$ and the conclusion of the lemma follows. Extracting subsequences, we are reduced to the case where
$|t_{k,n}|-|t_{j,n}|$ has a finite limit as $n\to\infty$. In view of the orthogonality condition (\ref{Eqn:PseudoOrth}), and since $\lambda_{j,n}=\lambda_{k,n}=1$, we deduce that
\begin{equation*}
\lim_{n\to\infty}t_{k,n}+t_{j,n}=\theta_0\in\RR, \quad \lim_{n\to\infty}\left|t_{k,n}-t_{j,n}\right|=\infty.
\end{equation*}
To fix ideas, we will assume
\begin{equation}
\label{lim_tjn_tkn}
 \lim_{n\to\infty} t_{k,n}=+\infty,\quad \lim_{n\to\infty} t_{j,n}=-\infty.
\end{equation}
(The proof is the same if these limits are exchanged). We next prove
\begin{multline}
\label{Eqn:reduc_r}
 \int_{\rho_n}^{\sigma_n}|\partial_{r,t}W^k_{\lin}(-t_{k,n},r)|^{m-2} \partial_{r,t}W^k_{\lin}(-t_{k,n},r)\partial_{r,t}W^j_{\lin}(-t_{j,n},r)r^m\,dr=\\
 \int_{\rho_n}^{\sigma_n}\left|\partial_{r,t}\left(rW^k_{\lin} \right)(-t_{k,n},r) \right|^{m-2}
 \partial_{r,t}\left(rW^k_{\lin} \right)(-t_{k,n},r) \partial_{r,t} (rW^j_{\lin} )(-t_{j,n},r) \,dr+o_n(1).
\end{multline}
Since $\supp W^k_{\lin}(-t_{k,n})\subset \left[|t_{k,n}|-K,|t_{k,n}|+K \right]$, and, by (\ref{Eqn:bound_Wj_Wk}),  $|W^k_{\lin}(-t_{k,n},r)|\lesssim 1/|t_{k,n}|$, we have
\begin{equation*}
 \lim_{n\to\infty}\int_0^{+\infty}\left|W^k_{\lin}(-t_{k,n},r)\right|^m\,dr+\int_0^{+\infty} \left|W^j_{\lin}(-t_{j,n},r)\right|^m\,dr=0,
\end{equation*}
and (\ref{Eqn:reduc_r}) follows, using that $\partial_{r,t}(rW^l(-t_{k,n},r))=r\partial_{r,t}W^l(-t_{k,n},r)+W^l(-t_{k,n},r)$, $l=j,k$.

Using (\ref{lim_tjn_tkn}), the formula (\ref{Eqn:Wj_Wk}) and the fact that $\dot{f}^j$ and $\dot{f}^k$ are compactly supported, we see that for large $n$, for any $r>0$,
$$\partial_r(rW^{k}_{\lin})(-t_{k,n},r)=\dot{f}^k(-t_{k,n}+r),\quad \partial_r(rW^{j}_{\lin})(-t_{j,n},r)=\dot{f}^{j}(-t_{j,n}-r)$$
and
$$\partial_t(rW^{k}_{\lin})(-t_{k,n},r)=\dot{f}^k(-t_{k,n}+r),\quad \partial_t(rW^{j}_{\lin})(-t_{j,n},r)=-\dot{f}^{j}(-t_{j,n}-r).$$
As a consequence, the term with $\partial_r$ and the term with $\partial_t$ in the second line of (\ref{Eqn:reduc_r}) cancel each other, and we obtain
\begin{equation*}
\lim_{n\to\infty}
\int_{\rho_n}^{\sigma_n}|\partial_{r,t}W^k_{\lin}(-t_{k,n},r)|^{m-2} \partial_{r,t}W^k_{\lin}(-t_{k,n},r)\partial_{r,t}W^j_{\lin}(-t_{j,n},r)r^m\,dr=0,
\end{equation*}
which concludes the proof of Lemma \ref{Lem:pseudo_orth1}.
\end{proof}
\begin{proof}[Proof of Lemma \ref{Lem:pseudo_orth2}]

\noindent\emph{Step 1. Preliminaries.}
As in the proof of Lemma \ref{Lem:pseudo_orth1}, we first note that it is sufficient to prove the lemma assuming
\begin{equation*}
 W_{0}^k,\;W_1^k \in C_0^{\infty}.
\end{equation*}
We will use the notations of the preceding proof: we recall from this proof that $rW^k_{\lin}(t,r)=f^k(t+r)-f^k(t-r)$, the condition (\ref{Eqn:supp_Wjln}) on the support of $\vW_{\lin,n}^k(0)$ and the estimate (\ref{Eqn:bound_Wj_Wk}) (with $l=k$). As in this proof, we see (using the change of variable $r=s/\lambda_{k,n}$) that we can assume $\lambda_{k,n}=1$ for all $n$. In this case, $\eps_n^J$ satisfies, in view of (\ref{Eqn:WeakProfile2}),
\begin{equation}
 \label{Eqn:weak_CV_eps}
 \BA{\eps_n^J}(t_{k,n}) \xrightharpoonup[n\to\infty]{} 0\text{ in }\HHH^{s_c}.
\end{equation}
Assuming, after extraction, that $\{t_{k,n}\}_n$ has a limit  $\theta_0 \in[-\infty,+\infty]$ as $n\to\infty$, we will distinguish the case where $\theta_0\in \RR$ (treated in step 2) and the case where $\theta_0\in \{-\infty,+\infty\}$ (treated in step 3).

\medskip

\noindent\emph{Step 2.} In this step we prove (\ref{Eqn:pseudo_orth2}) assuming:
$$\lim_{n\to\infty} t_{k,n}=\theta_0\in \RR.$$
Translating in time, we can assume $\theta_0=0$. Using the continuity of the linear flow in $\HHH^{s_{c}}$ and Lemma \ref{Lem:HardyIneq}, we can
even assume $t_{k,n}=0$ for all $n$. Thus we are reduced to prove that
\begin{equation}
 \label{Eqn:goto0_step2}
 \lim_{n\to\infty} \int_{\rho_n}^{\sigma_n} |\partial_{r,t}W^k_{\lin}(0,r)|^{m-2} \partial_{r,t}W^k_{\lin}(0,r)\partial_{r,t}\eps_n^J(0,r)r^m\,dr=0.
\end{equation}
If $\rho_n\to+\infty$ or $\sigma_n\to 0$ as $n\to\infty$, (\ref{Eqn:goto0_step2}) is obvious. Thus we can assume (extracting subsequences) $\rho_n\to \rho_{\infty}\in [0,+\infty)$, $\sigma_n\to\sigma_{\infty}\in (0,+\infty]$ as $n\to\infty$, where $\rho_{\infty}\leq \sigma_{\infty}$. In this case
\begin{multline*}
\int_{\rho_n}^{\sigma_n} |\partial_{r,t}W^k_{\lin}(0,r)|^{m-2} \partial_{r,t}W^k_{\lin}(0,r)\partial_{r,t}\eps_n^J(0,r)r^m\,dr\\
=\int_{\rho_{\infty}}^{\sigma_{\infty}} |\partial_{r,t}W^k_{\lin}(0,r)|^{m-2} \partial_{r,t}W^k_{\lin}(0,r)\partial_{r,t}\eps_n^J(0,r)r^m\,dr+o_n(1),
\end{multline*}
and (\ref{Eqn:goto0_step2}) follows from (\ref{Eqn:weak_CV_eps}) (with $t_{k,n}=0$), since by Lemma \ref{Lem:HardyIneq},
$$f\mapsto \int_{\rho_{\infty}}^{\sigma_{\infty}} |\partial_{r,t}W^k_{\lin}(0,r)|^{m-2} \partial_{r,t}W^k_{\lin}(0,r)f(r)r^m\,dr$$
is a continuous linear form on $\dot{H}^{s_c-1}$.

\medskip

\noindent\emph{Step 3.} In this step we prove (\ref{Eqn:pseudo_orth2}) assuming:
$$\lim_{n\to\infty} t_{k,n}\in  \{-\infty,+\infty\}.$$
To fix ideas, we assume
\begin{equation}
\label{Eqn:lim_tkn_infty}
 \lim_{n\to\infty} t_{k,n}=+\infty.
\end{equation}
Using, as in the proof of Lemma \ref{Lem:pseudo_orth1}, that (\ref{Eqn:lim_tkn_infty}) implies
$$\lim_{n\to\infty} \int_0^{+\infty}\left|W^k_{\lin}(-t_{k,n},r)\right|^m\,dr=0,$$
we obtain
\begin{multline}
\label{Eqn:step3a}
\int_{\rho_n}^{\sigma_n} |\partial_{r,t}W^k_{\lin}(0,r)|^{m-2} \partial_{r,t}W^k_{\lin}(0,r)\partial_{r,t}\eps_n^J(0,r)r^m\,dr\\=\int_{\rho_n}^{\sigma_n} \left|\partial_{r,t} (rW^k_{\lin} )(0,r)\right|^{m-2} \partial_{r,t}\left(rW^k_{\lin} \right) (0,r) \partial_{r,t}\eps_n^J(0,r)r\,dr+o_n(1)
\end{multline}
Furthermore, by the radial Sobolev embedding,
$$|\eps_n^J(0,r)| \lesssim \frac{1}{r^{1/m}}\left\|\eps_n^J(0)\right\|_{\dot{H}^{s_c}}.$$
Combining with the condition (\ref{Eqn:supp_Wjln}) on the support of $\BA{W^k_{\lin,n}}(0)$, we obtain
$$\int_0^{+\infty} \left|\partial_{t,r}\left(rW^k_{\lin}(-t_{k,n},r)\right)\right|^{m-1}\left|\eps_n^J(0,r)\right|\,dr\lesssim \frac{\|\eps_n^J(0)\|_{\dot{H}^{s_c}}}{(|t_{k,n}|-K)^{\frac{1}{m}}}\underset{n\to\infty}{\longrightarrow} 0,
$$
in view of (\ref{Eqn:lim_tkn_infty}). Combining with (\ref{Eqn:step3a}), we deduce
\begin{multline}
\label{Eqn:step3b}
\int_{\rho_n}^{\sigma_n} |\partial_{r,t}W^k_{\lin}(0,r)|^{m-2} \partial_{r,t}W^k_{\lin}(0,r)\partial_{r,t}\eps_n^J(0,r)r^m\,dr\\=\int_{\rho_n}^{\sigma_n} \left|\partial_{r,t}\left(rW^k_{\lin} \right)(0,r)\right|^{m-2} \partial_{r,t}\left(rW^k_{\lin} \right)(0,r)
\partial_{r,t}\left(r\eps_n^J \right)(0,r)\,dr+o_n(1)
\end{multline}
We write, as in (\ref{Eqn:wf})
$$r\eps_n^J(t,r)=g_{n}^J(t+r)-g_{n}^J(t-r),$$
where $g_{n}^J$ is defined in a similar way as $f$ after (\ref{Eqn:wf}). Hence $g_{n}^{J} \in L_{loc}^{m}(\mathbb{R})$ and $\dot{g}_{n}^{J} \in
L^{m}(\mathbb{R})$.

Since $\dot{f}^k$ has compact support, we deduce from (\ref{Eqn:lim_tkn_infty}) that $\dot{f}^k(-t_{k,n}-r)=0$ for large $n$ and all $r>0$, and we can rewrite (\ref{Eqn:step3b}) as
\begin{multline*}
\int_{\rho_n}^{\sigma_n} |\partial_{r,t}W^k_{\lin}(0,r)|^{m-2} \partial_{r,t}W^k_{\lin}(0,r)\partial_{r,t}\eps_n^J(0,r)r^m\,dr\\=\int_{\rho_n}^{\sigma_n} \left|\dot{f}^k(-t_{k,n}+r)\right|^{m-2}\dot{f}^k(-t_{k,n}+r)\left(\dot{g}^J_n(r)+\dot{g}^J_n(-r)\right)\,dr
\\+
\int_{\rho_n}^{\sigma_n} \left|\dot{f}^k(-t_{k,n}+r)\right|^{m-2} \dot{f}^k(-t_{k,n}+r)\left(\dot{g}^J_n(r)-\dot{g}^J_n(-r)\right)\,dr+
o_n(1)\\
=2\int_{\rho_n}^{\sigma_n} \left|\dot{f}^k(-t_{k,n}+r)\right|^{m-2} \dot{f}^k(-t_{k,n}+r)\dot{g}^J_n(r)\,dr+o_n(1)\\
=2\int_{\rho_n-t_{k,n}}^{\sigma_n-t_{k,n}} \left|\dot{f}^k(r)\right|^{m-2} \dot{f}^k(r)\dot{g}^J_n(r+t_{k,n})\,dr.
\end{multline*}
If $\rho_n-t_{k,n}\to +\infty$ or $\sigma_n-t_{k,n}\to-\infty$ as $n\to+\infty$, then we are done. Assume
$$\lim_{n\to\infty}\rho_n-t_{k,n}=R_{\infty}\in [-\infty,+\infty),\quad \lim_{n\to\infty}\sigma_n-t_{k,n}=S_{\infty}\in (-\infty,+\infty],$$
with $R_{\infty}\leq S_{\infty}$.
Then
\begin{multline*}
 \int_{\rho_n}^{\sigma_n} |\partial_{r,t}W^k_{\lin}(0,r)|^{m-2}  \partial_{r,t}W^k_{\lin}(0,r)\partial_{r,t}\eps_n^J(0,r)r^m\,dr\\
=2\int_{R_{\infty}}^{S_{\infty}} \left|\dot{f}^k(r)\right|^{m-2} \dot{f}^k(r)\dot{g}^J_n(r+t_{k,n})\,dr+o_n(1).
\end{multline*}
Since, for $r\geq 0$,
\begin{align*}
\dot{g}^J_n(r+t_{k,n})&=\frac 12 \left(\partial_r(r\eps_n^J(t_{k,n},r))+\partial_t(r\eps_n^J(t_{k,n},r))\right), \\
\dot{g}^J_n(-r+t_{k,n})&=\frac 12 \left(\partial_r(r\eps_n^J(t_{k,n},r))-\partial_t(r\eps_n^J(t_{k,n},r))\right),
\end{align*}
we deduce from (\ref{Eqn:HardyPosOrigin1}), (\ref{Eqn:HardyDerOrigin}), and (\ref{Eqn:weak_CV_eps}) that
$$\dot{g}^J_n(t_{k,n}+\cdot)\xrightharpoonup[n\to\infty]{}0\text{ in } L^{m}(\mathbb{R}),$$
which concludes the proof, since
$$ h\mapsto 2 \int_{R_{\infty}}^{S_{\infty}} \left|\dot{f}^k(r)\right|^{m-2} \dot{f}^k(r)
h (r)\, dr$$   is a continuous linear form on $L^{m}(\mathbb{R})$.

\end{proof}

\section{Channels of energy for nonzero solutions}
\label{Sec:channels}
\subsection{Results}
In this section we consider all nonzero solutions of (\ref{Eqn:WaveSup}). We will prove below Propositions \ref{Prop:NonzeroSol} and \ref{Prop:NonzeroSol2} that state roughly that these solutions have a dispersive behavior
in the sense of the exterior energy estimate (\ref{Eqn:ConcentrationEta}). Recall from Proposition \ref{Prop:SolStat} the definitions of $Z_{\ell}$ and $R_{\ell}$. If $(f,g)\in \HHH^{s_c}$, and $\ell\neq 0$, we let
$$\sigma_{\ell}(f,g)=\min\Big\{R>R_{\ell}\;:\; \int_{R}^{\infty}|f(r)-Z_{\ell}(r)|+|g(r)|\,dr=0 \Big\}\in (R_{\ell},\infty].$$
If $\int_{R}^{\infty}|f(r)-Z_{\ell}(r)|+|g(r)|\,dr$ is not $0$ for large $R$, we just let $\sigma_{\ell}(f,g)=\infty$.
The fact that $\sigma_{\ell}(f,g)$ is greater than $R_{\ell}$ if $(f,g)\in \HHH^{s_c}$ follows from Proposition \ref{Prop:SolStat}.

For technical reasons, we will distinguish between the case where $\sigma_{\ell}(w_0,w_1)$ is infinite for all $\ell \neq 0$ (Proposition \ref{Prop:NonzeroSol}) and where it is finite for some $\ell$ (Proposition \ref{Prop:NonzeroSol2}).
\begin{prop}
\label{Prop:NonzeroSol}
Let $w$ be a solution of (\ref{Eqn:WaveSup}) with initial data $\BA{w_0} \in \HHH^{s_c}$ such that for all $\ell\neq 0$, $\sigma_{\ell}(\BA{w_0})=\infty$. Then there exist $R > 0$, $\eta > 0$, and a global, scattering solution $\cw$ of (\ref{Eqn:WaveSup}), with initial data $\BA{\cw_0} \in \HHH^{s_c}$, such that
\begin{equation}
\BA{\cw_{0}}(r)= \BA{w_{0}}(r), \, r \geq R
\label{Eqn:egal_init_data}
\end{equation}
and the following holds for all $t \geq 0$ or for all $t \leq 0$
\begin{align}
\int_{|x| \geq R + |t|} | r\partial_{r} \cw(t,r)|^{m} + | r \partial_{t} \cw(t,r)|^{m}\, dr \geq \eta.
\label{Eqn:ConcentrationEta}
\end{align}
\end{prop}
\begin{prop}
\label{Prop:NonzeroSol2}
Let $w$ be a solution of (\ref{Eqn:WaveSup}) with initial data $\BA{w_0}$ such that $\sigma_{1}(\BA{w_0})$ is finite. Then
\begin{enumerate}
 \item \label{I:rhol}There exists $\theta= \theta(\sigma_{1}(\BA{w_{0}}))>0$, a solution $\tilde{w}$ of (\ref{Eqn:WaveSup}), defined on $[-\theta,\theta]$, with initial data $\BA{\tilde{w}_0} \in \HHH^{s_c}$, and $R \in (0,\sigma_{1}(\BA{w_0}))$ close enough to $\sigma_{1}(\BA{w_{0}})$
     such that
\begin{equation}
\BA{\tilde{w}}_0(r)=\BA{w_0}(r)\text{ if } r>R,
\label{Eqn:brwEquality}
\end{equation}
the following holds for all $t\in [0,\theta]$ or for all $t\in [-\theta,0]$
\begin{equation}
 \sigma_{1}(\BA{\tilde{w}}(t))=\sigma_{1}(\BA{w_0})+|t|.
\label{Eqn:PropSupp}
\end{equation}
Moreover
\begin{equation}
\label{Eqn:bound_below}
\forall R'>R_1,\quad \inf_{\sigma_{1} \geq R'} \theta(\sigma_{1})>0.
\end{equation}
\item \label{I:channel} There exists $S_1>R_1$ such that if $\sigma_{1}(\BA{w_0}) > S_1$, and $R\in (0, \sigma_{1}(\BA{w_0}))$ is close enough to $\sigma_1(\BA{w_0})$, there exists a global, scattering solution $\cw$ of (\ref{Eqn:WaveSup}) such that (\ref{Eqn:egal_init_data}) and (\ref{Eqn:ConcentrationEta}) hold.
\end{enumerate}
\end{prop}
\begin{rem}
Proposition \ref{Prop:NonzeroSol2} is stated for solutions such that $\sigma_{1}(\BA{w_0})$ is finite. However, in view of Remark \ref{R:Zell},   one obtains from Proposition \ref{Prop:NonzeroSol2} a similar conclusion for solutions such that $\sigma_{\ell}(\BA{w_0})$ is finite, for any $\ell\neq 0$.
\end{rem}
The proof of Propositions \ref{Prop:NonzeroSol} and \ref{Prop:NonzeroSol2} is given in Subsections \ref{Subsec:ExtNrjFirstCase}
and \ref{Subsec:ExtNrjSecondCase}.\\
We next state a profile version of Proposition \ref{Prop:NonzeroSol} and Proposition \ref{Prop:NonzeroSol2} (\ref{I:channel}), i.e
\begin{prop}
Consider a profile decomposition $\left\{W^j_{\lin},\left\{ (\lambda_{j,n},t_{j,n}) \right\}_{n\in \mathbb{N}}\right\}_{j\geq 1}$ as in Section \ref{Sec:ProfDecomp}. Assume that there exists $j \geq 1$ such that $W_{\lin}^{j} \neq 0$. Assume furthermore
$$\lim_{n\to\infty}\frac{-t_{j,n}}{\lambda_{j,n}}\in \{-\infty,+\infty\},$$
or $\sigma_1\left(\vW^j_{\lin}(0)\right)>S_1$, where $S_1$ is given by Proposition \ref{Prop:NonzeroSol2}.
Then there
exists a solution $\tW_{\lin}^{j}$ of the linear wave equation, and a sequence $\{ \rho_{j,n} \}_{n \in \mathbb{N}}$ of
positive numbers such that the nonlinear profile $\tW^{j}$ associated to $\tW_{\lin}^{j}$ and $\left\{ (\lambda_{j,n},t_{j,n}) \right\}_{n \in \mathbb{N}}$ scatters as $t \rightarrow \pm \infty$,
\begin{equation}
|x| > \rho_{j,n} \Longrightarrow \BA{\tW_{\lin,n}^{j}} (0,x) = \BA{W_{\lin,n}^{j}}(0,x),
\label{Eqn:DispPropScaling1}
\end{equation}
and there exists $\eta_{j} > 0$ such that the following holds for all $t \geq 0$ or for all $t \leq 0$
\begin{align}
\int_{\rho_{j,n} + |t| }^{\infty} \left| r \partial_{r} \tW_{n}^{j}(t,x) \right|^{m}
+ \left| r \partial_{t} \tW_{n}^{j}(t,x) \right|^{m} \, dr \geq \eta_{j}
\label{Eqn:DispPropScaling2}
\end{align}
\label{Prop:DispPropScaling}
\end{prop}
In the statement of Proposition \ref{Prop:DispPropScaling}, the modulated linear profile $W_{\lin,n}^j$ is defined as usual by (\ref{Eqn:defWlnj}). The modulated linear and nonlinear profiles $\tW_{\lin,n}^j$ and $\tW_{n}^j$ are defined the same way, replacing $W^j_{\lin}$ by $\tW^j_{\lin}$ and $\tW^j$ respectively.

The proof of Proposition \ref{Prop:DispPropScaling} is given in Subsection \ref{Subsec:ProofProfVers}.

\subsection{Proof of the exterior energy property in the first case} \label{Subsec:ExtNrjFirstCase}
We give here the proof of Proposition \ref{Prop:NonzeroSol}. The proof is close to the proof of \cite[Proposition 2.1]{DuyKenMerClass}, with the generalized energy replacing the energy.

 Let $w$ be a nonzero solution of (\ref{Eqn:WaveSup}) with data $\BA{w_0}:= (w_0,w_1) \in \HHH^{s_c}$ such that for all $\ell\neq 0$, $\sigma_{\ell}(\BA{w_0})=\infty$.  We prove Proposition \ref{Prop:NonzeroSol} by contradiction, assuming that for all scattering solution $\tilde{w}$ of (\ref{Eqn:WaveSup}) and all $R>0$ if $\BA{\tilde{w}}(0,r)=\BA{w_0}(r)$ for $r>R$, then
\begin{equation}
 \label{channel_zero}
\liminf_{t\to\pm \infty}\int_{R+|t|}^{\infty}|r\partial_r\tilde{w}(t,r)|^m+|r\partial_t\tilde{w}(t,r)|^m\,dr=0.
\end{equation}
We first prove:
\begin{res}
Let $A:=\|\BA{w_{0}}\|_{\HHH^{s_c}}$. There exists $\delta_3=\delta_3(A)>0$, such that, for any $r_0>0$, if
\begin{equation}
 \label{Eqn:AssSecru0}
\int_{r_0}^{\infty} |r\partial_r w_0|^m+|rw_1|^m\,dr=\delta\leq \delta_3,
\end{equation}
then, letting $(v_0,v_1):=r \BA{w_0}$, we have
\begin{align}
\int_{r_{0}}^{\infty} |\partial_{r} v_{0}|^{m} + |v_{1}|^{m} \, dr \lesssim_A \frac{|v_{0}(r_{0})|^{\frac{3}{4}m^2}}{r_{0}^{\frac{3}{4}m(m-1)}}\cdot
\label{Eqn:DecayKinet}
\end{align}
Furthermore, for all $ r_{0} \leq r \leq r' \leq 2 r_{0}$
\begin{align}
|v_{0}(r) - v_{0}(r') | \lesssim_A   \frac{|v_{0}(r)|^{\frac{3}{4}m}}{r_0^{(m-1)\left(\frac{3}{4}-\frac{1}{m}\right)}} \lesssim_A \delta^{\frac{3}{4}-\frac{1}{m}} |v_{0}(r)| \cdot
\label{Eqn:HoldCont}
\end{align}
\label{Lem:w0v0}
\end{res}

\begin{proof}
We first see that (\ref{Eqn:HoldCont}) can be easily derived from (\ref{Eqn:DecayKinet}) and (\ref{Eqn:AssSecru0}). Indeed,
\begin{align*}
|v_{0}(r) - v_{0}(r')| & \lesssim (r' -r)^{\frac{m-1}{m}} \left( \int_{r}^{r'} \left|\partial_{\bar{r}} v_{0}\right|^{m} \, d \bar{r} \right)^{\frac{1}{m}} \\
& \lesssim_A r_0^{\frac{m-1}{m}-\frac{3}{4}(m-1)} |v_{0}(r)|^{\frac{3}{4}m},
\end{align*}
which yields the first inequality in (\ref{Eqn:HoldCont}). Furthermore, using Lemma \ref{Lem:UtoV} and (\ref{Eqn:AssSecru0}), we obtain $r_0^{1-m}|v_0(r)|^m\leq \delta$, which yields the second inequality in (\ref{Eqn:HoldCont}).

We next prove (\ref{Eqn:DecayKinet}). We let
$$\BA{\brw_0}:=(\brw_0,\brw_1) :=(\TTT_{r_0}w_0,\textbf{1}_{\RR^3\setminus B_{r_0}} w_1),$$
so that
\begin{equation}
\label{Eqn:tildew_delta}
\int_0^{+\infty} |r\partial_r \brw_0|^m+|r \brw_1|^m\,dr=\int_{r_0}^{\infty} |r\partial_r w_0|^m+|rw_1|^m\,dr=\delta,
\end{equation}
and, by Results \ref{Res:BoundCharac} and \ref{Res:BoundOp} in the appendix.
$\left\| \BA{\brw_0} \right\|_{\HHH^{s_c}} \lesssim_{A} 1$.

We denote by $\brw_{\lin}$ the solution of the linear wave equation with initial data $(\brw_0,\brw_1)$, and let $(\brv_0,\brv_1)=r(\brw_0,\brw_1)$ and $\brv_{\lin} :=r\brw_{\lin}$. We note that $w(t,r)=\brw(t,r)$ if $r\geq r_0+|t|$ by finite speed of propagation.

By Lemma \ref{Lem:Nrjchannel}, the following holds for all $t\geq 0$ or for all $t\leq 0$:
\begin{multline*}
\int_{r_0}^{\infty} |\partial_r v_0|^m+|v_1|^m\,dr=\int_{r_0}^{\infty} |\partial_r \brv_0|^m+ |\brv_1|^m\,dr\\
\lesssim \int_{r_0+|t|}^{\infty} |\partial_r \brv_{\lin}(t,r)|^m + |\partial_t \brv_{\lin}(t,r)|^m\,dr.
\end{multline*}
But then, using Lemma \ref{Lem:UtoV}, Proposition \ref{Prop:CauchyPb}, and (\ref{Eqn:tildew_delta}), we deduce
\begin{align*}
\int_{r_0}^{\infty} |\partial_rv_0|^m+|v_1|^m\,dr &\lesssim \int_{r_0+|t|}^{\infty}|r\partial_r \brw_{\lin}|^m+|r\partial_t \brw_{\lin}|^m\,dr\\
& \lesssim_A \int_{r_0+|t|}^{\infty}|r \partial_r \brw|^m+|r\partial_t \brw|^m\,dr\\&\qquad + \left(\int_{r_0}^{+\infty} |r\partial_r w_0|^m+|rw_1|^m\,dr\right)^{\frac{3}{4}m}.
\end{align*}
Letting $t\to +\infty$ or $t\to-\infty$ and using (\ref{channel_zero}), we deduce
$$\int_{r_0}^{\infty} |\partial_rv_0|^m+|v_1|^m\,dr \lesssim_A \left(\int_{r_0}^{\infty} |r\partial_r w_0|^m+|rw_1|^m\,dr\right)^{\frac{3}{4}m}.$$
By Lemma \ref{Lem:UtoV},
\begin{equation}
\label{Eqn:last}
\int_{r_0}^{\infty} |\partial_rv_0|^m+|v_1|^m\,dr \lesssim_A \left(r_0^{1-m}|v_0 (r_{0})|^m+\int_{r_0}^{\infty} |\partial_rv_0|^m+|v_1|^m\,dr\right)^{\frac{3}{4}m}.
\end{equation}
By (\ref{Eqn:AssSecru0}) and Lemma \ref{Lem:UtoV} again, we have
$$r_0^{1-m}|v_0(r_0)|^m+\int_{r_0}^{\infty} |\partial_r v_0|^m+|v_1|^m\,dr \lesssim \delta_3.$$
Using that $\delta_3$ is small, we see that (\ref{Eqn:last}) implies (\ref{Eqn:DecayKinet}).
\end{proof}

\begin{res}
We have
\begin{align}
\label{Eqn:finite_energy}
\BA{w_0} \in \HHH^1
\end{align}
and
\begin{align}
\label{Eqn:bound_1/r}
\forall r>0, \; |v_0(r)|\lesssim_A 1\quad\text{and} \quad\exists \ell>0,\quad \forall r\geq 1\quad \left|v_0(r)-\ell\right|\lesssim_A \frac{1}{r^{\alpha}},
\end{align}
where $\alpha := (m-1)\left(\frac{3}{4}-\frac 1m\right)>\frac 14$.
Moreover, if $\barR>0$ and $\barw$ is a global solution with initial data $\BA{\barw_0}(r):=(\barw_0,\barw_1) \in \HHH^1\cap\HHH^{s_c}$ such that $\BA{\barw_0}(r)=\BA{w_0}(r)$ for $r\geq \barR$, then
\begin{equation}
\label{Eqn:ChannelH1}
\lim_{t \rightarrow \pm \infty} \int_{|x| \geq \barR + |t|}  | \partial_{r} \barw(t,x) |^{2} + | \partial_{t} \barw (t,x) |^{2} \, dx=0.
\end{equation}
\label{Res:VoLimit}
\end{res}
\begin{proof}
\noindent\emph{Step 1.} We prove that for all $\epsilon>0$,
\begin{equation}
 \label{Eqn:Growthvor}
 \forall r\geq 1, \quad
 |v_0(r)|\lesssim_{A,\epsilon}r^{\epsilon}.
\end{equation}
We may assume without loss of generality that $r \gg 1$.
Choose $r_{1} \gg 1$ such that (\ref{Eqn:AssSecru0}) holds with $r_0 :=r_1 $ and a small $ 0 < \delta:= \delta(A) \ll 1$.  Let $n \in \mathbb{N}$. From (\ref{Eqn:HoldCont}) with $r_0=2^nr_1$, we see that there exists $ 0 < \delta':= \delta'(A) \ll 1$ such that $ | v_{0}(2^{n+1}r_1) | \leq (1 + \delta')  | v_{0}(2^{n} r_1) | $.
By a straightforward induction,
\begin{align*}
| v_{0}(2^{n} r_{1}) | \leq (1 + \delta')^{n} |v_{0}(r_{1})| \cdot
\end{align*}
Let $r \gg 1$. Choosing $n$ such that $2^{n} r_{1} \leq r \leq 2^{n+1} r_{1}$ and $\delta$ small enough, and using (\ref{Eqn:HoldCont}) with $r_0:=2^nr_1$, we obtain (\ref{Eqn:Growthvor}).

\medskip
\noindent\emph{Step 2. Proof of (\ref{Eqn:bound_1/r})}. Again we may assume without loss of generality that  $r \gg 1$.
We fix  $r_1 \gg 1$. Combining (\ref{Eqn:HoldCont}) and Step 1, we obtain that for all $n$ in $\mathbb{N}$,
$$ |v_0(2^nr_1)-v_0(2^{n+1}r_1)|\lesssim_{A,\epsilon}(2^nr_1)^{\frac{3}{4}m\epsilon-(m-1)(\frac{3}{4}-\frac{1}{m})} \lesssim_{A,\epsilon} \frac{1}{(2^nr_1)^{\epsilon'}},$$
where $\epsilon'$ is a positive constant if $\epsilon$ is small enough. As a consequence,
$$\sum_{n}|v_0(2^nr_1)-v_0(2^{n+1}r_1)|\lesssim_A 1$$
which shows that $v_0(2^nr_1)$ has a finite limit $\ell$ as $n\to\infty$. Using again (\ref{Eqn:HoldCont}) and Step 1, we deduce that $v_0(r)\to\ell$ as $r\to\infty$. The bound $|v_0(r)|\lesssim_A 1$ follows immediately. Going back to  (\ref{Eqn:HoldCont}), we obtain for large $r_1$
$$|v_0(2^nr_1)-v_0(2^{n+1}r_1)| \lesssim_A (2^nr_1)^{-\alpha}.$$
Summing up over $n\in \NN$, we deduce
$$|v_0(r_1)-\ell| \lesssim_A \frac{1}{r_1^{\alpha}},$$
concluding the proof of (\ref{Eqn:bound_1/r}).

\medskip

\noindent\emph{Step 3. Proof of (\ref{Eqn:finite_energy}).}
Next we prove that $ \BA{w_{0}} \in \HHH^1$. We see from H\"older's inequality,
(\ref{Eqn:DecayKinet}) and (\ref{Eqn:bound_1/r}) that
\begin{align*}
\int_{2^{n} r_{1}}^{ 2^{n+1} r_{1}}  |\partial_{r} v_{0}|^{2} + |v_{1}|^{2} \, dr &  \lesssim  (2^{n} r_{1})^{\frac{m-2}{m}}
\left( \int_{2^{n} r_{1}}^{ 2^{n+1} r_{1}}  |\partial_{r} v_{0}|^{m} + |v_{1}|^{m} \, dr \right)^{\frac{2}{m}} \\
& \lesssim_A (2^nr_1)^{\frac{m-2}{m}-\frac 32(m-1)} \cdot
\end{align*}
Hence, noting that the exponent $\frac{m-2}{m}-\frac 32(m-1)$ is negative (since $m>2$) and summing over $n$, we see that
\begin{equation*}
\int_{r_{1}}^{\infty}  |\partial_{r} v_{0}|^{2} + |v_{1}|^{2} \, dr < \infty.
\end{equation*}
We also have
\begin{equation*}
\int_{0}^{r_{1}} |\partial_{r} v_{0}|^{2} + |v_{1}|^{2} \, dr \lesssim r_1^{\frac{m-2}{m}} \left(\int_{0}^{r_{1}}  |\partial_{r} v_{0}|^{m} + |v_{1}|^{m} \, dr\right)^{\frac{2}{m}}  <\infty.
\end{equation*}
Hence $( v_{0}, v_{1}) \in \HHH^1(\RR)$, and we see from Lemma \ref{Lem:UtoV} that
$\BA{w_{0}} \in \HHH^1$.

\medskip

\noindent\emph{Step 4. Proof of (\ref{Eqn:ChannelH1}).}

Let  $0 < \epsilon \ll 1$. Choose $r_{1} \gg \bar{R}$ such that
\begin{equation*}
\int_{|x| \geq r_{1}} |\partial_{r} w_{0}|^{2} + |w_{1}|^{2} \, dx  \leq \epsilon.
\end{equation*}
Then, let $(\brw_0,\brw_1) :=(\TTT_{r_{1}} w_{0}, \mathbf{1}_{\mathbb{R}^{3} \setminus B_{r_{1}}} w_{1})$, and $\check{w}$ be the solution of
\begin{equation*}
\left\{
\begin{aligned}
\partial_{tt} \brw - \triangle \brw = |\chi_{r_{1}} \brw|^{p-1} \chi_{r_{1}} \brw \\
\BA{\brw}(0) := (\brw_0,\brw_1)
\end{aligned}
\right.
\end{equation*}
given by Proposition \ref{Prop:CauchyPbH1}. By the conclusion of this proposition,
\begin{align}
\sup_{t \in \mathbb{R}} \left\| \overrightarrow{ \brw}(t)  - \overrightarrow{S(t) \BA{\brw} (0)}(t)
\right\|_{\HHH^1} \lesssim \frac{1}{r_1^{\frac{p-5}{2}}} \| \BA{\brw}(0) \|^{p}_{\HHH^1}\cdot
\label{Eqn:Approxtildew}
\end{align}
Hence (using also finite speed of propagation)
\begin{equation}
\label{Eqn:bound_extH1_w}
\sup_{T_-(w)<t<T_+(w)} \int_{|x| \geq r_{1} + |t|} |\partial_{r} w(t,x)|^{2} + | \partial_{t} w(t,x) |^{2} \, dx  \lesssim \epsilon.
\end{equation}
From (\ref{Eqn:HardyPosOrigin1}) we see that
\begin{equation}
\begin{array}{l}
\int_{ |x| \leq r_{1} + |t| } | \partial_{r} w(t,x) |^{2} +  | \partial_{t} w(t,x) |^{2} \, dx \\
\lesssim (r_{1} + |t|)^{\frac{m-2}{m}} \left( \int_{|x| \leq r_{1} + |t|} \left|r\partial_{r} w (t,x)\right|^{m} + \left|r\partial_{t} w
(t,x)\right|^{m} \, dx \right)^{\frac{2}{m}} \\
< \infty \cdot
\end{array}
\nonumber
\end{equation}
Hence $ \vec{w}(t) \in \HHH^1$ for all $t$ in the domain of existence of $w$.

Next we prove (\ref{Eqn:ChannelH1}). The case $ t \rightarrow -\infty $ is a straightforward modification of the
case $ t \rightarrow \infty $ and is left to the reader.

Let $\barw$ that satisfies the assumptions of Result \ref{Res:VoLimit}. We first note that the proof of (\ref{Eqn:bound_extH1_w}) yields (fixing a small $\epsilon>0$ and choosing $r_1 \gg \max{(\barR,1)}$),
\begin{equation}
\label{Eqn:bound_extH1_w_bis}
\sup_{t\in \RR} \int_{|x| \geq r_{1} + |t|}  | \partial_{r} \barw(t,x) |^{2} + | \partial_{t} \barw(t,x) |^{2} \, dx  \lesssim \epsilon.
\end{equation}
We have, for $t\geq 0$
\begin{multline}
\int_{ \barR + |t|}^{ r_{1} +|t|} \left| r \partial_{r} \barw(t,r) \right|^{2} +  \left| r \partial_{t} \barw(t,r) \right|^{2} \, dr \\
\lesssim (r_{1} - \barR)^{\frac{m-2}{m}} \left( \int_{\barR + |t|}^{\infty} \left| r \partial_{r} \barw (t,r)\right|^{m} +
\left|r \partial_{t} \barw (t,r)\right|^{m} \, dr \right)^{\frac{2}{m}}.
\nonumber
\end{multline}
Hence, combining the estimates above with (\ref{channel_zero}), we see that
\begin{align*}
\liminf_{t\to+\infty} \int_{|x| \geq \barR + |t|} | \partial_{r}  \barw (t,x) | ^{2} + | \partial_{t} \barw(t,x) |^{2} \, dx  \lesssim \epsilon.
\end{align*}
Hence
\begin{equation*}
\liminf_{t\rightarrow \infty} \int_{|x| \geq \barR + |t|}  | \partial_{r} \barw(t,x)|^{2} + | \partial_{t} \barw (t,x)|^{2} \, dx
=0.
\end{equation*}
It remains to replace the lower limit by a limit to obtain (\ref{Eqn:ChannelH1}).
Defining $\brw$ as the solution of
\begin{equation*}
\left\{
\begin{array}{l}
\partial_{tt} \brw - \triangle \brw = |\chi_{\barR +\brt} \brw|^{p-1} \chi_{\barR + \brt} \brw \\
\BA{\brw}(0): = (\TTT_{\barR+ \brt} \barw(\brt), \mathbf{1}_{\mathbb{R}^{3} \setminus
B_{\barR + \brt}} \barw(\brt)),
\end{array}
\right.
\end{equation*}
where $\brt$ is such that
$$\int_{|x| \geq \barR + |\brt|}  | \partial_{r} \barw(\brt,x) |^{2} +  | \partial_{t} \barw (\brt,x) |^{2} \, dx\leq \epsilon$$
we see again from Proposition \ref{Prop:CauchyPbH1} that (\ref{Eqn:Approxtildew}) holds for $|t| \geq |\brt| \gg 1  $
and, by finite speed of propagation,
\begin{equation}
\forall t \geq \brt,\quad \int_{|x| \geq \barR + |t|} | \partial_{r}  \barw (t,r) | ^{2} + | \partial_{t} \barw(t,r) |^{2} \, dx  \lesssim \epsilon.
\end{equation}
\label{end_of_result}
Hence (\ref{Eqn:ChannelH1}) holds.
\end{proof}
We are now in position to conclude the proof of Proposition \ref{Prop:NonzeroSol}. We distinguish between two cases.

\medskip

\noindent\emph{Case 1: $\ell=0$.}
We first prove that $\BA{w_0}$ is compactly supported. By (\ref{Eqn:bound_1/r}), since $\ell=0$,
\begin{equation}
\label{bound_above_v}
|v_0(r)|\lesssim_A \frac{1}{r^{\alpha}},\quad r \geq 1.
\end{equation}
Fix $\epsilon>0$ so small that $1-\epsilon>\frac{1}{2^{\alpha}}$. By (\ref{Eqn:HoldCont}), for large $r_0$,
\begin{equation*}
\forall n\geq 0,\quad  |v_0(2^{n+1}r_0)|\geq (1-\epsilon)|v_0(2^nr_0)|.
\end{equation*}
An easy induction gives $|v_0(2^nr_0)|\geq (1-\epsilon)^{n}|v_0(r_0)|$, contradicting (\ref{bound_above_v}) unless $v_0(r_0)=0$. This proves that $v_0$ (thus $w_0$) has compact support. By (\ref{Eqn:DecayKinet}) in  Result \ref{Lem:w0v0} we deduce that we can assume (changing $w_1$ on a set of measure $0$ if necessary), that $w_1$ is  a compactly supported function.

Since by our assumption $w\neq 0$, we can choose $r_0>0$ close to the boundary of the support of $(w_0,w_1)$, so that
\begin{equation}
\label{Eqn:bound_int}
\int_{2r_0}^{\infty} |r\partial_rw_0|^m+|rw_1|^m\,dr=0,\quad 0<\int_{r_0}^{\infty} |r\partial_rw_0|^m+|rw_1|^m\,dr\leq \delta_3
\end{equation}
i.e. (\ref{Eqn:AssSecru0}) is satisfied. If $r_0\leq r\leq 2r_0$, we have, by (\ref{Eqn:HoldCont}),
$$|v_0(r)|=|v_0(r)-v_0(2r_0)| \lesssim_A \frac{1}{r_0^{\alpha}}|v_0(r)|^{\frac{3}{4}m}.$$
Hence, for $r\in [r_0,2r_0]$,
$$v_0(r)=0\quad \text{or}\quad |v_0(r)|^{\frac{3}{4}m-1} \gtrsim_A r_0^{\alpha}.$$
By continuity of $v_0$,  $v_{0}(r_{0})=0$. By (\ref{Eqn:DecayKinet}),
$v_0=0$ for $r>r_0$, and thus $w_0=0$ for $r>r_0$, contradicting (\ref{Eqn:bound_int}) and concluding the proof in the case $\ell=0$.

\medskip

\noindent\emph{Case 2: $\ell\neq 0$.}
In this case one can prove that $\sigma_{\ell}(\BA{w_0})$ is finite, contradicting our assumptions. We omit the proof: in view of (\ref{Eqn:ChannelH1}), it is exactly as Step 1 of the proof of Lemma 3.11 in \cite{DuyKenMerScattSuper}, using Proposition \ref{Prop:CauchyPbH1} with the potential given by Remark \ref{R:potentiel} (\ref{I:V2}), and the $\HHH^1$ exterior energy estimate given by Remark \ref{Rem:NrjchannelH1}.

%
%
%
%
%
%
%
%

\qed

\subsection{Proof of the exterior energy property in the second case} \label{Subsec:ExtNrjSecondCase}
We prove here Proposition \ref{Prop:NonzeroSol2}.
\begin{proof}[Proof of (\ref{I:rhol})]
Let $\BA{w_0} \in\HHH^{s_c}$ such that $\sigma:=\sigma_1(\BA{w_0})$ is finite. If $0<R<\sigma$, we denote by $\cw_R$ the solution of (\ref{Eqn:WaveSup}) with initial data
$$\cw_{0,R}=\TTT_Rw_0,\quad \cw_{1,R}=\mathbf{1}_{\RR^3\setminus B_R}w_1.$$
In this step, we prove that there exists $\rho$ such that $\frac{R_1+\sigma}{2}<\rho<\sigma$ and
\begin{gather}
 \label{Eqn:C2}
 \forall R\in [\rho,\sigma),\quad \left\|(\cw_{0,R}-z_{\rho}(0),\cw_{1,R})\right\|_{\HHH^1}\leq \delta_1\rho^{\frac{p-5}{2(p-1)}}\\
 \label{Eqn:C3}
 \forall R\in [\rho,\sigma),\quad \left[\frac{-\theta_{\sigma}}{2},\frac{\theta_{\sigma}}{2}\right]\subset \left(T_-(\cw_R),T_+(\cw_R)\right)\cap \left(T_-(z_R),T_+(z_R)\right),
 \end{gather}
where $\delta_1$ is given by Proposition \ref{Prop:CauchyPbH1}, $R_1$ by Proposition \ref{Prop:SolStat} and $z_{\rho}$ and $\theta_{\sigma}$ by Lemma \ref{Lem:zR}. Note that
\begin{multline*}
 \left\|(\cw_{0,R}-z_{\rho}(0),\cw_{1,R})\right\|_{\HHH^1}^2\\
 =\int_{\rho\leq |x|\leq R} \left|\nabla Z_1(0)\right|^2\,dx+\int_{R\leq |x|\leq \sigma} \left|\nabla w_0-\nabla Z_1\right|^2\,dx+\int_{R\leq |x|\leq \sigma} |w_1|^2\,dx,
\end{multline*}
and (\ref{Eqn:C2}) follows if $\rho$ is close enough to $\sigma$.

Since $(-\theta_{\sigma},\theta_{\sigma})$ is in the domain of existence of $z_{\sigma}$, the following limits will imply (\ref{Eqn:C3}) with $\rho$ close to $\sigma$ by standard perturbation theory:
\begin{align}
 \label{Eqn:C5}
 \lim_{R\overset{<}{\to}\sigma}\left\|(\cw_{0,R}-z_{\sigma}(0),\cw_{1,R}-\partial_t z_{\sigma}(0)\right\|_{\HHH^{s_c}}
 &=0\\
\label{Eqn:C5'}
\lim_{R\overset{<}{\to}\sigma}\left\|\vz_R(0)-\vz_{\sigma}(0)\right\|_{\HHH^{s_c}}&=0.
\end{align}

The limit (\ref{Eqn:C5'}) follows from Result \ref{Res:Cont_R} in the appendix.

To obtain (\ref{Eqn:C5}), notice that $\partial_tz_{\sigma}(0)=0$ and
\begin{equation*}
 \cw_{0,R}-z_{\sigma}(0)=\TTT_R(w_0)-\TTT_{\sigma}Z_1=(\TTT_R-\TTT_{\sigma})w_0+\TTT_{\sigma}(w_0-Z_1)=(\TTT_R-\TTT_{\sigma})w_0,
\end{equation*}
since by the definition of $\sigma$, $\TTT_{\sigma}(w_0-Z_1)=0$. Thus we are reduced to prove
\begin{equation*}
 \lim_{R\overset{<}{\to}\sigma}\left\|(\TTT_R-\TTT_{\sigma})w_0\right\|_{\dot{H}^{s_c}}+\left\|\mathbf{1}_{B_{\sigma}\setminus B_R}w_1\right\|_{\dot{H}^{s_c-1}}=0,
\end{equation*}
which follows again from Result \ref{Res:Cont_R} in the appendix. Step 1 is complete.

We will prove in the two remaining steps that $\cw=\cw_{\rho}$ satisfies the conclusion of Proposition \ref{Prop:NonzeroSol2} (\ref{I:rhol}). The sequel of the proof is close to Step 1 and Step 2 of the proof of the corresponding result for $p=5$ (Proposition 2.2 (a) in \cite{DuyKenMerClass}) with additional technicalities due to the fact that the nonlinearity is supercritical.

\medskip

\noindent\emph{Step 2. Linearization around $z_{\rho}$.}

Let $R \in [ \rho, \sigma)$.

By (\ref{Eqn:C2}) and Remark \ref{R:potentiel}, there exists $\theta'>0$ depending only on $\rho$\footnote{Since $\rho$ depends only on $\sigma$, $\theta'$ depends only on $\sigma$} such that the solution
$h_{R}$ of
\begin{equation}
\left\{
\begin{aligned}
\label{Eqn:C7}
\partial_{tt} h_{R} - \triangle h_{R}  & = | \TTT_{\rho}Z_1+\chi_{R} h_{R}|^{p-1} (\TTT_{\rho}Z_1+\chi_{R} h_{R})-|\TTT_{\rho}Z_1|^{p-1}\TTT_{\rho}Z_1 \\
\BA{h_{R}}(0)& := (\cw_{0,R}-z_{\rho}(0),\cw_{1,R})
\end{aligned}\right.
\end{equation}
satisfies the assumptions of Proposition \ref{Prop:CauchyPbH1} with $I=[-\theta',\theta']$.\\
By Proposition \ref{Prop:CauchyPbH1}, denoting by $h_{\lin,R}$ the solution of the linear wave equation with initial data $(\cw_{0,R}-z_{\rho}(0),\cw_{1,R})$, we have
\begin{equation}
 \label{Eqn:C8}
 \sup_{-\theta' \leq t \leq \theta'} \left\|\BA{h_{R}}(t)- \BA{h_{\lin,R}}(t)\right\|_{\HHH^1}\leq \frac{1}{100}\left\|\BA{h_{R}}(0)\right\|_{\HHH^1}.
\end{equation}
We let $$\theta :=\min\left(\frac{\theta_{\sigma}}{2},\theta'\right).$$
Letting $\tilde{w}:=\cw_{R}$, we see that (\ref{Eqn:brwEquality}) and (\ref{Eqn:bound_below}) are satisfied.\\
We also define $\ch_{R}:= \cw_{R} -z_{\rho}$. Then $\ch_R$ is solution of the equation
\begin{equation}
\label{Eqn:C6}
\left\{
\begin{aligned}
 \partial_{tt}\ch_{R}-\triangle \ch_{R} &=\left| z_{\rho}+\ch_{R} \right|^{p-1}\left( z_{\rho}+\ch_{R} \right)-|z_{\rho}|^{p-1}z_{\rho},\\
 \vec{\ch}_{R}(0) & =(\cw_{0,R}-z_{\rho}(0),\cw_{1,R}).
\end{aligned}
\right.
\end{equation}
By finite speed of propagation 
\begin{equation}
 \label{Eqn:C9}
\left\{
\begin{aligned}
\ch_{R}(t,r)&= 0 \quad &&\text{if} \quad r\geq \sigma+|t|, \quad \text{and} \\
\ch_{R}(t,r) &= h_{R}(t,r) \quad &&\text{if} \quad r\geq R+ |t|, \quad -\theta\leq t\leq \theta.
\end{aligned}\right.
\end{equation}

\medskip

\noindent\emph{Step 3. Propagation of the support.}
Let $R' \in [R, \sigma)$. By (\ref{Eqn:UtoVprec}) and Remark \ref{Rem:NrjchannelH1}, the following holds for all $t\in [0,\theta]$ or for all $t\in [-\theta,0]$:
\begin{multline}
 \label{Eqn:C10}
 \int_{|x|\geq R'+|t|}|\nabla h_{\lin,R}(t,x)|^2+(\partial_t h_{\lin,R}(t,x))^2\,dx \\
 \geq \int_{R'+|t|}^{+\infty} \big(\partial_r (rh_{\lin,R}(t,r))\big)^2+\big(\partial_t(r h_{\lin,R}(t,r))\big)^2\,dr\\
 \geq \frac{1}{2}\int_{R'}^{+\infty} \left(\partial_r (r\ch_{R}(0,r))\right)^2+\left(\partial_t (r\ch_{R}(0,r))\right)^2\,dr\\
 =\frac{1}{2}\int_{|x|>R'} \left|\nabla \ch_{R}(0,x)\right|^2+(\partial_t \ch_{R}(0,x))^2\,dx-\frac{R'}{2}\left(\ch_{R}(0,R')\right)^2.
\end{multline}
Since
\begin{multline}
\label{Eqn:C10'}
\left|\ch_{R}(0,R')\right|=\left|\int_{R'}^{\sigma} \partial_r \ch_{R}(0,r)\,dr\right|\leq \sqrt{(\sigma-R')\int_{R'}^{\sigma} \left( \partial_r \ch_{R}(0,r) \right)^2\,dr}\\
\leq \frac{\sqrt{\sigma-R'}}{R'}\sqrt{\int_{R'}^{\sigma} \left(\partial_r \ch_{R}(0,r)\right)^2r^2\,dr},
\end{multline}
we see that (\ref{Eqn:C10}) implies, if $R'$ is close enough to $\sigma$, that the following holds for all $t \in [0,\theta]$ or for all $t\in [-\theta,0]$:
\begin{multline*}
 \int_{|x| \geq R' +|t|}|\nabla h_{\lin,R}(t,x)|^2+(\partial_t h_{\lin,R}(t,x))^2\,dx\\
 \geq \frac 14 \int_{|x|\geq R'}|\nabla \ch_{R}(0,x)|^2+(\partial_t \ch_{R}(0,x))^2\,dx.
\end{multline*}
Thus, by (\ref{Eqn:C8}), the following holds for all $t\in [0,\theta]$ or for all $t\in [-\theta,0]$:
\begin{multline}
 \label{Eqn:C11}
 \int_{|x|\geq R'+|t|}|\nabla h_{R}(t,x)|^2+(\partial_t h_{R}(t,x))^2\,dx\\
 \geq \frac 18 \int_{|x|\geq R'}|\nabla \ch_{R}(0,x)|^2+(\partial_t \ch_{R}(0,x))^2\,dx.
\end{multline}
Note that
$$|x|>R' \Longrightarrow \BA{\ch_R}(0,x)=\left( w_0(x)-Z_1(x),w_1(x) \right).$$
Since $R'<\sigma_1(\BA{w_0})$, the right-hand side of (\ref{Eqn:C11}) is positive.

By (\ref{Eqn:C9}) we can replace, in the left-hand side of (\ref{Eqn:C11}), $h_{R}$ by $\ch_{R}$.  This implies
$$\sigma_1\big( \BA{\cw_{R}}(t) \big)\geq R' +|t|\text{ for all }t\in [0,\theta]\text{ or for all }t\in [-\theta,0].$$
Finally, since $R'$ can be taken arbitrarily close to $\sigma$, we obtain (\ref{Eqn:PropSupp}).
\end{proof}

\begin{proof}[Proof of (\ref{I:channel})]
We next prove the second point of Proposition \ref{Prop:NonzeroSol2}. The proof is inspired by the proof of the analogous energy-critical result, (b) of Proposition 2.2 in \cite{DuyKenMerClass}.

By Result \ref{Res:Cont_R} in the appendix, we can choose a large constant $S_1>0$ such that $S_1$ satisfies (\ref{I:V2})
in Remark \ref{R:potentiel} and
\begin{equation}
 \label{Eqn:D1}
 \forall R> S_1,\;
 \left\|\TTT_{R}Z_1\right\|_{\dot{H}^{s_c}} < \frac{\delta_0}{2C_0}, \text{ and }\left\|\TTT_{R}Z_1\right\|_{\dot{H}^{1}} <
 \frac{\delta_1}{2} \sqrt{R}^{\frac{p-5}{p-1}},
\end{equation}
where $\delta_0$, $C_0$ are given by the small data theory in $\HHH^{s_c}$ (Proposition \ref{Prop:CauchyPbHsc}), and $\delta_1$ is given by Proposition \ref{Prop:CauchyPbH1}.

\medskip

\noindent\emph{Step 1.}
Let $\BA{w_0}:=(w_0,w_1)\in \HHH^{s_c}$ such that $S_1<\sigma_1(\BA{w_0})<\infty$. In particular, $\BA{w_0} \in \HHH^1$. Let $R$ be such that $S_1<R<\sigma_1(\BA{w_0})$,
\begin{gather}
 \label{Eqn:D2}
 \left\|\left(\TTT_Rw_0,\mathbf{1}_{\RR^3\setminus B_R}w_1\right)\right\|_{\HHH^{s_c}} < \frac{\delta_{0}}{C_{0}}, \, and
 \\
 \label{Eqn:D3}
 \left\|\left( \TTT_R(w_0-Z_1),\mathbf{1}_{\RR^3\setminus B_R}w_1 \right)\right\|_{\HHH^1}  <
 \delta_1\sqrt{R}^{\frac{p-5}{p-1}}.
\end{gather}
Note that (\ref{Eqn:D2}) and (\ref{Eqn:D3}) are always possible for $R$ close to $\sigma_1(\BA{w_0})$ in view of the first inequality in (\ref{Eqn:D1}) and Result \ref{Res:Cont_R} in the appendix.

Let $V$ be given by (\ref{I:V2}) in Remark \ref{R:potentiel}.
Let $(h_0,h_1)=\BA{w_0}-(\TTT_{S_1}Z_1,0)$, $\BA{g_0}:=(g_0,g_1):=\left(\TTT_Rh_0,\mathbf{1}_{\RR^3\setminus B_R}h_1\right)$ and $g$ the solution of
$$\partial_{tt}g-\triangle g=|V+g|^{p-1}(V+g)-|g|^{p-1}g,\quad \vec{g}_{\restriction t=0}=\BA{g_0}$$
given by Proposition \ref{Prop:CauchyPbH1}. Recall that by (\ref{Eqn:D3}), $g$ is globally defined and
\begin{equation}
 \label{Eqn:D4}
 \sup_{t\in \RR} \left\|\vec{g}_{\lin}(t)-\vec{g}(t)\right\|_{\HHH^1}\leq \frac{1}{100}\|\BA{g_0}\|_{\HHH^1},
\end{equation}
where $g_{\lin}(t) := S(t)\BA{g_0}$. By Lemma \ref{Lem:UtoV} and Remark \ref{Rem:NrjchannelH1}, the following holds for all $t\geq 0$ or for all $t\leq 0$:
\begin{multline*}
 \int_{|x|\geq R+|t|}\left| \nabla g_{\lin}(t,x)\right|^2 + \left| \partial_tg_{\lin}(t,x) \right|^2\,dx\geq \int_{R+|t|}^{+\infty} \left| \partial_r(rg_{\lin})(t,r)\right|^2 + \left| \partial_t(rg_{\lin})(t,r) \right|^2\,dr\\
 \geq \frac{1}{2}\int_R^{+\infty} \left| \partial_r(rg_0) \right|^2 + |g_1|^2\,dr=\frac{1}{2}\left( \|\nabla g_0\|_{L^2}^2+\|g_1\|^2_{L^2} \right)-R|g_0(R)|^{2}.
\end{multline*}
Notice that $g_0$ is supported in $B_{\sigma}$. Bounding $g_0(R)$ similarly to $\ch_R(0,R')$ in (\ref{Eqn:C10'}), we obtain
\begin{equation*}
 \int_{|x|\geq R+|t|}\left|\nabla g_{\lin}(t,x)\right|^2 + \left| \partial_tg_{\lin}(t,x) \right|^2\,dx \geq \frac{1}{4} \left( \|\nabla g_0\|_{L^2}^2+\|g_1\|^2_{L^2} \right)
\end{equation*}
if $R$ is close enough to $\sigma$. Combining with (\ref{Eqn:D4}), we deduce that the following holds for all $t\geq 0$ or for all $t\leq 0$:
\begin{equation}
\label{Eqn:D5}
 \int_{|x|\geq R+|t|}\left|\nabla g(t,x)\right|^2 + \left| \partial_tg(t,x) \right|^2\,dx\geq\frac{1}{8}\left( \|\nabla g_0\|_{L^2}^2+\|g_1\|^2_{L^2} \right)=: \eta'>0.
\end{equation}
Indeed, $\eta$ is positive by the definition of $\BA{g_0}$ since $S_1<R<\sigma_1(\BA{w_0})$.

\medskip

\noindent\emph{Step 2.} Let $\tw$ be the solution of (\ref{Eqn:WaveSup}) with initial data $\BA{\tw_0}=\left(\TTT_Rw_0,\mathbf{1}_{\RR^3\setminus B_R}w_1\right)$. By (\ref{Eqn:D2}), $\tw$ is globally defined and scatters. Furthermore $\BA{\tw_0}=\BA{w_0}$ for $|x|>R$. Let $\tilde{h}=\tw-z_{S_1}$. Then
$$\partial_{tt}\tilde{h}-\triangle\tilde{h}=\left|z_{S_1}+\tilde{h}\right|^{p-1}(z_{S_1}+\tilde{h})-|\tilde{h}|^{p-1}\tilde{h}$$
(in the usual Duhamel sense), and, by the definition of $V$,
$$\left|z_{S_1}+\tilde{h}\right|^{p-1}(z_{S_1}+\tilde{h})-|\tilde{h}|^{p-1}\tilde{h}=\left|V+\tilde{h}\right|^{p-1}(V+\tilde{h})-|\tilde{h}|^{p-1}\tilde{h},$$
for all $(t,x)$ such that $|x|\geq S_1+|t|$. In other words, $\tilde{h}$ and $g$ satisfy the same equation for $|x|\geq S_1+|t|$. We also have $\vecc{\tilde{h}}(0,r)=\vec{g}(0,r)$ for $r\geq R>S_1$.  By finite speed of propagation and a standard argument (see again the comments after (2.27) in \cite{DuyKenMerClass}),
$$|x|\geq R+|t|\Longrightarrow g(t,x)=\tilde{h}(t,x).$$
By (\ref{Eqn:D5}), the following holds for all $t\geq 0$ or for all $t\leq 0$:
\begin{equation*}
 \int_{|x|\geq R+|t|}\left|\nabla \tilde{h}(t,x)\right|^2 + \left| \partial_t\tilde{h}(t,x) \right|^2\,dx\geq\eta'.
\end{equation*}
Since $\sigma_1(\overrightarrow{w}_0)$ is finite, we know by finite speed of propagation that
$$ \supp \overrightarrow{\tilde{h}}(t)\subset \big\{ r\leq \sigma_1(\overrightarrow{w}_0)+|t|\big\}.$$
Using H\"older's inequality, we deduce, for all $t\geq 0$ or for all $t\leq 0$,
$$ \int_{R+|t|}^{+\infty} |r\partial_r \tilde{h}(t,r)|^m+|r\partial_t \tilde{h}(t,r)|^m\,dr\geq \eta,$$
where $\eta=\frac{1}{C}{\eta'}^{\frac m2}$ for a large constant $C$ (depending on $\sigma_1(\overrightarrow{w}_0)$).
Finally, since $\tilde{w}(t,x)=\tilde{h}(t,x)+z_{S_1}(t,x)=\tilde{h}(t,x)+Z_1(x)$ if $|x|\geq R+|t|$, we obtain
$$\liminf_{t\to+\infty}\text{ or }\liminf_{t\to-\infty}\int_{R+|t|}^{+\infty} |r\partial_r \tilde{w}(t,r)|^m+|r\partial_t \tilde{w}(t,r)|^m\,dr
\geq\eta.$$
Using a similar argument to that below ``It remains to replace the lower limit...'' in the proof of Result \ref{Res:VoLimit}
we can replace the lower limit by an uniform lower bound and get (\ref{Eqn:ConcentrationEta}), concluding the proof of (\ref{I:channel}).
\end{proof}
\subsection{Proof of the profile version}
\label{Subsec:ProofProfVers}

The proof of Proposition \ref{Prop:DispPropScaling} relies upon the following lemma.

\subsubsection{Exterior energy in positive times for linear solutions}
The following lemma states roughly that the exterior generalized energy of any solution of the linear wave equation satisfies an asymptotic lower bound \emph{in both time directions}. This can be compared with Lemma \ref{Lem:Nrjchannel}, where one cannot choose the time direction, but the lower bound is more precise.
\begin{lem}
Let $w_{\lin}$ be a nonzero solution of the linear wave equation on $\RR \times \RR^3$, with initial data $\BA{w_0}:=(w_0,w_1)\in \HHH^{s_c}$, and $\delta>0$. Then there exists a solution $\tilde{w}_{\lin}$ of the linear wave equation, with initial data $\BA{\tilde{w}_0}:=(\tw_0,\tw_1) \in \HHH^{s_c}$,
and $t_{0} > 0$, $\eta > 0$, and $\rho_0 \in \mathbb{R}$ such that
\begin{gather}
\label{Eqn:ProprRes1}
\left\|\BA{\tilde{w}_{0}}\right\|_{\HHH^{s_c}}\lesssim\left\|\BA{w_{0}}\right\|_{\HHH^{s_c}},\\
\label{Eqn:ProprRes1'}
\left\|r^{1-\frac{2}{m}}(\partial_r \tilde{w}_{0},\tilde{w}_{1})\right\|_{L^m\times L^m}\leq \delta,\\
\label{Eqn:ProprRes2}
\BA{\tilde{w}_{\lin}}(t,x) = \BA{w_{\lin}}(t,x), \text{ if }t \geq t_{0}, \, |x| > \rho_0 + t
\end{gather}
and
\begin{equation}
\forall t \geq t_{0},\quad \int_{\rho_0 + t}^{+\infty} \left|r\partial_{r} \tilde{w}_{\lin} (t,x)\right|^{m} +
\left|r \partial_{t} \tilde{w}_{\lin}(t,x)\right|^{m} \, dr \geq \eta.
\label{Eqn:ProprRes3}
\end{equation}
\label{Lem:PropDispLin}
\end{lem}
We first assume the lemma and prove Proposition \ref{Prop:DispPropScaling}.
The proof of the lemma is given in Subsubsection \ref{Subsub:LemPropDispLin}.

\subsubsection{Main proof}

Let us first assume that $t_{j,n} = 0$ for all $n$. Then, by Proposition \ref{Prop:NonzeroSol} or Proposition \ref{Prop:NonzeroSol2}, there exists a solution
$\tW^{j}$ of (\ref{Eqn:WaveSup}) such that $\tW^{j}$ scatters as $t \rightarrow \pm \infty$ and positive numbers
$R_{j} > 0$, $\eta_{j} > 0$ such that
\begin{equation*}
|x| \geq R_{j}\Longrightarrow \BA{\tW^{j}}(0,x)= \BA{W^{j}}(0,x),
\end{equation*}
and the following holds for $t \geq 0$ or $t \leq 0$
\begin{equation*}
\int_{R_{j} + |t|}^{\infty}  \left|r \partial_{r} \tW^{j}(t,x)\right|^{m} + \left|r \partial_{t} \tW^{j}(t,x)\right|^{m} \, dr
\geq \eta_{j}
\end{equation*}
Hence, choosing $\rho_{j,n}:= \lambda_{j,n} R_{j}$ and letting $\tW_{\lin}^{j}$ to be the linear solution with data
$\vecc{\tW}^{j}(0)$, we see that (\ref{Eqn:DispPropScaling1}) holds and (\ref{Eqn:DispPropScaling2}) holds for all $t\leq 0$ or for all $t\geq 0$.

Assume now that $\lim_{n \rightarrow \infty} \frac{-t_{j,n}}{ \lambda_{j,n} } = +\infty$ (the case where this limit is $-\infty$ is the same). Let $\tW^{j}$ be the solution of
(\ref{Eqn:WaveSup}) such that
\begin{equation*}
\lim_{t \rightarrow +\infty } \left\| \BA{\tW^{j}}(t) - \BA{\tW_{\lin}^{j}}(t) \right\|_{\HHH^{s_c}}=0
\end{equation*}
with $\tW^{j}_{\lin}$ that is derived from $W_{\lin}^{j}$ in Lemma \ref{Lem:PropDispLin} for some small $\delta>0$. By Lemma \ref{Lem:HardyIneq},
\begin{equation}
\label{Eqn:limit_good}
\lim_{t \rightarrow \infty } \left\| r^{1-\frac{2}{m}} \partial_{r} ( \tW^{j} - \tW_{\lin}^{j} )(t) \right\|_{L^{m}}
+ \left\| r^{1-\frac{2}{m}} \partial_{t} (\tW^{j} - \tW_{\lin}^{j} )(t) \right\|_{L^{m}} =0,
\end{equation}
By Proposition \ref{Prop:CauchyPb}, if $\delta$ is small enough, $\tW^j$ scatters in both time directions.
By (\ref{Eqn:ProprRes3}) and (\ref{Eqn:limit_good}), there exists $t_{j}$ and $\eta_{j}>0$ such that
\begin{equation*}
t \geq t_{j}\Longrightarrow  \int_{\rho_{j} + t}^{+\infty} \left|r\partial_{r} \tW^{j}(t,r)\right|^{m} +
\left|r \partial_{t} \tW^{j}(t,r)\right|^{m} \, dr \geq \eta_{j} \cdot
\end{equation*}
Hence, choosing $n \gg 1$ so that $-\frac{t_{j,n}}{\lambda_{j,n}} \geq t_{j}$, and letting
$\rho_{j,n}:= \rho_{j} \lambda_{j,n} - t_{j,n}$, we see that
(\ref{Eqn:DispPropScaling1}) holds, and (\ref{Eqn:DispPropScaling2}) holds for all $t\geq 0$.
\begin{rem}
 \label{Rem:alphajn}
 We see from the proof that
 \begin{align*}
 \left( \forall n,\; t_{j,n}=0 \right) &\Longrightarrow \exists R_j>0,\;\forall n,\;\rho_{j,n}=\lambda_{j,n}R_j\\
 \lim_{n\to\infty} -\frac{t_{j,n}}{\lambda_{j,n}}\in \{\pm\infty\}&\Longrightarrow \exists \rho_j,\;\rho_{j,n}=\left|\rho_j\lambda_{j,n}-t_{j,n}\right|\sim|t_{j,n}|\text{ as }n\to\infty.
 \end{align*}
\end{rem}

\subsubsection{Proof of the asymptotic lower bound for linear solutions}
\label{Subsub:LemPropDispLin}
We next prove Lemma \ref{Lem:PropDispLin}.
Recall that if $w_{\lin}$ is a solution of the linear wave equation, with initial data in $\HHH^{s_c}$, then
$r w_{\lin}(t,r) = f(t+r) - f(t-r)$, with $f \in L_{loc}^{m}(\mathbb{R})$ and $\dot{f} \in L^{m}(\mathbb{R})$.

By Lemma \ref{Lem:UtoV}, we see that for $\rho_{0} \in \mathbb{R}$ and $t \gg |\rho_{0}|$
\begin{multline}
\label{Eqn:estimate_channel}
\int_{\rho_{0} + t}^{\infty}  \left| r \partial_{r} w_{\lin}(t,x) \right|^{m}
+ \left| r \partial_{t} w_{\lin}(t,x) \right|^{m} \, dr   \\
\approx  \int_{\rho_{0}+ t}^{\infty} \big| \dot{f}(t+r) + \dot{f}(t-r) \big|^{m} + \big|\dot{f}(t+r)- \dot{f}(t-r)\big|^{m} dr + (\rho_{0} + t) \big|w_{\lin}(t, \rho_{0} + t)\big|^{m} \\
\approx \int_{\rho_{0} +t}^{\infty} \big|\dot{f}(t+r) \big|^{m} + \big|\dot{f}(t-r) \big|^{m}  \, dr
+ (\rho_{0} + t) \big|w_{\lin}(t, \rho_{0} + t)\big|^{m}\\
\approx \int_{\rho_{0} + 2 t}^{\infty} \big|\dot{f} (r) \big|^{m} \, dr + \int_{-\infty}^{-\rho_{0}} \big|\dot{f}(r)\big|^{m} \, dr + \frac{1}{(\rho_{0} +t)^{m-1}}
\big|f(\rho_{0} + 2t) - f(- \rho_{0}) \big|^{m}\cdot
\end{multline}
Let $\rho_0 \in \mathbb{R}$ and $\delta \gg \eta > 0$  be such that
\begin{equation}
\label{Eqn:bound_dotf}
\int_{-\infty}^{-\rho_0} |\dot{f}(r)|^{m} \, dr = 2 \eta.
\end{equation}
Let $0<\epsilon\ll 1$ and $R\gg 1$ such that $\int_{R}^{\infty}\big|\dot{f}\big|^m\leq \epsilon$. Then, for $r\geq R$,
\begin{equation*}
|f(r)-f(R)|^m=\left|\int_R^r \dot{f}(s)\,ds\right|^m\leq (r-R)^{m-1}\int_{R}^r |\dot{f}(s)|^m\,ds\leq (r-R)^{m-1}\epsilon,
\end{equation*}
which proves that $\limsup_{r\to\infty} r^{1-m}|f(r)|^m\leq \epsilon$, and finally (with the same proof for $r\to-\infty$),
\begin{equation}
 \label{Eqn:bndterm_zero}
\lim_{r\to\pm \infty} r^{1-m}|f(r)|^m=0.
\end{equation}
Let $t_0\gg 1$ such that
\begin{equation}
\label{Eqn:small_remainder}
\forall t\geq t_0,\quad \int_{\rho_{0} + 2 t}^{\infty} \big|\dot{f}(r) \big|^{m} \, dr
+ \frac{1}{(\rho_{0} +t)^{m-1}}
\big|f(\rho_{0} + 2t) - f(- \rho_{0}) \big|^{m}\leq \eta.
\end{equation}
If $\tilde{w}_{\lin}$ is the solution of the linear wave equation with initial data
$$(\tw_0,\tw_1)=\left(\TTT_{\rho_{0} + t_{0}} w_{\lin}(t_{0}), \mathbf{1}_{\mathbb{R}^{3}/ B_{\rho_{0} + t_{0}}} \partial_{t} w_{\lin}(t_{0}) \right)$$ at $t=t_0$, then we see that (\ref{Eqn:ProprRes1}), (\ref{Eqn:ProprRes2}) and (\ref{Eqn:ProprRes3}) hold. Furthermore, using the pseudo-conservation of the generalized energy (Lemma \ref{Lem:ConservQuantLin}), we also get (\ref{Eqn:ProprRes1'}).

\section{ Proof of Theorem in the global case}
\label{Sec:global}
\subsection{Sketch of proof}

In this section we prove Theorem \ref{Thm:Main} in the global case (i.e $T_{+}(w) = \infty$ or $T_{-}(w) = - \infty$). We will assume $T_+(w)=+\infty$ and that (\ref{Eqn:Blowuptype1}) does not hold, and prove that $w$ scatters for positive times. The case of negative times follows by looking at the solution $w(-t,x)$ of (\ref{Eqn:WaveSup}).

We first show  (Proposition \ref{Prop:FreeWave}) that if $w$ is a solution of (\ref{Eqn:WaveSup}) that exists globally in the forward direction, and
such that there exists a sequence $t_{n} \rightarrow \infty$  along which its critical $\HHH^{s_c}$ norm is bounded,  then there exists a solution of the linear wave equation that approaches well $w$ in the weighted $L^m$-norm as $t \rightarrow \infty$, in the region $|x|\geq t-A$ ($A$ arbitrary large).

We next prove (Proposition \ref{Prop:ProfileDecompSpec}) that if $\vec{w}(t_n)$ is bounded in $\HHH^{s_c}$ and has a
profile decomposition, the only (nonlinear) profile that can have exterior generalized energy is the one corresponding to the linear component $w_{\lin}$
constructed in Proposition \ref{Prop:FreeWave}.

We then conclude the proof using the channels of generalized energy property for nonzero profiles given in the preceding section, proving that the only admissible profiles in the preceding profile decomposition is the one corresponding to $w_{\lin}$ (see Subsection \ref{Subsec:scat}). This is the core of the proof, and also its most technical part, mainly because we must take a particular care at the profiles with initial data equal to $Z_{\ell}$ (for some $\ell\neq 0$) for large $r$: see Cases 2 and 3 of Lemma \ref{Lem:Cases}.

We first state Propositions \ref{Prop:FreeWave} and \ref{Prop:ProfileDecompSpec}. Subsection  \ref{Subsec:scat} is devoted to the proof of
Theorem \ref{Thm:Main} assuming these two propositions, which are proved in Subsections \ref{Subsec:ProofFreeWave} and \ref{Subsec:GlobalProfileDecompSpec} respectively.

\begin{prop}
\label{Prop:FreeWave}
Let $w$ be a solution of (\ref{Eqn:WaveSup}) such that $T_{+}(w) = \infty$. Assume that there exists a sequence
$t_{n} \rightarrow \infty$ such that
\begin{equation}
\sup_{n \in \mathbb{N}} \left\|  \vw(t_{n})\right\|_{ \HHH^{s_c}}
<
\infty .
\label{Eqn:BoundednessHsc}
\end{equation}
Then there exists a linear solution $w_{\lin}$ such that for all $A$ in $\mathbb{R}$
\begin{equation}
\lim_{t \rightarrow +\infty} \int_{t -A}^{\infty} \left| r \partial_{r} (w - w_{\lin})(t,r) \right|^{m} +
\left|r\partial_{t} (w - w_{\lin})(t,r) \right|^{m} \, dr =0.
\label{Eqn:FreeWave}
\end{equation}
\end{prop}
\begin{prop}
\label{Prop:ProfileDecompSpec}
Let $w$ and $w_{L}$ be as in Proposition \ref{Prop:FreeWave}.
Let  $\{ \rho_{n} \}_{n \in \mathbb{N}}$ be a sequence of nonnegative numbers. There does not exist a
sequence $\{ t_{n} \}_{n \in \mathbb{N}} \rightarrow \infty $ such that
\begin{equation}
\label{Eqn:profile_Spec}
\vw(t_n)\simP \BA{w_{\lin}}(t_{n}) + \sum_{j\geq 1} \BA{W_{\lin,n}^{j}}(0),\quad |x|>\rho_n
\end{equation}
where the corresponding nonlinear profiles $W^{j}$ scatter as $t \rightarrow \pm \infty$, and there exists
$j_0 \geq 1$  such that
\begin{equation}
\int_{\rho_{n} + |t|}^{\infty} \left| r  \partial_{r} W_{n}^{j_0}(t,r) \right|^{m}
+ \left| r \partial_{t} W_{n}^{j_0}(t,r) \right|^{m} \, dr \geq \eta,
\label{Eqn:PropProf1}
\end{equation}
for some $\epsilon >0$ and for all $ t \geq 0 $ or for all $ t \leq 0 $.
\end{prop}
The statement (\ref{Eqn:profile_Spec}) means:
$$ \vw(t_n)\simP \sum_{j\geq 0} \BA{W_{\lin,n}^{j}}(0),\quad |x|>\rho_n,$$
where  $W^0_{\lin}=w_{\lin}$, $t_{0,n}=t_n$ and $\lambda_{0,n}=1$ for all $n$.

\subsection{Scattering}
\label{Subsec:scat}

We first assume Propositions \ref{Prop:FreeWave} and \ref{Prop:ProfileDecompSpec} and prove Theorem \ref{Thm:Main} for positive times, when $T_+(u)=+\infty$. The proof relies on the following lemma:
\begin{lem}
\label{Lem:Cases}
Let $w$ be a solution of (\ref{Eqn:WaveSup}) such that $T_+(w)=+\infty$, that does not scatter forward in time, and such that (\ref{Eqn:Blowuptype1}) does not hold. Let $w_{\lin}$ be given by Proposition \ref{Prop:FreeWave}.  Then, replacing $w$ by $-w$ if necessary, there exists a sequence of times $\{t_n\}_n\to\infty$, a sequence $\{\rho_n\}_n$ of positive numbers,
such that
\begin{equation}
\label{Eqn:profile_cases}
\vw(t_n)\simP \BA{w_{\lin}}(t_{n}) + \sum_{j\geq 1} \BA{W_{\lin,n}^{j}}(0),\quad |x|>\rho_n
\end{equation}
and one of the following holds
\begin{itemize}
 \item\emph{Case 1.} For all $j\geq 1$, $W^j$ scatters in both time directions and there exists $\eta>0$, $j_0\geq 1$ such that the following holds for all $t\geq 0$ or for all $t\leq 0$:
\begin{equation}
 \label{Eqn:A2}
 \int_{\rho_n+|t|}^{\infty} \left|r\partial_rW_n^{j_0}(t,r)\right|^m+\left|r\partial_tW_n^{j_0}(t,r)\right|^m\,dr\geq \eta.
\end{equation}
\item  \emph{Case 2.} For all $j\geq 2$, $W^j$ scatters in both time directions and there exists $\eta>0$, ${j_0}\geq 2$ such that (\ref{Eqn:A2}) holds for all $t\geq 0$ or for all $t\leq 0$.
Furthermore,
\begin{gather*}
\lim_{n\to\infty}-\frac{t_{j_0,n}}{\lambda_{j_0,n}}\in \{\pm\infty\}\\
\forall n,\; t_{1,n}=0; \quad \text{and}\quad (W^1_0,W^1_1)=\left(\TTT_{r_1}Z,0\right)
\end{gather*}
for some $r_1>0$ such that
\begin{equation}
 \label{Eqn:A3'}
 \forall n,\quad \rho_n\geq r_1\lambda_{1,n}.
\end{equation}
\item\emph{Case 3.} For all $j\geq 2$, $W^j$ scatters in both time directions,
$$ \forall n, \quad t_{1,n}=0; \quad \sigma_1(W_0^1,W^1_1)<\infty; \text{ and }\limsup_{n\to \infty}\frac{\rho_n}{\lambda_{1,n}}<\sigma_1(W^1_0,W^1_1).$$
 \end{itemize}
\end{lem}
\label{after_cases}

(The notation $\sigma_1$ is defined at the beginning of Section \ref{Sec:channels}).
Let us postpone the proof of Lemma \ref{Lem:Cases} to the end of this subsection. Assuming Lemma \ref{Lem:Cases}, we obtain a contradiction in each of the three cases as follows.

\medskip

\emph{Case 1.} In this case, the assumptions of Proposition \ref{Prop:ProfileDecompSpec} are satisfied which immediately gives a contradiction.

\medskip

It remains to deal with  Cases 2 and 3. We will prove that in both cases, we can reduce to Case 1 along the sequence $\{t_n\}_n$ or another sequence of times.

\medskip

\emph{Case 2.} We first notice that if $r_1>\rho_Z$, where $\rho_Z$ is defined in Lemma \ref{Lem:zR}, then we are still in the setting of Proposition \ref{Prop:ProfileDecompSpec}. Indeed, by the definition of $\rho_Z$, $W^1=z_{r_1}$ scatters in both time directions. We will reduce to this setting by proving the following:
\begin{res}
 \label{Res:A2}
 \begin{itemize}
 \item Assume:
 \begin{equation}
 \label{Eqn:A4}
 \lim_{n\to\infty} \frac{-t_{j_0,n}}{\lambda_{j_0,n}}=+\infty.
\end{equation}
 Let $\theta_{r_1}$ be given by Lemma \ref{Lem:zR}, $\tilde{t}_n :=t_n+\lambda_{1,n}\frac{\theta_{r_1}}{2}$, $\tilde{r}_1:=r_1+\frac{\theta_{r_1}}{2}$, $\trho_n:=\rho_n+\lambda_{1,n}\frac{\theta_{r_1}}{2}$. Then 
 the sequence $\left\{\vw(\tilde{t}_n)\right\}_n$ is bounded in $\HHH^{s_c}$ and satisfies (after extraction) the conclusion of Case 2 of Lemma \ref{Lem:Cases}, with $r_1$ and $\rho_n$ replaced by $\tilde{r}_1$ and $\tilde{\rho}_n$.
\item Assume:
 \begin{equation}
 \label{Eqn:A4'}
 \lim_{n\to\infty} \frac{-t_{j_0,n}}{\lambda_{j_0,n}}=-\infty.
\end{equation}
If $\tilde{t}_n := t_n-\lambda_{1,n}\frac{\theta_{r_1}}{2}$, then $\tilde{t}_{n} > T_{-}(w)$ for $n \gg 1$ and the same conclusion holds with the same  $\tilde{r}_1$ and  $\trho_n$ defined above.
\end{itemize}
 \end{res}
It is easy to see that we can assume $r_1>\rho_Z$ after a finite number of iterations of Result \ref{Res:A2}. To conclude this case, it thus suffices to prove Result \ref{Res:A2}.
\begin{proof}[Proof of Result \ref{Res:A2}]
 We note that since $\sigma_{1}(W_0^1,W_1^1)$ is finite, $(W_0^1,W_1^1)$ is not compactly supported. Since for all $n$, $t_{1,n}=0$, we can use Proposition \ref{Prop:Orth} and finite speed of propagation to prove that
 \begin{equation}
 \label{Eqn:lambda1n_tn}
 \lim_{n\to\infty} \frac{\lambda_{1,n}}{t_{n}}=0.
 \end{equation}
This shows that $\tilde{t}_n\sim t_n$ as $n\to\infty$, and thus that $\tilde{t}_n$ is in both cases in the domain of existence of $w$ for large $n$.

We prove Result \ref{Res:A2} in the case where (\ref{Eqn:A4}) holds. The proof is the same when (\ref{Eqn:A4'}) holds. Note that by (\ref{Eqn:A4}), we can assume that (\ref{Eqn:A2}) must hold for all $t\geq 0$. Indeed, assume on the contrary that (\ref{Eqn:A2}) holds for all $t\leq 0$. Letting $t=t_{j_0,n}$ in (\ref{Eqn:A2}), we obtain
$$ \eta\leq \int_{\rho_n+|t_{j_0,n}|}^{+\infty} |r\partial_{r,t} W_n^{j_0}(t_{j_0,n},r)|^m\,dr=\int_{\frac{\rho_n+|t_{j_0,n}|}{\lambda_{j_0,n}}}^{+\infty} |r\partial_{r,t} W^{j_0}(0,r)|^m\,dr,$$
which goes to $0$ as $n$ goes to infinity by (\ref{Eqn:A4'}). This is a contradiction.

By the definition of $\theta_{r_1}$, $W^1=z_{r_1}$ is defined on $\left[0,\frac{\theta_{r_1}}{2}\right]$. Since all the nonlinear profiles $W^j$, $j\geq 2$ scatter in both time directions, we can apply  Proposition \ref{Prop:Perturb_outside} to the sequence $\left\{\vw(t_n)\right\}_n$ at the time $\theta_n=\frac{\theta_{r_1}}{2}\lambda_{1,n}$. After extraction of subsequences,
\begin{equation}
 \label{Eqn:A5}
\vec{w}\left(t_n+\frac{\lambda_{1,n}\theta_{r_1}}{2}\right)
\simP \BA{w_{\lin}}\left( t_n+\frac{\lambda_{1,n}\theta_{r_1}}{2} \right)+ \sum_{j \geq 1} \BA{V^j_{\lin,n}}(0),\quad |x|>\trho_n,
\end{equation}
where, for $j\geq 1$,
$$V_{\lin,n}^j(t,x):=\frac{1}{\lambda_{j,n}^{\frac{2}{p-1}}}V^j_{\lin}\left( \frac{t-s_{j,n}}{\lambda_{j,n}},\frac{x}{\lambda_{j,n}} \right),\quad s_{j,n}:=t_{j,n}-\frac{\lambda_{1,n}\theta_{r_1}}{2}$$
and $V_{\lin}^j$ is the only solution of the linear wave equation that satisfies
\begin{equation*}
\lim_{n\to\infty}\left\|\BA{W^j}\left(\frac{-s_{j,n}}{\lambda_{j,n}}\right)-\BA{V^j_{\lin}}\left(\frac{-s_{j,n}}{\lambda_{j,n}}\right)\right\|_{\HHH^{s_c}}=0.
\end{equation*}
By (\ref{Eqn:A4}), the nonlinear profile associated to $V^{j_0}_{\lin}$ and $ \left\{ (s_{{j_0},n},\lambda_{{j_0},n}) \right\}_{n \in \mathbb{N}}$
is exactly $W^{j_0}$. Denoting by
$$V^{j_0}_n(t,x)=\frac{1}{\lambda_{{j_0},n}^{\frac{2}{p-1}}}W^{j_0}\left(\frac{t-s_{{j_0},n}}{\lambda_{{j_0},n}},\frac{x}{\lambda_{{j_0},n}}\right),$$
we have, for all $t\geq 0$,
\begin{multline*}
 \int_{\trho_n+t}^{+\infty} \left|r\partial_rV_n^{j_0}(t,r)\right|^m+\left|r\partial_tV_n^{j_0}(t,r)\right|^m\,dr\\
 =\int_{\trho_n+t}^{+\infty} \left|r\partial_rW_n^{j_0}\left(t+\frac{\theta_{r_1}\lambda_{1,n}}{2},r\right)\right|^m+\left|r\partial_tW_n^{j_0}\left(t+\frac{\theta_{r_1}\lambda_{1,n}}{2},r\right)\right|^m\,dr\geq \eta>0
\end{multline*}
by (\ref{Eqn:A2}), which proves that the analog of (\ref{Eqn:A2}) is satisfied. Using finite speed of propagation, and since $\trho_n\geq \lambda_{1,n}\tilde{r}_1$ for large $n$, and $(W_0^1,W_1^1)=(\TTT_{r_1}Z_1,0)$, we can replace in (\ref{Eqn:A5}) the profile of the decomposition corresponding to $j=1$ by the profile with initial data $\left(\TTT_{\tilde{r}_1}Z_1,0\right)$, which concludes the proof of Result \ref{Res:A2} and thus of Case 2.
\end{proof}

\noindent\emph{Case 3.} First assume that
\begin{equation}
\label{Eqn:Case_3a}
\sigma_1(W^1_0,W^1_1)>S_1,
\end{equation}
where $S_1$ is given by Proposition \ref{Prop:NonzeroSol2} (\ref{I:channel}).
We fix $R$ close to $\sigma_1\left( W^1_0,W^1_1 \right)$ and such that
$$\max\left(\limsup_{n}\frac{\rho_n}{\lambda_{1,n}},S_1\right)<R<\sigma_1\left( W^1_0,W^1_1 \right).$$
By Proposition \ref{Prop:NonzeroSol2} (\ref{I:channel}), if $R$ is close enough to $\sigma_1\left( W^1_0,W^1_1  \right)$, there exists $\eta>0$ and a scattering solution
$\tW^1$ of (\ref{Eqn:WaveSup}), with initial data $\BA{\tW_0^1}:=(\tW^1_0,\tW^1_1)$ such that
\begin{equation}
\label{Eqn:A10}
\forall r\geq R,\quad \BA{\tW_0^1}(r)=\left(W^1_0,W^1_1\right)(r),
\end{equation}
and the following holds for all $t\geq 0$ or for all $t\leq 0$.
\begin{equation}
 \label{Eqn:A9}
 \int_{R+|t|}^{+\infty} \left|r\partial_r\tW^1(t,r)\right|^m+\left|r\partial_t\tW^1(t,r)\right|^m\,dr\geq \eta.
\end{equation}
By (\ref{Eqn:profile_cases}), (\ref{Eqn:A10}) and since, by our choice of $R$, $R\lambda_{1,n}>\rho_n$ for large $n$, we see that
\begin{multline*}
 \vw(t_n)
 \simP\vw_{\lin}(t_n)+\left(\frac{1}{\lambda_{1,n}^{\frac{2}{p-1}}}\tW_0^1,\frac{1}{\lambda_{1,n}^{\frac{2}{p-1}+1}}\tW^1_1\right)\left(\frac{\cdot}{\lambda_{1,n}}\right)+\sum_{j\geq 2}\BA{W_{\lin,n}^j}(0),\quad |x|>R\lambda_{1,n}.
\end{multline*}
Combining with (\ref{Eqn:A9}), we see that the assumptions of Proposition \ref{Prop:ProfileDecompSpec} are satisfied, concluding the proof when (\ref{Eqn:Case_3a}) holds.

We next prove that we can always reduce to the case where (\ref{Eqn:Case_3a}) holds, by changing the sequence of times $\{t_n\}_n$. By Proposition \ref{Prop:NonzeroSol2} (\ref{I:rhol}), there exists $\theta>0$, $R>0$ close to $\sigma_{1}(W_{0}^{1},W_{1}^{1})$ such that
\begin{equation}
 \label{Eqn:choice_R}
\limsup_n\frac{\rho_n}{\lambda_{1,n}}<R<\sigma_1(W^1_0,W^1_1),
\end{equation}
and a solution $\check{W}^1$, defined for $t\in [-\theta,\theta]$, with initial data $\BA{\check{W}^1_0}:=(\check{W}^1_0,\check{W}^1_1)$, and such that
$$\BA{\check{W}^1_0}(r)=(W^1_0(r),W_1^1(r)),\quad r>R$$
and the following holds for all $t\in [0,\theta]$ or for all $t\in [-\theta,0]$:
\begin{equation}
\label{Eqn:checkW}
\sigma_1\left(\BA{\check{W}^1}(t)\right)=\sigma_1\left(W^1_0,W^1_1\right)+|t|.
\end{equation}
We will prove:
\begin{res}
 \label{Res:A3}
 Assume that (\ref{Eqn:checkW}) holds for all $t\in [0,\theta]$ (respectively for all $t\in [-\theta,0]$). Let $\tilde{t}_n=t_n+\lambda_{1,n}\frac{\theta}{2}$ (respectively $\tilde{t}_n=t_n-\lambda_{1,n}\frac{\theta}{2}$). Then for large $n$, $\tilde{t}_n$ is in the maximal interval of existence of $w$, the sequence $\left\{\vec{w}(\tilde{t}_n)\right\}_n$ is bounded in $\HHH^{s_c}$ and has a profile decomposition $\left\{ \tW^j_{\lin},\left\{ (\tilde{t}_{j,n}, \lambda_{j,n} ) \right\}_{n \in \mathbb{N}} \right\}_{j\geq 1}$ that satisfies the conclusions of Case 3 of Lemma \ref{Lem:Cases}, with
 $\sigma_1\left(W_0^1,W_1^1\right)$ and $\rho_n$ replaced by $\sigma_1\left(\tW_{0}^1,\tW_1^1\right):=
 \sigma_1\left(W_0^1,W_1^1\right)+\frac{\theta}{2}$ and
 $\tilde{\rho}_n := R \lambda_{1,n} + \frac{\lambda_{1,n} \theta}{2}$ respectively.
\end{res}
Iterating Result \ref{Res:A3}, and using the bound from below (\ref{Eqn:bound_below}) of $\theta$ in Proposition \ref{Prop:NonzeroSol2}, we see that we can reduce to the case where (\ref{Eqn:Case_3a}) is satisfied. It remains to prove Result \ref{Res:A3}.
\begin{proof}[Proof of Result \ref{Res:A3}]
The proof is quite similar to the proof of Result \ref{Res:A2}. The fact that $\tilde{t}_n$ is in the domain of existence of $w$ for large $n$ follows from (\ref{Eqn:lambda1n_tn}).

Since $R\lambda_{1,n}>\rho_n$ for $n\gg 1$, we deduce from (\ref{Eqn:profile_cases}):
\begin{equation}
\label{Eqn:profiles_result}
\vw(t_n)\simP \BA{w_{\lin}}(t_{n}) + \BA{\check{W}_n^1}(0)+ \sum_{j\geq 2} \BA{W_{\lin,n}^{j}}(0),\quad |x|>R\lambda_{1,n}
\end{equation}
where
$$\check{W}_n^1(t,x)=\frac{1}{\lambda_{1,n}^{\frac{2}{p-1}}}\check{W}^1\left( \frac{t}{\lambda_{1,n}},\frac{x}{\lambda_{1,n}} \right).$$
Assume that (\ref{Eqn:checkW}) holds for $t\in [0,\theta]$ to fix ideas. Using (\ref{Eqn:profiles_result}), Proposition \ref{Prop:Perturb_outside}, and recalling that $\tilde{t}_n=t_n+\frac{\lambda_{1,n}\theta}{2}$, we get
\begin{equation}
\label{Eqn:Prof_cw}
\vec{w}\left(\tilde{t}_n\right)\simP
\BA{w_{\lin}} \left(\tilde{t}_n\right) + \BA{\check{V}^1_{\lin,n}} \left(0\right)+\sum_{j\geq 2} \BA{V_{\lin,n}^j } \left(0\right),\quad |x|> \trho_n,
\end{equation}
where the modulated linear profiles $\check{V}^1_{\lin,n}$ and $V^j_{\lin,n}$ are as in Remark \ref{Rem:Perturb}.

Note that $\check{V}^1_{\lin}$ is the solution of the linear wave equation with initial data $\BA{\check{W}^1}(\frac{\theta}{2})$ that satisfies, by (\ref{Eqn:checkW}),
\begin{equation}
\label{Eqn:sigma1}
\sigma_1\Big(\BA{\check{W}^1}(\theta/2)\Big)=\sigma_1(W_0^1,W_1^1)+\frac{\theta}{2}.
\end{equation}
Note also that
$$\limsup_{n\to\infty} \frac{\tilde{\rho}_n}{\lambda_{1,n}}=R+\frac{\theta}{2}<\sigma_1 \Big(\BA{\check{W}^1}(\theta/2)\Big)$$
by (\ref{Eqn:sigma1}) and the choice (\ref{Eqn:choice_R}) of $R$.
Finally, we see that the assumptions of Case 3 are satisfied, which concludes the proof of Result \ref{Res:A3}.
\end{proof}

We have proved Theorem \ref{Thm:Main} in the global case assuming Lemma \ref{Lem:Cases}.  It remains to prove the lemma.
\begin{proof}[Proof of Lemma \ref{Lem:Cases}]
 Let $t_n\to+\infty$ such that $\vw(t_n)$ is bounded in $\HHH^{s_c}$. Extracting subsequences, we can assume that
 $\vw(t_n)$ has a profile decomposition
 \begin{equation}
 \label{Eqn:profiles_wtn}
 \vw(t_n)\simP \sum_{j\geq 1} \BA{W_{\lin,n}^j}(0).
 \end{equation}
 If all the profiles are zero, then the solution $w$ scatters, contradicting our assumptions. Thus at least one profile, says $W^1_{\lin}$ is nonzero.

Let $\epsilon>0$ be a small number such that $\epsilon<\left\|\BA{W_{\lin}^1}(0)\right\|_{\HHH^{s_c}}$, and such that any solution of (\ref{Eqn:WaveSup}) with initial data $<\epsilon$ in $\HHH^{s_c}$ is globally defined and scatters in both time directions.

We reorder the profiles so that there exists $J_0\geq 1$ such that
\begin{equation}
\left\{
\begin{aligned}
 \forall j\in \{1,\ldots,J_0\},&\quad \left\|\BA{W_{\lin}^j}(0)\right\|_{\HHH^{s_c}}\geq \epsilon\\
 \forall j\geq J_0+1,&\quad \left\|\BA{W_{\lin}^j}(0) \right\|_{\HHH^{s_c}}<\epsilon.
\end{aligned}\right.
\label{Eqn:DefJ0}
\end{equation}
We next define two subsets $\III$ and $\JJJ$ of $\{1,\ldots,J_0\}$.

\medskip

\noindent\emph{Definition of $\III$.} Recall that we can assume that the parameters $t_{j,n}$ and $\lambda_{j,n})$ satisfy (\ref{Eqn:AssParam}). If $j\in \{1,\ldots,J_0\}$, we let
\begin{equation*}
 \begin{cases}
\alpha_{j,n}=\lambda_{j,n}&\text{ if }\forall n,\;t_{j,n}=0\\
\alpha_{j,n}=|t_{j,n}|&\text{ if }\lim_n \frac{-t_{j,n}}{\lambda_{j,n}}=\pm\infty,
 \end{cases}
\end{equation*}
Then $W_{\lin,n}^j$ is essentially localized close to $\{r=\alpha_{j,n}\}$ (see Proposition \ref{Prop:Huyg}).

We define $\III\subset\{1,\ldots,J_0\}$ as the set of indices corresponding to the most exterior profiles: if $j\in \{1,\ldots,J_0\}$
\begin{equation}
 \label{Eqn:defIII}
 j\in \III\iff \forall k\in \{1,\ldots,J_0\},\quad \alpha_{k,n}\lesssim \alpha_{j,n}.
\end{equation}
Extracting subsequences, we can always assume that $\III$ is not empty, and that the following holds:
\begin{equation}
 \label{Eqn:PropIII}
 \forall j\in \III,\;\forall k\in\{1,\ldots,J_0\}\setminus \III,\quad \lim_{n\to\infty} \frac{\alpha_{k,n}}{\alpha_{j,n}}=0.
\end{equation}

\medskip

\noindent\emph{Definition of $\JJJ$.} Recall from the beginning of Section \ref{Sec:channels} the definition of $\sigma_{\ell}$.   We define $\JJJ$ as the set of indices $j\in \{1,\ldots,J_0\}$ such that $t_{j,n}=0$ for all $n$ and there exists $\ell\in \RR\setminus\{0\}$ such that
$\sigma_{\ell}(\BA{W_{\lin}^j}(0))
< \infty$. If $j\in \JJJ$ we can assume, rescaling $W^j_{\lin}$ if necessary
\begin{equation}
 \label{Eqn:defJJJ}
 \sigma_{1}\left(\BA{W^j_{\lin}}(0)\right)<\infty \quad \text{or} \quad \sigma_{-1}\left( \BA{W_{\lin}^j}(0) \right)<\infty.
\end{equation}
We distinguish three cases.

\medskip

\noindent\emph{Case 1. $\III\cap \JJJ=\emptyset$.} If $j\in\{1,\ldots,J_0\}\setminus \JJJ$, we let $\tW^j_{\lin}$ and $\{\rho_{j,n}\}_n$ be given by Proposition \ref{Prop:DispPropScaling}. Note that by Remark \ref{Rem:alphajn},
\begin{equation}
\label{Eqn:equiv_alphajn}
\rho_{j,n}\approx \alpha_{j,n}\text{ as }n\to\infty.
\end{equation}
If $j\in \JJJ$, we let $\tW^j_{\lin}$ be the solution of the linear wave equation with initial data
$$(\tW^j_0,\tW^j_1)=(\TTT_{1+\rho_Z}Z_{\pm 1},0),$$
where the sign $+$ or $-$ is the same as in (\ref{Eqn:defJJJ}), and $\rho_Z$ is defined in Lemma \ref{Lem:zR}.

Extracting subsequences, rescaling the profiles, and reordering them if necessary, we can assume, in view of (\ref{Eqn:defIII}), (\ref{Eqn:PropIII}) and (\ref{Eqn:equiv_alphajn}),
\begin{equation}
\label{Eqn:rho1_grd}
1\in \III,\quad \forall j\in \{1,\ldots, J_0\}\setminus \JJJ,\; \forall n,\; \rho_{1,n}\geq \rho_{j,n}.
\end{equation}
Let $\rho_n=\rho_{1,n}$. By (\ref{Eqn:rho1_grd}) and the definition of $\tW^j$, we have
$$\big(j\in \{1,\ldots,J_0\}\setminus\JJJ\text{ and }r\geq \rho_n\big)\Longrightarrow \BA{W_{\lin,n}^j}(0,r)=\BA{\tW_{\lin,n}^j}(0,r).$$
By (\ref{Eqn:PropIII}) and the assumption $\III\cap \JJJ=\emptyset$,
$$\forall j\in \JJJ,\quad \lim_{n\to\infty}\frac{\lambda_{j,n}}{\rho_n}=\lim_{n\to\infty}\frac{\alpha_{j,n}}{\rho_n}=0.$$
Thus for large $n$, if $j\in \JJJ$,
$$\rho_n\geq \lambda_{j,n}\max\left( 1+\rho_{Z},\sigma_{\pm 1}\left( W^j_0,W^j_1 \right) \right),$$
where the sign in $\sigma_{\pm 1}$ is again given by (\ref{Eqn:defJJJ}).
We thus obtain
$$\big(j\in \JJJ\text{ and }r\geq \rho_n\big)\Longrightarrow \BA{W_{\lin,n}^j}(0,r)=\BA{\tW_{\lin,n}^j}(0,r).$$
We can thus rewrite (\ref{Eqn:profiles_wtn}) as
$$\vec{w}(t_n)\simP \BA{w_{\lin}}(t_n)+\sum_{j=1}^{J_0} \BA{\tW^j_{\lin,n}}(0)+\sum_{j\geq J_0+1} \BA{W^{j}_{\lin,n}}(0),\quad |x|>\rho_n,$$
and the assumptions of Case 1 of Lemma \ref{Lem:Cases} are satisfied with $j_0=1$.
\\
\\
It remains to treat the other cases, i.e. when $\III\cap \JJJ$ is not empty. Since for $j\in \JJJ$, $t_{j,n}=0$ for all $n$, and thus $\alpha_{j,n}=\lambda_{j,n}$, (\ref{Eqn:defIII}) and the pseudo-orthogonality property (\ref{Eqn:PseudoOrth}) imply that $\III\cap\JJJ$ has only one element. Reordering and replacing $w$ with $-w$ if necessary, we can assume $\III\cap\JJJ=\{1\}$ and $\sigma_1(W_0^1,W_1^1) < \infty$. Using again (\ref{Eqn:PseudoOrth}), (\ref{Eqn:AssParam}) and (\ref{Eqn:defIII}), we also have:
\begin{equation}
 \label{Eqn:Case23a}
j\in \III\setminus\{1\}\Longrightarrow \lim_{n\to\infty} \frac{-t_{j,n}}{\lambda_{j,n}}\in \{\pm \infty\}.
 \end{equation}
 For $j=2\ldots J_0$, we let $\tW^j$ and $\rho_{j,n}$ be given by Proposition \ref{Prop:DispPropScaling}.
Reordering and extracting subsequences, we can assume that 
\begin{equation}
\label{Eqn:Case23b}
\forall j\in \{2,\ldots J_0\},\; \forall n,\; \rho_{2,n}\geq \rho_{j,n}.
\end{equation}
\medskip

\noindent\emph{Case 2. } In this case we also assume that

$$ \limsup_{n\to\infty} \frac{\rho_{2,n}}{\lambda_{1,n}}  \geq \sigma_1\left( W^1_0,W_1^1\right).$$
Hence, after extraction of a subsequence,
\begin{equation}
\label{Eqn:ass_Case2}
\lim_{n\to\infty} \frac{\rho_{2,n}}{\lambda_{1,n}}\geq \sigma_1\left( W_0^1,W_1^1\right).
\end{equation}

We first make a slightly stronger assumption than (\ref{Eqn:ass_Case2}):
\begin{equation}
 \label{Eqn:good}
 \forall n,\quad \rho_{2,n}\geq \lambda_{1,n}\sigma_1\left( W_0^1,W_1^1\right).
\end{equation}
Let $r_1=\sigma_1\left( W_0^1,W_1^1\right)$. By (\ref{Eqn:good}), (\ref{Eqn:A3'}) is satisfied with $\rho_n=\rho_{2,n}$. We let
$$\left( \tW_0^1,\tW_1^1\right)=\left( T_{r_1}Z_1,0 \right).$$
By the definition of $r_1$, we have
$$r>r_1\Longrightarrow \left( \tW_0^1,\tW_1^1 \right)(r)=(W_0^1,W_1^1)(r).$$
For $j=1\ldots J_0$, we define
\begin{equation*}
 \tW^j_{\lin,n}(t,x)=\frac{1}{\lambda_{j,n}^{\frac{2}{p-1}}} \tW^j_{\lin}\left( \frac{t-t_{j,n}}{\lambda_{j,n}},\frac{x}{\lambda_{j,n}} \right)
\end{equation*}
(recall that $\tW^{1}_L$ is the solution of the linear wave equation with initial data $(\tW_0^1, \tW_1^1)$ and that for $j=2 \ldots J_0$,
$\tW_{L}^j$ is given by Proposition \ref{Prop:DispPropScaling}).
By (\ref{Eqn:Case23b}) and (\ref{Eqn:A3'}),
$$\vec{w}(t_n)\simP \BA{w_{\lin}}(t_n)+\sum_{j=1}^{J_0} \BA{\tW_{\lin,n}^j}(0)
+ \sum_{j\geq J_0+1} \BA{W_{\lin,n}^j}(0),\quad |x|>\rho_n.$$
By the definition of $\rho_n=\rho_{2,n}$, $\tW^2$ satisfies (\ref{Eqn:A2}). Thus all the assumptions of Case 2 of Lemma \ref{Lem:Cases} are satisfied.

It remains to treat the general case (i.e (\ref{Eqn:good}) is not satisfied). 
By (\ref{Eqn:ass_Case2}), after extraction of a subsequence
\begin{equation}
 \label{Eqn:bad2}
 \lim_{n\to\infty}\frac{\rho_{2,n}}{\lambda_{1,n}}=\sigma:=\sigma_1\left( W^1_0,W_1^1\right).
\end{equation}
Let $ \tlambda_{1,n}=\frac{\rho_{2,n}}{\sigma}.$ Then by (\ref{Eqn:bad2}), $\lim_n\tlambda_{1,n}/\lambda_{1,n}=1$. This implies that
\begin{align*}
 \lim_{n\to\infty} \left\| \frac{1}{\lambda_{1,n}^{\frac{2}{p-1}}}W_0^1\left( \frac{\cdot}{\lambda_{1,n}} \right)-\frac{1}{\tlambda_{1,n}^{\frac{2}{p-1}}}W_0^1\left( \frac{\cdot}{\tlambda_{1,n}} \right)\right\|_{\dot{H}^{s_c}}&=0, \quad \text{and} \\
 \lim_{n\to\infty} \left\| \frac{1}{\lambda_{1,n}^{1+\frac{2}{p-1}}}W_1^1\left( \frac{\cdot}{\lambda_{1,n}} \right)-\frac{1}{\tlambda_{1,n}^{1+\frac{2}{p-1}}}W_1^1\left( \frac{\cdot}{\tlambda_{1,n}} \right)\right\|_{\dot{H}^{s_c-1}}&=0.
 \end{align*}
This proves that $\vw(t_n)$ has a profile decomposition for $|x|>\rho_n$, with profiles $\{w_{\lin}\}\cup \left\{W_{\lin}^j\right\}_{j\geq 1}$ and parameters
$$ \left\{ (1,t_n) \right\}_{n \in \mathbb{N}} \cup \left\{ (\tlambda_{1,n},t_{1,n}) \right\}_{n \in \mathbb{N}} \cup\left\{ \left\{(\lambda_{j,n},t_{j,n})\right\}_{n \in \mathbb{N}} \right\}_{j\geq 2}.$$
Since $\rho_{2,n}=\tlambda_{1,n}\sigma_1\left( W^1_0,W_1^1\right)$, we see that we are reduced to the case where (\ref{Eqn:good}) holds, concluding the proof.

\medskip

\noindent\emph{Case 3.}
In this case we also assume that
$$ \limsup_{n\to\infty} \frac{\rho_{2,n}}{\lambda_{1,n}}<\sigma_1\left( W^1_0,W_1^1\right).$$
Let $\rho_n:=\rho_{2,n}$. Then by (\ref{Eqn:Case23b}),
$$\vec{w}(t_n)\simP \BA{w_{\lin}}(t_n) + \BA{W_{L,n}^{1}}(0) + \sum_{j=2}^{J_0} \BA{\tW_{\lin,n}^j}(0)
+\sum_{j\geq J_0+1} \BA{W_{\lin,n}^j}(0),\quad |x|>\rho_n,$$
and it is easy to check that all the assumptions of Case 3 of Lemma \ref{Lem:Cases} are satisfied.
\end{proof}

\subsection{Existence of the free wave}
\label{Subsec:ProofFreeWave}
We next prove Proposition \ref{Prop:FreeWave}. We start with a preliminary lemma.
\subsubsection{Scattering outside wave cones}
\begin{lem}
\label{Lem:HuygConseq}
Let $w$ be a solution of (\ref{Eqn:WaveSup}) such that $T_+(w)=+\infty$ and (\ref{Eqn:Blowuptype1}) does not hold. Then there exists a sequence $\{s_n\}_n\to+\infty$ such that
$$\limsup_{n\to\infty}\left\|\vec{w}(s_n)\right\|_{\HHH^{s_c}}<\infty,$$
a sequence $\{ (w_{0,n},w_{1,n}) \}_{n \in \mathbb{N}}$, bounded in $\HHH^{s_c}$, and a small $\epsilon>0$ such that for large $n$, the solution $w_n$ of (\ref{Eqn:WaveSup}) with initial data $(w_{0,n},w_{1,n})$ scatters forward in time and
\begin{equation}
 \label{Eqn:G2}
|x| > (1- \epsilon) s_n \Longrightarrow \vec{w}(s_n,x) = (w_{0,n}(x),w_{1,n}(x)).
\end{equation}
\end{lem}
\begin{proof}
\noindent\emph{Step 1.} Let $\{t_n\}\to\infty$ such that $\{\vw(t_n)\}_n$ is bounded in $\HHH^{s_c}$. In this step we prove that there exists (after extraction and for large $n$) an $\epsilon' >0$ such that
\begin{equation}
 \label{Eqn:equals_tn}
\vec{w}(t_n) \simP  \sum_{j \geq 1} \BA{\tW_{L,n}^{j}}(0), \quad |x| > (1-\epsilon')t_n,
\end{equation}
with nonlinear profiles $\tW^j$, such that for all $j$, $\tW^j$ scatters in both time directions or
\begin{equation}
 \label{Eqn:G2'} \lim_{n\to\infty} \frac{-t_{j,n}}{\lambda_{j,n}}\in \{\pm\infty\}\text{ and }\forall n,\; |t_{j,n}|> (1-2\epsilon')t_n.
\end{equation}
Extracting subsequences, we can assume that
\begin{equation}
\vec{w}(t_n) \simP  \sum_{j \geq 1} \BA{W_{L,n}^{j}}(0).
\nonumber
\end{equation}
If all the corresponding nonlinear profiles $W^j$ scatters forward in time then $w$ scatters by Proposition \ref{Prop:Perturb} and
and Remark \ref{Rem:ConseqLWP}, and one can choose $\epsilon':=1$. If not, we reorder the profiles so that for $j\geq J_0+1$, $W^j$ scatters in both time directions, and for $1\leq j\leq J_0$, $W^j$ does not scatter, at least in one time direction. Let $j\geq 1$. If $j\geq 1+J_0$, we let $\tW_{\lin}^j=W_{\lin}^j$. If $1\leq j\leq J_0$, we will obtain $\tW^j_{\lin}$ from $W^j_{\lin}$ by truncation as follows.

\medskip

\noindent\emph{Case 1.} Assume
\begin{equation}
 \label{Eqn:G3'}
 \forall n,\quad t_{j,n}=0.
\end{equation}
By Proposition \ref{Prop:small_exterior} and Proposition \ref{Prop:Orth} 
the sequence $\left\{\lambda_{j,n}/t_n\right\}_n$ is bounded, and we distinguish between two subscases.

\medskip

\noindent{\emph{Case 1a.}}
Assume (\ref{Eqn:G3'}) and
\begin{equation}
\label{Eqn:G6}
 \lim_{n\to\infty} \frac{\lambda_{j,n}}{t_n}\in (0,\infty).
\end{equation}
By the pseudo-orthogonality conditions (\ref{Eqn:PseudoOrth}), there is no other $j$ satisfying (\ref{Eqn:G3'}) and (\ref{Eqn:G6}).
Rescaling $(W_0^j,W_1^j)$, we can assume that the preceding limit is $1$, and that $\lambda_{j,n}=t_n$ for all $n$. Then, changing $W^j_0$ and
$W^j_1$ on a negligible set if necessary:
\begin{equation}
 \label{Eqn:G7}
 \supp(W^j_0,W^j_1)\subset B_1
\end{equation}
Indeed, Proposition \ref{Prop:Orth} yields for $R \gg 1$ and $n \gg 1$
\begin{multline*}
\| r^{1-\frac{2}{m}}\partial_{r,t} w(t_{n}) \|^{m}_{L^{m}\left(\left\{|x| \geq t_{n} + R\right\}\right)}
\gtrsim \left\| r^{1-\frac{2}{m}} \partial_{r,t}W^{j}_{\lin,n} \left(0\right)
\right\|^{m}_{L^{m} \left(\left\{|x| \geq t_{n} +R\right\}\right)} + o_{n}(1) \\
= \left\| r^{1-\frac{2}{m}} \partial_{r} W^{j}_0 \right\|^{m}_{L^{m} \left( \left\{|x| \geq 1 + \frac{R}{t_{n}} \right\}\right)} +
 \left\| r^{1-\frac{2}{m}} W^{j}_1 \right\|^{m}_{L^{m} \left( \left\{|x| \geq 1 + \frac{R}{t_{n}} \right\}\right)}+ o_{n}(1).
\end{multline*}
Letting  $n \rightarrow \infty$ and $R \rightarrow \infty$, and using Proposition \ref{Prop:small_exterior}, we obtain (\ref{Eqn:G7}).

By Result \ref{Res:Cont_R}, we can choose $0<\epsilon'\ll 1$ such that
$$\left(\tW_0^j,\tW_1^j\right):=\left( \TTT_{1-\epsilon'} W_0^j,\textbf{1}_{\RR^3\setminus B_{1-\epsilon'}} W_1^j \right)$$
satisfies
$$\left\|\left(\tW_0^j,\tW_1^j\right)\right\|_{\HHH^{s_c}} \leq\frac{\delta_0}{C_0},$$
where $\delta_0$ and $C_0$ are given by Proposition \ref{Prop:CauchyPbHsc} (the small data theory for the equation (\ref{Eqn:WaveSup})). Letting $\tW^j$ be the the nonlinear profile associated to $ \tW^j_{\lin} $ and  $ \left\{ (t_{j,n}, \lambda_{j,n}) \right\}_{n \in \mathbb{N}} $, we see that $\tW^j$ scatters in both time directions by Proposition \ref{Prop:CauchyPbHsc}. Furthermore,
\begin{equation}
 \label{Eqn:G8}
 |x|\geq (1-\epsilon')t_n\\
\Longrightarrow \BA{\tW^j_{\lin,n}}(0,x)= \BA{W^j_{\lin,n}}(0,x).
\end{equation}
\medskip

\noindent{\emph{Case 1b.}}
Assume (\ref{Eqn:G3'}) and
\begin{equation}
 \label{Eqn:G4}
 \lim_{n\to\infty}\frac{\lambda_{j,n}}{t_n}=0.
\end{equation}
As before, by Result \ref{Res:Cont_R} in the appendix and Proposition \ref{Prop:CauchyPbHsc}, we can choose $R_j\gg 1$ such that the solution $\tW^j$ with initial data
\begin{equation}
\label{Eqn:deftW0j}
\left(\tW_0^j,\tW_1^j\right):=\left( \TTT_{R_j} W_0^j,\textbf{1}_{\RR^3\setminus B_{R_j}} W_1^j \right)
\end{equation}
scatters in both time directions. Note by the definition of $(\tW^j_0,\tW^j_1)$ that
$$|x|\geq R_j\lambda_{j,n}\Longrightarrow \BA{\tW^j_{\lin,n}}(0,x)= \BA{W^j_{\lin,n}}(0,x).$$
In view of (\ref{Eqn:G4}), we see that (\ref{Eqn:G8}) is satisfied for large $n$.

\medskip

\noindent{\emph{Case 2.}} We next assume
\begin{equation}
 \label{Eqn:G9}
 \lim_{n\to\infty} \frac{-t_{j,n}}{\lambda_{j,n}}\in \{\pm\infty\}.
\end{equation}
Again, we distinguish between two subcases.

\medskip

\noindent{\emph{Case 2a.}} We assume (\ref{Eqn:G9}) and, after extraction of a subsequence
\begin{equation}
 \label{Eqn:G9'}
 \forall n,\quad |t_{j,n}|\leq (1-2\epsilon')t_n.
\end{equation}
Using Result \ref{Res:Cont_R}, we define again $(\tW^j_0,\tW^j_1)$ by (\ref{Eqn:deftW0j}), where $R_j\gg 1$ is such that
$$ \left\|(\tW_0^j,\tW^j_1)\right\|_{\HHH^{s_c}}\leq \frac{\delta_0} {2 C_0}.$$
As a consequence, for all $t\in \RR$,
$$\left\|\BA{\tW^j_{\lin}}(t)\right\|_{\HHH^{s_c}}=\left\|(\tW_0^j,\tW^j_1)\right\|_{\HHH^{s_c}}\leq \frac{\delta_0}{2C_0}.$$
Thus by Proposition \ref{Prop:CauchyPbHsc} and (\ref{Eqn:ScattProf}), the nonlinear profile $\tW^j$ associated to $ \tW^j_{\lin} $ and
$\left\{ (\lambda_{j,n},t_{j,n}) \right\}_{n \in \mathbb{N}}$ scatters in both time directions. Furthermore we obtain by finite speed of propagation:
\begin{equation*}
 \frac{|x|}{\lambda_{j,n}}\geq \frac{|t_{j,n}|}{\lambda_{j,n}}+R_j\Longrightarrow \BA{\tW^j_{\lin,n}}(0,x)=\BA{W^j_{\lin,n}}(0,x).
\end{equation*}
In view of (\ref{Eqn:G9}) and (\ref{Eqn:G9'}), we see that (\ref{Eqn:G8}) is satisfied for large $n$.

\medskip

\noindent{\emph{Case 2b.}} We assume (\ref{Eqn:G9}) and after extraction of a subsequence
\begin{equation}
 \label{Eqn:G9''}
 \forall n,\quad |t_{j,n}|>(1-2\epsilon')t_n.
\end{equation}
In this case, we simply let $\tW^j_{\lin}:=W^j_{\lin}$.\\
\\
Recalling that (\ref{Eqn:G8}) is satisfied for all $j\geq 1$, we obtain as announced (\ref{Eqn:equals_tn}),
where for all $j\geq 1$, the corresponding nonlinear profile $\tW^j$ scatters in both time directions or satisfies (\ref{Eqn:G2'}).

\medskip

\noindent\emph{Step 2. Conclusion of the proof.} If all the nonlinear profiles $\tW^j$ corresponding to the preceding profile decomposition scatter forward in time, then by Proposition \ref{Prop:Perturb}, the solution $\tw_n$ with initial data $(\tw_{0,n},\tw_{1,n})
:= \sum_{j \geq 1} \BA{\tilde{W}_{L,n}^{j}}(0)$ scatters for large $n$, and the conclusion of the lemma is satisfied with $s_n:=t_n$, $\epsilon=\epsilon'$, $(w_{0,n},w_{1,n}):=(\tw_{0,n},\tw_{1,n})$.

We thus assume that there is at least one nonlinear profile $\tW^j$ that does not scatter forward in time. We reorder the profiles, so that (for some $J_1\geq 1$)
if $1\leq j\leq J_1$ then $\tW^j$ does not scatter forward in time, and if $j\geq J_1+1$, $\tW^j$ scatters forward in time. We will prove the conclusion of the lemma with $s_n:=\frac{3t_n}{2}$.

We first note that the assumptions of Proposition \ref{Prop:Perturb} are satisfied with $\theta_n=\frac{t_n}{2}$. Indeed, by Step 1, if $1  \leq j \leq J_1$, then
$$\lim_{n\to\infty}\frac{-t_{j,n}}{\lambda_{j,n}}=-\infty\text{ and }|t_{j,n}|>(1-2\epsilon')t_n\text{ for large }n.$$
As a consequence, for such a profile, we have, using that $\epsilon'<1/4$,
$$ \limsup_{n\to\infty} \frac{\frac{t_n}{2}-t_{j,n}}{\lambda_{j,n}}=-\infty,$$
and thus the assumption (\ref{Eqn:limsup}) of Proposition \ref{Prop:Perturb} is satisfied. By the conclusion of the proposition, we have
for all $J \geq 1$
\begin{equation}
 \label{Eqn:G11}
\BA{\tw_n}\left( \frac{t_n}{2} \right)=\sum_{j=1}^{J} \BA{\tW^j_n} \left( \frac{t_n}{2} \right)+ \BA{\eps_n^J} \left( \frac{t_n}{2} \right)
+ \BA{r^J_n} \left( \frac{t_n}{2} \right).
\end{equation}
Let $\psi$ be a smooth function such that
$$ |x|\geq 1\Longrightarrow \psi(x)=1,\quad |x|\leq \frac{3}{4}\Longrightarrow \psi(x)=0.$$
Let
\begin{equation}
 \label{Eqn:G12}
(w_{0,n},w_{1,n}):= \psi\left( \frac{\cdot}{t_n} \right)\sum_{j=1}^{J_{1}} \BA{\tW^j_n} \left( \frac{t_n}{2} \right) +
  \BA{\eps_n^{J_1}} \left( \frac{t_n}{2} \right) + \BA{r^{J_1}_n} \left( \frac{t_n}{2} \right),
\end{equation}
By finite speed of propagation and (\ref{Eqn:equals_tn}),
$$|x|\geq \left(\frac{3}{2}-\epsilon'\right)t_n\Longrightarrow \BA{\tw_n} \left( \frac{t_n}{2},x \right)= \vec{w} \left( \frac{3t_n}{2},x \right).$$
By the definition of $\psi$,  (\ref{Eqn:G11}) and (\ref{Eqn:G12}),
$$|x|\geq t_n\Longrightarrow \BA{\tw_n}  \left( \frac{t_n}{2},x \right)=\left(w_{0,n},w_{1,n}\right)(x).$$
Combining the two equalities above, we deduce that (\ref{Eqn:G2}) holds with $s_n:=\frac{3t_n}{2}$, $\epsilon:=\frac{2}{3}\epsilon'$.

Furthermore by Step 1, if $1\leq j\leq J_1$,
$$\lim_{n\to\infty}\frac{-t_{j,n}}{\lambda_{j,n}}=-\infty\text{ and } (1-2\epsilon')t_n < t_{j,n}. $$
We also claim that for $n \gg 1$

\begin{align}
t_{j,n}<(1+\eps')t_n
\label{Eqn:Boundtjn}
\end{align}
Indeed, we first observe by Lemma \ref{Lem:UtoV} and Lemma \ref{Lem:ConservQuantLin} that

\begin{align*}
\begin{array}{ll}
\left\| r^{1 - \frac{2}{m}} \partial_{r,t}
\tW_{L,n}^{j}(0) \right\|_{L^{m}}  &
= \left\| r^{1 - \frac{2}{m}} \partial_{r,t} \tW_{L}^{j} ( -t_{j,n}, x )
\right\|_{L^{m}}  \\
& \approx \left\| \left( r^{1 - \frac{2}{m}} \partial_{r} W_{0}^{j}, r^{ 1 - \frac{2}{m}} W_{1}^{j} \right) \right\|_{L^{m}}
\end{array}
\end{align*}
Then, by Proposition \ref{Prop:Huyg}, Proposition \ref{Prop:small_exterior}, and Proposition \ref{Prop:Orth}, we can choose $R \gg 1$ and $n \gg 1$ such that

\begin{align*}
\left\| r^{1 - \frac{2}{m}} \partial_{r,t}
\tW_{L,n}^{j}(0) \right\|_{L^{m} \left( \left| |x| - |t_{j,n}|   \right| \geq R \lambda_{j,n} \right)}  & \ll \left\| \left( r^{1 - \frac{2}{m}} \partial_{r} W_{0}^{j}, r^{ 1 - \frac{2}{m}} W_{1}^{j} \right) \right\|_{L^{m}}, \quad \text{and} \\
\left\| r^{1- \frac{2}{m}}  \partial_{r,t} \tW_{L,n}^{j}(0) \right\|_{L^{m} \left( |x| \geq t_n +R \right) } & \ll \left\| \left( r^{1 - \frac{2}{m}} \partial_{r} W_{0}^{j}, r^{ 1 - \frac{2}{m}} W_{1}^{j} \right) \right\|_{L^{m}}
\end{align*}
Thus if (\ref{Eqn:Boundtjn}) were not true then this would lead to a contradiction.\\
\\
Hence
$$\lim_{n\to\infty} \frac{\frac{t_n}{2}-t_{j,n}}{\lambda_{j,n}}=-\infty,\quad \text{and for large }n,\; \left|\frac{t_n}{2}-t_{j,n}\right|\leq \left(\frac{1}{2}+\epsilon'\right)t_n.$$
Hence
$$\lim_{n\to\infty}\left\|\psi\left( \frac{\cdot}{t_n} \right)\sum_{j=1}^{J_1} \BA{\tW^j_n}\left( \frac{t_n}{2} \right)\right\|_{\HHH^{s_c}}=0$$
(indeed by (\ref{Eqn:ScattProf2}), we can substitute $\BA{\tW_{\lin,n}^{j}} \left(\frac{t_{n}}{2}\right)$  for $\BA{\tW^j_n} \left( \frac{t_n}{2} \right)$
\footnote{Recall
that $ \left\| \psi \left( \frac{x}{t_{n}} \right) g \right\|_{\dot{H}^{q}} \lesssim \| g \|_{\dot{H}^{q}} $
for  $g \in \dot{H}^{q}$ and $q \in \{s_{c},s_{c}-1 \}$}; but then by finite speed of propagation the
statement is obvious if $\BA{\tW^j_{\lin}}(0)$ is compactly supported; the general case follows by a density argument).
Combining with (\ref{Eqn:G11}), (\ref{Eqn:G12}), we deduce that for all $J$,
\begin{equation}
\label{Eqn:G11'}
(w_{0,n},w_{1,n})=\sum_{j=1+J_1}^{J} \BA{\tW^j_n} \left( \frac{t_n}{2} \right)+ \left(\check{\eps}_{0,n}^{J},\check{\eps}_{1,n}^J\right),
\end{equation}
where $\left(\check{\eps}_{0,n}^{J},\check{\eps}_{1,n}^J\right)$ satisfies (\ref{Eqn:RemSmall}), in view of (\ref{Eqn:Strich}).

Observing (see Remark \ref{Rem:Perturb})  that (\ref{Eqn:G11'}) yields a profile decomposition of the sequence $(w_{0,n},w_{1,n})$, where the nonlinear profiles are the $\tW^j$, $j\geq 1+J_1$ (and thus all these nonlinear profiles scatter forward in time), we deduce from Proposition \ref{Prop:Perturb} and Remark
\ref{Rem:ConseqLWP} that the solution $w_n$ with initial data $(w_{0,n},w_{1,n})$ scatters forward in time, concluding the proof of the lemma.

\end{proof}

\subsubsection{Construction of the free wave}
\label{subsubsec:ProofFreeWave}

We are now in position to prove  Proposition \ref{Prop:FreeWave}.

\medskip

\noindent\emph{Step 1.} We prove that $\forall A \in \mathbb{R}$, there exists a solution $w_{\lin}^{A}$ of the linear wave equation such that
\begin{align}
\lim_{t \rightarrow \infty} \int_{t - A}^{\infty} \left|r \partial_{r}(w - w_{\lin}^{A})(t,r) \right|^{m} +
\left|r \partial_{t} (w -w_{\lin}^{A}) (t,r) \right|^{m} \, dr =0.
\label{Eqn:FreeWaveStep1}
\end{align}
Consider the sequence $\{ s_{n} \}_{n \in \mathbb{N}}$  and the solution $w_{n}$ of (\ref{Eqn:WaveSup}) given by Lemma \ref{Lem:HuygConseq}. Since $w_{n}$ scatters as $t \rightarrow \infty$, there exists
$\tilde{w}_{L,n}$ solution of the linear wave equation such that
\begin{align*}
\lim_{t \rightarrow \infty} \| \BA{w_{n}}(t) - \BA{\tilde{w}_{L,n}}(t) \|_{\HHH^{s_c} }
=0
\end{align*}
By finite speed of propagation $\vec{w}(t + s_{n},x) = \vec{w}_{n}(t,x)$ for $|x| \geq (1- \epsilon) s_{n} + t$ and $t \geq 0$. Hence, choosing
$n$ such that $\epsilon s_{n} \gg A $ and using Lemma \ref{Lem:HardyIneq} we see that (\ref{Eqn:FreeWaveStep1}) holds with
$w_{\lin}^{A}(t,x) := \tilde{w}_{L,n}(t-s_{n},x)$.

\medskip

\noindent\emph{Step 2.}
Since $ \left( S(-s_{n}) \vec{w}(s_{n}), \partial_{t} S (-s_{n}) \vec{w}(s_{n}) \right) $ is a bounded sequence in $\HHH^{s_c}$, one can
assume (after extraction of a subsequence) that
$$ \left( S(-s_{n}) \vec{w}(s_{n}), \partial_{t} S (-s_{n}) \vec{w}(s_{n}) \right)  \xrightharpoonup[n\to\infty]{} (v_0,v_1)  \text{ in }
\HHH^{s_c} .$$
Let $w_{\lin}(t):=S(t)(v_0,v_1)$.
We may assume that (after extraction again) that
$\vec{w}(s_{n}) - \BA{w_{\lin}^{A}}(s_{n}) $ has a profile decomposition
\begin{equation*}
\vec{w}(s_{n}) - \BA{w_{\lin}^{A}}(s_{n}) \simP \BA{w_{\lin}}(s_{n}) - \BA{w_{\lin}^{A}}(s_{n}) + \sum_{j\geq 2} \BA{W^{j}_{L,n}}(0),
\end{equation*}
where the first profile is  $ \vW^{1}_{L,n}(0):=  \vec{w}_{\lin}(s_{n}) - \BA{w_{\lin}^{A}}(s_{n})$. Hence we see from Proposition
\ref{Prop:Orth} and (\ref{Eqn:FreeWaveStep1}) that
\begin{equation}
\label{Eqn:lim0}
\lim_{n \rightarrow \infty}
\int_{s_{n} - A}^{\infty}  \left| r\partial_{r} (w_{\lin}^{A} - w_{\lin})  (s_{n},r)  \right|^{m}
+ \left| r\partial_{t} (w_{\lin}^{A} - w_{\lin})  (s_{n},r) \right|^{m} \, dr  =0.
\end{equation}
Let $n \gg 1$ such that $s_{n} \gg A $. Let $\tilde{w}_{\lin}$ (resp. $\tilde{w}^{A}_{\lin}$) be the solution of the linear wave equation with data
$ \BA{\tilde{w}_{\lin}}(s_{n})  := \left( \TTT_{s_{n} - A} w_{\lin}(s_{n}), \mathbf{1}_{\mathbb{R}^{3}/ B_{s_{n}-A}} \partial_{t}w_{\lin} (s_{n}) \right)$ ( resp. $ \BA{\tilde{w}^{A}_{\lin}}(s_{n})  :=
\left( \TTT_{s_{n} - A} w_{\lin}^{A}(s_{n}), \mathbf{1}_{\mathbb{R}^{3}/ B_{s_{n} - A}} \partial_{t} w_{\lin}^{A} (s_{n}) \right) $ ). By using finite speed of propagation and the pseudo-conservation of the generalized energy (Lemma \ref{Lem:ConservQuantLin}), we see that for $t \gg s_{n}$
\begin{multline*}
\int_{t - A}^{\infty}
\left| r \partial_{r,t} (w_{\lin}^{A} - w_{\lin}) (t,r)  \right|^{m}\,dr=
\int_{t - A}^{\infty}
\left| r \partial_{r,t} (\tw_{\lin}^{A} - \tw_{\lin}) (t,r)  \right|^{m}\,dr\\
\lesssim
\int_{0}^{\infty}
\left| r \partial_{r,t} (\tw_{\lin}^{A} - \tw_{\lin}) (s_n,r)  \right|^{m}\,dr
=\int_{s_n - A}^{\infty}
\left| r \partial_{r,t} (w_{\lin}^{A} - w_{\lin}) (s_n,r)  \right|^{m}\,dr.
\end{multline*}
Hence, by (\ref{Eqn:lim0}),
\begin{equation}
\lim_{t \rightarrow \infty} \int_{t-A}^{\infty}  \left|r\partial_{r,t} (w_{\lin}^{A} - w_{\lin})(t,r) \right|^{m}
\, dr =0 \cdot
\label{Eqn:Limitwl}
\end{equation}
Combining (\ref{Eqn:Limitwl}) with (\ref{Eqn:FreeWaveStep1}) we get (\ref{Eqn:FreeWave}). \qed

\subsection{Nonexistence of profile with exterior generalized energy} 
\label{Subsec:GlobalProfileDecompSpec}

We prove here Proposition \ref{Prop:ProfileDecompSpec}.
Let $\bar{w}$ be the solution of (\ref{Eqn:WaveSup}) such that
\begin{equation}
\lim_{t \rightarrow \infty} \| \BA{\bar{w}}(t) - \BA{w_{\lin}}(t) \|_{\HHH^{s_c}} =0.
\label{Eqn:Defbarw}
\end{equation}
Translating $w$ and $\bar{w}$ in time if necessary, we may assume without loss of generality that $\bar{w}$ is defined on $[0, \infty)$.

By the assumptions of Proposition \ref{Prop:ProfileDecompSpec}, there exists $(w_{0,n},w_{1,n})$ such that
\begin{equation}
 \label{Eqn:PropProf0}
(w_{0,n},w_{1,n})(r)=\vec{w}(t_n,r), \quad r>\rho_n,
\end{equation}
and $(w_{0,n},w_{1,n})$ has the profile decomposition
\begin{equation}
\label{Eqn:PropProf2}
(w_{0,n},w_{1,n})\simP \BA{w_{\lin}}(t_{n}) + \sum_{j\geq 1} \BA{W_{\lin,n}^{j}}(0).
\end{equation}

Let $w_{n}$ be the solution with data $(w_{0,n},w_{1,n})$ at $t=0$. By Proposition \ref{Prop:Perturb}, we see that
\begin{equation*}
\BA{w_{n}}(t,x) = \BA{\bar{w}}(t_{n}+ t,x) + \sum_{j=1}^{J} \BA{W_{n}^{j}}(t,x) + \BA{\eps_{n}^{J}}(t,x) + \BA{r_{n}^{J}}(t,x),\quad t\in[-t_n,+\infty)
\end{equation*}
with
\begin{align*}
\lim_{J \rightarrow \infty} \limsup_{n \rightarrow \infty} \sup_{t \in [-t_{n}, \infty)}  \| \BA{r_{n}^{J}}(t)
\|_{\HHH^{s_c} } =0 .
\end{align*}
Assume that (\ref{Eqn:PropProf1}) holds for $t \geq 0$. Then, using also
Proposition \ref{Prop:Orth}, (\ref{Eqn:Defbarw}), and Lemma \ref{Lem:HardyIneq}, we see that for $n \gg 1$
\begin{align*}
\forall t\geq 0,\quad \int_{\rho_{n} + t}^{\infty} \left|  r \partial_{r,t} w_{n}(t,r) - r \partial_{r,t} w_{\lin}(t_{n} + t,r) \right|^{m}\, dr \gtrsim \eta.
\end{align*}
Hence, by finite speed of propagation and (\ref{Eqn:PropProf0})
\begin{equation*}
\forall t\geq 0,\quad \int_{\rho_{n} + t}^{\infty}  \left| r \partial_{r,t} w(t_{n}+ t,r)  - r \partial_{r,t} w_{\lin}(t_{n}+t,r) \right|^{m}\,dr
\gtrsim \eta.
\end{equation*}
Letting $t\to\infty$, we see that it contradicts Proposition \ref{Prop:FreeWave}.

Next assume that (\ref{Eqn:PropProf1}) holds for $t \leq 0$. Then
\begin{align*}
\int_{\rho_{n} + t_{n}}^{\infty} \left| r  \partial_{r,t} w_{n}(-t_{n},r) - r\partial_{r,t} \overline{w}(0,r)  \right|^{m}\,dr \gtrsim \eta.
\end{align*}
Hence, by using again finite speed of propagation and (\ref{Eqn:PropProf0})
\begin{align*}
\int_{\rho_{n} + t_{n}}^{\infty} \left| r \partial_{r,t} w(0,r) - r \partial_{r,t} \overline{w}(0,r) \right|^{m}\,dr\gtrsim \eta.
\end{align*}
Letting $n \rightarrow \infty$, we see that it is impossible.

\section{Proof of Theorem in the finite maximal time of existence case}
\label{Sec:blowup}
\subsection{Outline of the proof}

This section is devoted to the proof of Theorem  \ref{Thm:Main} in the finite maximal time of existence case (i.e
$T_{+}(w) < \infty$  or  $ T_{-}(w) > -\infty$). By the invariances of the equation, we can without loss of generalities consider only positive times and assume $T_+(w)=1$.

We first show (Proposition \ref{Prop:ConstructionV}) that if $w$ is a solution of (\ref{Eqn:WaveSup}) such that $T_+(w)=1$, and
there exists a sequence $t_{n} \rightarrow 1$ along which its critical $\HHH^{s_{c}} $ norm is bounded, then there exists a solution $v_{+}$ of (\ref{Eqn:WaveSup}) such that
$w=v_{+}$ outside the light cone $|x|=1-t$. We then prove (Proposition
\ref{Prop:ProfileDecompSpecFinite}) the analog of Proposition \ref{Prop:ProfileDecompSpec} for the finite time of existence case, namely that if $t_n\to
1$  and $\BA{w}(t_n)$ is bounded in $\HHH^{s_c}$ and has a profile decomposition, then
the corresponding nonlinear profiles (except the one corresponding to the solution $v_{+}$) do not have any exterior generalized energy. The end of the proof is very close to the corresponding proof in the global case and we will only sketch it.
\begin{prop}
Assume that $T_{+}(w) =1$, and that there exist a sequence $t_{n} \rightarrow 1$ such that
(\ref{Eqn:BoundednessHsc}) holds. Then there exists a solution $v_{+}$ of (\ref{Eqn:WaveSup}), defined in a neighborhood of $t=1$, and
$T_{0} \in (0,1)$ such that
\begin{equation*}
\left(t \in [T_{0},1)\quad\text{and}\quad |x| \geq 1-t\right)\Longrightarrow v_{+}(t,x) =w(t,x).
\end{equation*}
\label{Prop:ConstructionV}
\end{prop}
\begin{prop} Let $w$ be as in Proposition \ref{Prop:ConstructionV}.
Let  $\{ \rho_{n} \}_{n \in \mathbb{N}}$ be a sequence of nonnegative numbers. There does not exist a
sequence $\{ t_{n} \}_{n \in \mathbb{N}} \rightarrow 1 $  such that $\{\vw(t_n)\}_n$ has a profile decomposition for $|x|>\rho_n$:
\begin{align}
\vec{w}(t_n)\simP\vec{v}_+(1) + \sum_{j\geq 1} \BA{W_{\lin,n}^{j}}(0),\quad |x|>\rho_n,
\label{Eqn:ProfDecompPropProf}
\end{align}
where the corresponding nonlinear profiles $W^{j}$ scatter as $t \rightarrow \pm \infty$, and there exists
$j_0 \geq 1$  such that
\begin{align}
\int_{\rho_{n} + |t|}^{\infty} \left| r\partial_{r} W_{n}^{j_0}(t,r) \right|^{m}
+ \left| r \partial_{t} W_{n}^{j_0}(t,r) \right|^{m} \, dr \geq \epsilon,
\label{Eqn:PropProf1finite}
\end{align}
for some $\epsilon >0$ and for all $ t \geq 0 $ or for all $ t \leq 0 $.
\label{Prop:ProfileDecompSpecFinite}
\end{prop}

\subsection{Nonexistence of type II blow-up solution}
Assuming Propositions \ref{Prop:ConstructionV} and \ref{Prop:ProfileDecompSpecFinite}, the proof of the nonexistence of solutions of (\ref{Eqn:WaveSup}) that do not satisfy (\ref{Eqn:Blowuptype1}) and blow up in finite time relies on the following lemma, which is the analog of Lemma \ref{Lem:Cases} for the finite time blow-up case:
\begin{lem}
\label{Lem:Cases_Blowup}
Let $w$ be a solution of (\ref{Eqn:WaveSup}) such that $T_+(w)=1$ and such that (\ref{Eqn:Blowuptype1}) does not hold. Then, replacing $w$ by $-w$ if necessary, there exists a sequence of times $\{t_n\}_n\to 1$, in the domain of existence of $w$,
a sequence of positive number $\{\rho_n\}_n$ such that $\vw(t_n)$ has the following profile decomposition for $|x|>\rho_n$:
\begin{equation*}
 \vw(t_n)\simP \BA{v_+}(1) + \sum_{j\geq 1} \BA{W_{\lin,n}^j}(0),
\end{equation*}
and one of the following holds
\begin{itemize}
 \item \emph{Case 1.} For all $j\geq 1$, $W^j$ scatters in both time directions and there exists $\eta>0$, $j_0\geq 1$ such that the following holds for all $t\geq 0$ or for all $t\leq 0$:
\begin{equation}
 \label{Eqn:A2_blowup}
 \int_{\rho_n+|t|}^{\infty} \left|r\partial_rW_n^{j_0}(t,r)\right|^m+\left|r\partial_tW_n^{j_0}(t,r)\right|^m\,dr\geq \eta.
\end{equation}
\item  \emph{Case 2.} For all $j\geq 2$, $W^j$ scatters in both time directions and there exists $\eta>0$, ${j_0}\geq 2$ such that (\ref{Eqn:A2_blowup}) holds for all $t\geq 0$ or for all $t\leq 0$.
Furthermore,
\begin{gather*}
\lim_{n\to\infty}-\frac{t_{j_0,n}}{\lambda_{j_0,n}}\in \{\pm\infty\}\\
\forall n,\quad t_{1,n}=0\quad \text{and}\quad
(W_{0}^1,W_1^1)=\left(\TTT_{r_1}Z,0\right)
\end{gather*}
for some $r_1>0$ such that
\begin{equation*}
 \forall n,\quad \rho_n\geq r_1\lambda_{1,n}.
\end{equation*}
\item\emph{Case 3.} For all $j\geq 2$, $W^j$ scatters in both time directions,
$$\forall n,\quad t_{1,n}=0; \quad \sigma_1(W_0^1,W^1_1)<\infty; \quad \text{ and }\limsup_{n\to \infty}\frac{\rho_n}{\lambda_{1,n}}<\sigma_1(W^1_0,W^1_1).$$
 \end{itemize}
\end{lem}
The proof of Lemma \ref{Lem:Cases_Blowup} is almost the same as the proof of Lemma \ref{Lem:Cases}. The proof of Theorem \ref{Thm:Main} in the case $T_+(w)<\infty$, assuming Propositions \ref{Prop:ConstructionV} and \ref{Prop:ProfileDecompSpecFinite}, and Lemma \ref{Lem:Cases_Blowup} is also very close to the corresponding proof in the case $T_+(w)=\infty$ (see the paragraphs after Lemma \ref{Lem:Cases}, p. \pageref{after_cases}). We omit both proofs.

\subsection{Local strong limit outside the origin}
We prove here Proposition \ref{Prop:ConstructionV}.
We will need the following lemma, based on standard energy estimates and the radial Sobolev embedding.
\begin{lem}
\label{Lem:Linfty}
For all $(a,b)\in \RR^2$ such that $0<a<b$, there exists a constant $C_{a,b}$ with the following property. Let $I$ be an interval, $\tau_{0} \in I$, $F\in L^1_t(I,L^2_x)$ and  $(u_0,u_1)\in \HHH^{s_c}$. Let $u$ be the solution of

 \begin{equation}
 \left\{
 \begin{array}{l}
\partial_{tt}u-\triangle u=F,\quad (t,x)\in I\times \RR^3\\
   \vec{u}(\tau_0,x)=(u_0(x),u_1(x)).
 \end{array}
 \right.
\end{equation}
For $t\in I$, let $J_t=[a + |t-\tau_0|, b-|t-\tau_0|]$. Then (assuming that the right-hand side is finite),
\begin{multline}
 \label{Eqn:Linfty_bnd}
\forall t \in I,\quad \|u(t)\|_{L^{\infty}_{x}(J_t)}\\
\leq C_{a,b}\left[\left(\int_{a/2\leq |x|\leq 2b} |\nabla u_0|^2+|u_0|^2+|u_1|^2\,dx\right)^{1/2}+ \left|\int_{\tau_0}^t
\|F(s)\|_{L_{x}^{\infty}(J_s)}\,ds\right|\right]
\end{multline}
\end{lem}

\begin{proof}
Let
$$M=\left(\int_{a/2\leq |x|\leq 2b} |\nabla u_0|^2+|u_0|^2+|u_1|^2\,dx\right)^{1/2}+\int_I \| F(s) \|_{L_{x}^{\infty}(J_s)}\,ds<\infty.$$
It is sufficient to prove
\begin{equation}
 \label{Eqn:Linfty_bnd_bis}
 \forall t\in I,\quad \|u(t)\|_{L_{x}^{\infty}(J_t)}\leq C_{a,b}M.
\end{equation}

In all the proof, we will denote by $C$ a large constant, that may depend on $a$ and $b$ and change from line to line.

By a classical extension lemma, there exists $(\tu_0,\tu_1)\in \HHH^1$ such that
$$ (\tu_0,\tu_1)(r)=(u_0,u_1)(r),\; a\leq r\leq b,\quad \text{and }\left\|(\tu_0,\tu_1)\right\|_{\HHH^1}\leq C M.$$
Let $\tu$ be the solution of
\begin{equation}
 \label{Eqn:eq_tu}
 \left\{
 \begin{array}{l}
\partial_{tt}\tu-\triangle \tu=\mathbf{1}_{J_t}(r)F,\quad (t,x)\in I\times \RR^3\\
   \tu(\tau_0,x)=\tu_0(x),\quad \partial_t \tu(\tau_0,x)=\tu_1(x),
 \end{array}
 \right.
\end{equation}
where $\mathbf{1}_{J_t}$ is the characteristic function of $J_t$.
Note that by the assumptions of the Lemma, the right-hand side of the first equation of (\ref{Eqn:eq_tu}) is in $L_{t}^1 L_{x}^2 (I)$, so that
$$ \BA{\tu} \in C^{0}(I,\HHH^1).$$
By finite speed of propagation, $\tu(t,r)=u(t,r)$ if $t\in I$, $r\in J_t$. It then suffices to prove (\ref{Eqn:Linfty_bnd_bis}) for $\tu$ instead of $u$.
By energy estimates:
$$ \forall t\in I,\quad \|\tu(t)\|_{\dot{H}^1}\leq \sqrt{\|\tu_0\|_{\dot{H}^1}^2+\|\tu_1\|_{L^2}^2}+\int_I \left\|\mathbf{1}_{J_s}F\right\|_{L^2}\,ds,$$
which, in view of the radial Sobolev inequality:
$$ |f(r)|\lesssim \frac{1}{r^{1/2}}\|f\|_{\dot{H}^1},$$
and noting that $J_t$ has finite length smaller than $b-a$, concludes the proof of the lemma.
\end{proof}
We next prove Proposition \ref{Prop:ConstructionV}.
\begin{proof}
Let $\{t_n\}_n\to 1$, $t_n<1$, such that $\vec{w}(t_n)$ is bounded in $\HHH^{s_c}$. Extracting
a subsequence, we can assume
 \begin{equation}
  \label{Eqn:weak_CV}
  \vec{w}(t_n)\xrightharpoonup[n\to\infty]{} (v_0,v_1)\in \HHH^{s_c},
 \end{equation}
 weakly in $\HHH^{s_c}$. Let $v_+$ be the solution of (\ref{Eqn:WaveSup}) with initial data $(v_0,v_1)$ at $t=1$. Let $T$ such that
 $ T_-(v_+) < T <1$, and $(a',b')\in \RR^2$ such that
 \begin{equation}
 \label{Eqn:a'b'}
 1-T<a'<b'.
 \end{equation}
 We will prove
 \begin{equation}
  \label{Eqn:v+w}
  w(T,r)=v_+(T,r),\quad a'<r<b',
 \end{equation}
which will yield the conclusion of the proposition since $a'$ and $b'$ can be chosen to be arbitrary numbers satisfying (\ref{Eqn:a'b'}). We have
\begin{equation*}
(\partial_{tt}-\triangle)(v_+-w)=\iota \left(|v_+|^{p-1}v_+ -|w|^{p-1}w\right).
\end{equation*}
We fix a large $n$ (so that $t_n>T$), and use Lemma \ref{Lem:Linfty} with $\tau_0=t_n$, $u=v_+-w$, $a=a'-(t_n-T)$, $b=b'+t_n-T$ (so that $0<a<b$) and $I=[T,t_n]$. Note that
for these choices of $\tau_0$, $a$ and $b$,
$$J_t=[a'-(t-T),b'+t-T],\quad t\in [T,t_n].$$
Let
$$B(t)=\|v_+(t)-w(t)\|_{L^{\infty}(J_t)},\quad T\leq t\leq t_n.$$
By Lemma \ref{Lem:Linfty} there exists a constant $C_{0}:=C_{0}(a',b')$ such that
\begin{multline}
\label{Eqn:bound_B}
T\leq t\leq t_n\Longrightarrow \\
B(t)\leq C_0 \left( \varepsilon_n + \int_{t}^{t_n} B(s)\left(\|v_+(s)\|^{p-1}_{L^{\infty}(J_s)}+B^{p-1}(s)\right) \, ds \right),
\end{multline}
where $\varepsilon_n$ is defined as the following integral:
\begin{equation*}
\int_{\frac{a'-(t_n-T)}{2}\leq |x|\leq 2(b'+t_n-T)} \left(\left|\nabla v_+-\nabla w\right|^2+\left|\partial_t v_+-\partial_t
w\right|^2+\left|v_+-w\right|^2\right)(t_n,x)\,dx.
 \end{equation*}
 Since $t_n-T\leq 1-T$, we can bound this integral by the integral on the set of $x$ such that $\frac{a'-(1-T)}{2}\leq |x|\leq 2(b'+(1-T))$. By
 (\ref{Eqn:weak_CV}), and the Rellich-Kondrachov theorem,
 $$\lim_{n\to\infty}\varepsilon_n=0.$$

 Let $\epsilon>0$. Let $n$ such that $C_0\varepsilon_n<\epsilon$. We will prove by bootstrap that if $\epsilon$ is small enough,
 \begin{equation}
 \label{Eqn:bootstrap}
 \forall t\in [T,t_n], \quad B(t)< C_2\epsilon,
 \end{equation}
 where $C_2=e^{2C_0C_1(1-T)}+1$ with $C_0$ defined in (\ref{Eqn:bound_B}) and
 $$ C_1=\max_{T<s<1} \|v(s)\|^{p-1}_{L^{\infty}(J_s)},$$
 which is finite, by the radial Sobolev embedding and since $v \in C^0\left([T,1],\dot{H}^{s_c} \right)$. \\
 Note that $B(t_n)\leq C_0\eps_n<\epsilon$ by (\ref{Eqn:bound_B}). To prove (\ref{Eqn:bootstrap}), we argue by contradiction and assume that there exists $\tau\in [T,t_n]$ such that
 \begin{equation}
 \label{Eqn:bootstrap2}
 B(\tau)=C_2 \epsilon \text{ and }\forall t\in (\tau,t_n], \quad B(t)< C_2 \epsilon.
 \end{equation}
 By (\ref{Eqn:bound_B}),
 $$ \forall t\in [\tau,t_n], \quad B(t) \leq \epsilon + C_0 \int_{t}^{t_n} B(s)\left(C_1+C_2^{p-1}\epsilon^{p-1}\right)\,ds.$$
 Thus by Gronwall's Lemma,
 $$ \forall t\in [\tau,t_n], \quad B(t)\leq \epsilon e^{C_0 \left(C_1 + C_2^{p-1}\epsilon^{p-1}\right)(t_n-t)}\leq \epsilon
 e^{C_0 \left(C_1 +C_2^{p-1}\epsilon^{p-1}\right)(1-T)}.$$
 If $\epsilon$ is so small that $C_2^{p-1}\epsilon^{p-1}\leq C_1$, we obtain:
 $$B(\tau) \leq \epsilon e^{2C_0 C_1 (1-T)}<\epsilon C_2,$$
 contradicting (\ref{Eqn:bootstrap2}) and concluding the proof of (\ref{Eqn:bootstrap}).

 By (\ref{Eqn:bootstrap}),
 $$ \forall \epsilon>0,\quad B(T)<C_2 \epsilon.$$
 Thus $B(T)=0$, which concludes the proof of the proposition.
\end{proof}

\subsection{Nonexistence of profile with exterior generalized energy}
In this subsection we prove Proposition \ref{Prop:ProfileDecompSpecFinite}.

From (\ref{Eqn:ProfDecompPropProf}) we know that there exists $(w_{0,n},w_{1,n})$ such that (\ref{Eqn:PropProf0}) and (\ref{Eqn:PropProf2}) hold, with
$\BA{w_{L}}(t_n)$ replaced with $\BA{v_+}(1)$. Let $w_n$ be the solution of (\ref{Eqn:WaveSup}) with initial data $(w_{0,n},w_{1,n})$ at $t=0$. Let $\theta_0>0$ be small, so that
$$T_-(v_+)<1-\theta_0<1+\theta_0<T_+(v_+).$$
By Proposition \ref{Prop:Perturb}, we see that $w_n$ is defined on $[1-\theta_0,1+\theta_0]$ for large $n$ and
\begin{equation}
 \label{Eqn:decompo_wn}
\forall t\in [-\theta_0,+\theta_0],\quad \BA{w_n}(t,x)= \BA{v_+}(1+t,x) + \sum_{j=1}^J  \BA{W_{n}^j}(t,x) + \BA{\eps_n^J}(t,x) + \BA{r_n^J}(t,x),
\end{equation}
where
$$\lim_{J\to\infty}\limsup_{n\to\infty} \sup_{-\theta_0<t<\theta_0} \|\BA{r_n^J}(t)\|_{\HHH^{s_c}}=0.$$
By finite speed of propagation, for all $t\in [-\theta_0,\theta_0]$ such that $0<t_n+t<1$, we have:
\begin{equation}
\label{Eqn:ext_equal1}
|x|>\rho_n+|t|\Longrightarrow \BA{w_n}(t,x)=\vec{w}(t+t_n,x).
\end{equation}
By Proposition \ref{Prop:ConstructionV}, we deduce, in the same range for $t$,
\begin{equation*}
|x|>\max(1-(t+t_n),\rho_n+|t|)\Longrightarrow \BA{w_n}(t,x)=\vec{v}_+(t+t_n,x).
\end{equation*}
Using the continuity of the $\HHH^{s_c}$-valued maps $t \mapsto \BA{w_n}(t)$ and $t\mapsto \BA{v_+}(t)$, we get that the preceding equality holds also for $t=1-t_n$, i.e.
\begin{equation}
 \label{Eqn:ext_equal2}
(1-t_n\leq \theta_0\text{ and } |x|>\rho_n+1-t_n)\Longrightarrow \BA{w_n}(1-t_n,x)= \BA{v_+}(1,x).
\end{equation}

First assume that (\ref{Eqn:PropProf1finite}) holds for all $t\geq 0$. By (\ref{Eqn:decompo_wn}) and Proposition \ref{Prop:Orth}, the following holds for large $n$:
$$\int_{|x|\geq \rho_n+1-t_n}\left|r^{1-\frac{2}{m}}\partial_{r,t}\big(w_n(1-t_n,x)-v_+(2-t_n,x)\big)\right|^m\,dx\geq \frac{\epsilon}{2}.$$
Combining with (\ref{Eqn:ext_equal2}), we obtain
$$\int_{|x|\geq \rho_n+1-t_n} \left|r^{1-\frac{2}{m}}\partial_{r,t}\big(v_+(1,x)-v_+(2-t_n,x)\big)\right|^m\,dx\geq \frac{\epsilon}{2}.$$
Letting $n\to\infty$, we see by Lemma \ref{Lem:HardyIneq} that the left-hand side of the preceding inequality goes to zero, a contradiction.

Next assume that (\ref{Eqn:PropProf1finite}) holds for all $t\leq 0$. By (\ref{Eqn:decompo_wn}) and Proposition \ref{Prop:Orth}, the following holds for large $n$:
$$ \int_{|x|\geq \rho_n+t_n+\theta_0-1}\left|r^{1-\frac{2}{m}}\partial_{r,t}\big(w_n(1-\theta_0-t_n,x)-v_+(2-\theta_0-t_n,x)\big)\right|^m\,dx\geq \frac{\epsilon}{2}.$$
Hence, by (\ref{Eqn:ext_equal1}),
$$ \int_{|x|\geq \rho_n+t_n+\theta_0-1}\left|r^{1-\frac{2}{m}}\partial_{r,t}\big(w(1-\theta_0,x)-v_+(2-\theta_0-t_n,x)\big)\right|^m\,dx\geq \frac{\epsilon}{2}.$$
Since $v_+(1-\theta_0,x)=w(1-\theta_0,x)$ for $|x|>\theta_0$, we deduce
\begin{multline*}
\int_{|x|\geq \theta_0}\left|r^{1-\frac{2}{m}}\partial_{r,t}\big(v_+(1-\theta_0,x)-v_+(2-\theta_0-t_n,x)\big)\right|^m\,dx\\
+\int_{\rho_n+t_n+\theta_0-1\leq |x|\leq \theta_0}\left|r^{1-\frac{2}{m}}\partial_{r,t}\big(w(1-\theta_0,x)-v_+(2-\theta_0-t_n,x)\big)\right|^m\,dx\geq \frac{\epsilon}{2},
\end{multline*}
with the convention that the second integral on the left-hand side of the preceding line is zero if $\rho_n+t_n-1\geq 0$. Letting $n\to\infty$, we get again a contradiction.


\appendix

\section{Proof of the localization property}
\label{App:ProofStrongHuyg}

In this appendix we prove Proposition \ref{Prop:Huyg}.

Assume that $l \in \mathbb{R}$. Then, by continuity of the linear flow in $\HHH^{s_c} $
and Lemma \ref{Lem:HardyIneq}

\begin{align*}
\lim_{n\to\infty}\left\|w_{\lin,n}(t_{n},x) -\frac{1}{\lambda_{n}^{\frac{2}{p-1}}} w_{\lin} \left(l, \frac{\cdot}{ \lambda_{n}} \right)\right\|_{\dot{H}^{s_c}}&=0\\
\lim_{n\to\infty}\left\|\partial_{t} w_{\lin,n} (t_{n},x ) -\frac{1}{\lambda_{n}^{\frac{2}{p-1} +1 }}  \partial_{t} w_{\lin} \left(l, \frac{\cdot}{
\lambda_{n}} \right) \right\|_{\dot{H}^{s_c-1}}&=0.
\end{align*}
This easily leads to (\ref{Eqn:Huyg2}), taking also into account Lemma \ref{Lem:HardyIneq}.

Next assume that $l \in \pm \infty$. Let $\epsilon > 0$. Let $\bar{\chi}$ be a smooth function such that $\bar{\chi}(x) =1$ if $|x| \leq \frac{1}{2}$ and
$\bar{\chi}(x) = 0$ if $|x| \geq 1$. Let $R > 0$ and
$(w_{0}^{R},w_{1}^{R}):= \left( \bar{\chi} \left( \frac{x}{R} \right) w_{0}, \bar{\chi} \left( \frac{x}{R} \right) w_{1} \right) $. Since
\begin{equation*}
\lim_{R \rightarrow \infty} \big\| (w_{0}^{R} - w_{0},w_{1}^{R} - w_{1}) \big\|_{\HHH^{s_c}}=0,
\end{equation*}
one can choose $\tR$ such that for $R \geq \tR$
\begin{equation*}
\big\| (w_{0}^{R} - w_{0},w_{1}^{R} - w_{1}) \big\|_{\HHH^{s_c}}\ll \epsilon
\end{equation*}
Let $w_{\lin,n}^{R}$ be the solution of the linear wave equation with data
\begin{align*}
\left( w_{0,n}^{R}(x), w_{1,n}^{R} (x) \right) :=
\left( \frac{1}{\lambda_{n}^{\frac{2}{p-1}}} w_{0}^{R} \left(  \frac{x}{\lambda_{n}} \right), \frac{1}{\lambda_{n}^{\frac{2}{p-1} + 1}} w_{1}^{R} \left(
\frac{x}{\lambda_{n}} \right)
\right) \cdot
\end{align*}
Lemma \ref{Lem:HardyIneq} yields
\begin{align*}
\left\| r^{1-\frac{2}{m}} ( \partial_{r,t}  w_{\lin,n}(t_{n})  - \partial_{r,t} w_{\lin,n}^{R}(t_{n}) ) \right\|_{L^{m}}
&\lesssim \| \BA{w_{\lin,n}}(t_{n}) - \BA{w_{\lin,n}^{R}}(t_{n}) \|_{\HHH^{s_c}} \\
& \lesssim \| (w_{0,n}^{R} - w_{0,n}, w_{1,n}^{R} - w_{1,n})  \|_{\HHH^{s_c}}
 \\&\ll \epsilon
\end{align*}
Hence
\begin{multline*}
\left( \int_{ | |t_{n}| - |x|| \geq R \lambda_{n} }  \left|r^{1-\frac{2}{m}} \partial_{r} w_{\lin,n}(t_{n})\right|^{m} + \left|r^{1-\frac{2}{m}} \partial_{t} w_{\lin,n}(t_{n})\right|^{m} \, dx
 \right)^{\frac{1}{m}}  \\
\leq   \left( \int_{ | |t_{n}| - |x|| \geq R \lambda_{n}}  \left|r^{1-\frac{2}{m}} \partial_{r} w_{\lin,n}^R(t_{n})\right|^{m}  + \left|r^{1-\frac{2}{m}} \partial_{t} w_{\lin,n}^R(t_{n})\right|^{m} \,
dx \right)^{\frac{1}{m}} + O(\epsilon)  \\
\lesssim \epsilon,
\end{multline*}
using the strong Huygens principle at the last line, i.e  $\BA{w_{\lin,n}^{R}}(t_{n})$ is supported
in the ring $ |t_{n}| - R \lambda_{n} \leq |x| \leq |t_{n}| + R \lambda_{n}$.

\section{Useful Estimates}
\label{App:useful}
\begin{res}
Let $ 0 \leq  s  <  \frac{1}{2} $. Let $R > 0$. Let $f \in \dot{H}^{s}$. Then
\begin{equation}
\| \mathbf{1}_{B_R} f \|_{\dot{H}^{s}} \lesssim \| f \|_{\dot{H}^{s}},
\nonumber
\end{equation}
\label{Res:BoundCharac}
\end{res}
\begin{proof}
Since the $\dot{H}^{s}$ norm is invariant under the following scaling transform
\begin{equation*}
f \rightarrow \frac{1}{\lambda^{\frac 32-s}} f \left( \frac{\cdot}{\lambda} \right),
\end{equation*}
we may assume without loss of generality that $R=1$.

Recall the following estimate (see Theorem 2, p 151, of \cite{RunstSick}, and the references given in this book,
i.e \cite{Gu1,Gu2,MaSh})
\begin{align}
\| \mathbf{1}_{B_1} f  \|_{H^{s}}  \lesssim  \| f \|_{H^{s}} \cdot
\label{Eqn:BoundCharacInhom}
\end{align}
Let $\bar{\chi}$ be a smooth compactly supported function on $\mathbb{R}^{3}$ such that $\bar{\chi}(x)=1$ on $B_1$. Then,
applying (\ref{Eqn:BoundCharacInhom}) to $\bar{\chi} f$, we have
\begin{equation}
\begin{array}{ll}
\| \mathbf{1}_{B_1} f \|_{\dot{H}^{s}} & = \| \mathbf{1}_{B_1} \bar{\chi} f \|_{\dot{H}^{s}} \\
& \lesssim \| \bar{\chi}  f \|_{H^{s}} \\
& \lesssim  \| f \|_{\dot{H}^{s}},
\end{array}
\nonumber
\end{equation}
the last inequality coming from $ \| \bar{\chi} f \|_{\dot{H}^{s}} \lesssim \| f \|_{\dot{H}^{s}} $ and the following estimate
\begin{equation}
\| \bar{\chi}  f \|_{L^{2}} \lesssim \| \bar{\chi} \|_{L^{\frac 3s} } \| f \|_{L^{\frac{6}{3-2s}}}
\lesssim  \| f \|_{\dot{H}^{s}}\cdot
\nonumber
\end{equation}

\end{proof}

\begin{res}
Let $ 1 \leq s < \frac{3}{2}$. Let $R > 0$. Let $f \in \dot{H}^{s}$. Then
\begin{equation*}
\| \TTT_R (f) \|_{\dot{H}^{s}} \lesssim  \| f \|_{\dot{H}^{s}},
\end{equation*}
where again the implicit constant is independent of $R>0$.
\label{Res:BoundOp}
\end{res}
\begin{proof}
One may write $\TTT_{R} (f) = \mathbf{1}_{B_R} f(R) + f  \mathbf{1}_{\mathbb{R}^{3} \setminus B_R} $. By scaling (see proof
of Result \ref{Res:BoundCharac}), we may assume that $R=1$.

Recall that the Fourier transform of $f$ is given by
\begin{align*}
\hat{f}(\xi) = \frac{1}{|\xi|} \int_{0}^{\infty} \sin{(r |\xi|)} r f(r) \, dr
\end{align*}
Therefore, choosing
\begin{equation}
\tilde{f}(r):=
\left\{
\begin{array}{l}
r f(r), \,  r \geq 0 \\
r f(-r), \, r \leq 0,
\end{array}
\right.
\nonumber
\end{equation}
we have $ \| f \|_{\dot{H}^{s}} = \| \tilde{f} \|_{\dot{H}^{s}(\mathbb{R})} $. Hence we are reduced to show
that
\begin{equation}
\bigg\| \frac{d \widetilde{\TTT_1(f)}}{dr}  \bigg\|_{\dot{H}^{s-1}(\mathbb{R})} \lesssim \left\| \frac{d \tilde{f}}{dr}   \right\|_{\dot{H}^{s-1}(\mathbb{R})}\cdot
\nonumber
\end{equation}
We have
$$
\frac{d \widetilde{\TTT_1(f)}}{dr}
=  f(1) \mathbf{1}_{(-1,1)} + \mathbf{1}_{\mathbb{R}/ (-1,1)} \frac{d \tilde{f}}{dr}\cdot
$$
Since (see again \cite{RunstSick})
\begin{equation*}
 \bigg\|\mathbf{1}_{(-1,+1)} \frac{d\tilde{f}}{dr}\bigg\|_{\dot{H}^{s-1}(\mathbb{R})}\lesssim \left\|\tilde{f}\right\|_{\dot{H}^{s-1}(\RR)},
\end{equation*}
and taking into account that $\mathbf{1}_{(-1,1)} \in \dot{H}^{s-1}(\mathbb{R})$, we are reduced to show that
$ |f(1)|  \lesssim \left\| \frac{d \tilde{f}}{dr} \right\|_{\dot{H}^{s-1}(\mathbb{R})} $. But this follows from
\begin{align*}
|f(1)|  \lesssim \int_{0}^{1} \left| \frac{d(r f)}{dr} \right| \, ds \lesssim \left\| \frac{d(rf)}{dr}  \right\|_{L^{\frac{2}{3-2s}} ((0,1))}
\lesssim \left\| \frac{d \tilde{f}}{dr}   \right\|_{\dot{H}^{s-1}(\mathbb{R})} \cdot
\end{align*}
\end{proof}
\begin{res}
 Let $1 \leq s<3/2$ and $(f,g)\in \HHH^s$. Let $\sigma>0$. Then
 \begin{gather*}
 \lim_{R\to \sigma} \left\|\left((\TTT_R-\TTT_{\sigma})f,(\mathbf{1}_{B_R}-\mathbf{1}_{B_{\sigma}})g\right)\right\|_{\HHH^{s}}=0\\
 \lim_{R\to\infty}\left\|\left(\TTT_Rf,\mathbf{1}_{\RR^3\setminus B_R}g\right)\right\|_{\HHH^s}=0
 \end{gather*}
 \label{Res:Cont_R}
\end{res}
\begin{proof}
We prove the first estimate, the proof of the second one is close and left to the reader.

 By Results \ref{Res:BoundCharac} and \ref{Res:BoundOp} and a density argument, it is sufficient to prove this estimate for $(f,g)\in \left(C^{\infty}_0\right)^2$. For such $(f,g)$ we have obviously:
 $$ \lim_{R\to \sigma} \left\|\left((\TTT_R-\TTT_{\sigma})f,(\mathbf{1}_{B_R}-\mathbf{1}_{B_{\sigma}})g\right)\right\|_{\HHH^{1}}=0.$$
 And the results follows since for $s<s'<3/2$,
 $$\left\|\left((\TTT_R-\TTT_{\sigma})f,(\mathbf{1}_{B_R}-\mathbf{1}_{B_{\sigma}})g\right)\right\|_{\HHH^{s'}}$$
 is bounded independently of $R$ by Results \ref{Res:BoundCharac} and \ref{Res:BoundOp}.
\end{proof}

\end{document}